\documentclass[reqno,12pt]{amsart}
\usepackage[T1]{fontenc}
\usepackage{lmodern}
\usepackage[utf8]{inputenc}
\usepackage[english]{babel}
\usepackage{url}
\usepackage[colorlinks=true]{hyperref}
\usepackage{comment}
\usepackage{xcolor,color}
\usepackage{geometry}
\usepackage[all]{xy}
\usepackage{enumitem}
\usepackage{booktabs}
\usepackage{slashed}

\usepackage{soul}

\usepackage{textcomp}

\usepackage{amsmath}
\usepackage{amssymb,graphicx}
\usepackage{amsthm}
\usepackage{latexsym}
\usepackage{amsfonts}
\usepackage{mathrsfs}

\usepackage{etex}
\usepackage{tikz}
\usetikzlibrary{arrows,decorations.pathreplacing,decorations.markings,shapes.geometric}

\numberwithin{equation}{section}

\theoremstyle{plain}
\newtheorem{lemma}{Lemma}[section]
\newtheorem{proposition}[lemma]{Proposition}
\newtheorem{proposition/definition}[lemma]{Proposition/Definition}
\newtheorem{theorem}[lemma]{Theorem}
\newtheorem{corollary}[lemma]{Corollary}

\theoremstyle{definition}
\newtheorem{definition}[lemma]{Definition}
\newtheorem{remark}[lemma]{Remark}
\newtheorem{example}[lemma]{Example}

\DeclareMathOperator{\id}{id}
\DeclareMathOperator{\im}{im}

\DeclareMathOperator{\Der}{Der}
\DeclareMathOperator{\VDer}{V\frakX}

\DeclareMathOperator{\Diff}{{\textit D}}
\DeclareMathOperator{\End}{End}
\DeclareMathOperator{\Aut}{Aut}

\DeclareMathOperator{\MC}{MC}
\DeclareMathOperator{\BFV}{BFV}

\DeclareMathOperator{\BRST}{BRST}
\DeclareMathOperator{\Ham}{Ham}
\DeclareMathOperator{\Jac}{Jac}

\DeclareMathOperator{\weight}{\mathsf{weight}_{\nabla}}
\newcommand{\J}{J}
\newcommand{\vder}{\textnormal{V}\frakX}

\newcommand{\Mod}{\ \text{mod}\ }

\newcommand{\bfA}{\mathbf{A}}
\newcommand{\bfB}{\mathbf{B}}
\newcommand{\bfC}{\mathbf{C}}

\newcommand{\calA}{\mathcal{A}}

\newcommand{\calC}{\mathcal{C}}

\newcommand{\calE}{\mathcal{E}}
\newcommand{\calF}{\mathcal{F}}
\newcommand{\calG}{\mathcal{G}}

\newcommand{\calK}{\mathcal{K}}
\newcommand{\calL}{\mathcal{L}}

\newcommand{\calT}{\mathcal{T}}

\newcommand{\scrL}{\mathscr{L}}
\newcommand{\scrM}{\mathscr{M}}

\newcommand{\scrP}{\mathscr{P}}
\newcommand{\scrQ}{\mathscr{Q}}

\newcommand{\bbD}{\mathbb{D}}

\newcommand{\bbN}{\mathbb{N}}

\newcommand{\bbR}{\mathbb{R}}

\newcommand{\bbT}{\mathbb{T}}

\newcommand{\bbZ}{\mathbb{Z}}

\newcommand{\frakX}{\mathfrak{X}}

\newcommand{\frakg}{\mathfrak{g}}

\newcommand{\frakm}{\mathfrak{m}}

\renewcommand{\phi}{\varphi}

\renewcommand{\theta}{\vartheta}

\renewcommand{\tilde}[1]{\widetilde{#1}}
\renewcommand{\hat}[1]{\widehat{#1}}
\renewcommand{\bar}[1]{\overline{#1}}
\newcommand{\ldsb}{[\![}
\newcommand{\rdsb}{]\!]}

\title{Jacobi bundles and the BFV-complex}

\author{H\^ong V\^an L\^e}
\address{Institute of Mathematics of ASCR, Zitna 25, 11567 Praha 1, Czech Republic.}
\email{\href{mailto:hvle@math.cas.cz}{hvle@math.cas.cz}}

\author{Alfonso G.~Tortorella}
\address{Dipartimento di Matematica e Informatica ``Ulisse~Dini'', Universit\`a degli Studi di Firenze, Viale Morgagni 67/a,
	50134 Firenze, Italy.}
\email{\href{mailto:alfonso.tortorella@math.unifi.it}{alfonso.tortorella@math.unifi.it}}

\author{Luca Vitagliano}
\address{DipMat, Universit\`a degli Studi di Salerno \& Istituto Nazionale di Fisica Nucleare, GC Salerno, via Giovanni Paolo II n${}^\circ$ 123, 84084 Fisciano (SA) Italy.}
\email{\href{mailto:lvitagliano@unisa.it}{lvitagliano@unisa.it}}



\begin{document}
	
\begin{abstract}
	We extend the construction of the BFV-complex of a coisotropic submanifold from the Poisson setting to the Jacobi setting.
	In particular, our construction applies in the contact and l.c.s.~settings.
	The BFV-complex of a coisotropic submanifold $S$ controls the coisotropic deformation problem of $S$ under both Hamiltonian and Jacobi equivalence.
\end{abstract}

\maketitle

\tableofcontents

\section{Introduction}


Jacobi structures were introduced by Kirillov~\cite{kirillov1976local}, and independently by Lichnerowicz~\cite{Lich1978}.
They generalize and unify contact structures, locally conformal symplectic (l.c.s.) structures, and Poisson structures.
Note that Kirillov's local Lie algebras with one-dimensional fibers, also called Jacobi bundles~\cite{marle1991jacobi}, are slightly more general than Lichnerowicz's Jacobi manifolds.


Coisotropic submanifolds play a significant r\^ole in Jacobi geometry as already in contact and Poisson geometry.
For instance, coisotropic submanifolds naturally appear as the zero level sets of equivariant moment maps of an Hamiltonian Lie group action on a Jacobi bundle.
Moreover, coisotropic submanifolds come equipped with a Lie algebroid whose characteristic foliation allows to perform the (singular) Jacobi reduction.

Recall that a Jacobi bundle is a line bundle $L\to M$ equipped with a Jacobi bracket $J=\{-,-\}$ on its sections (see Definition \ref{definition:abstractj}).
If $L \to M$ is a Jacobi bundle, there is an $L_\infty$-algebra attached to any coisotropic submanifold $S$ in $M$ \cite{LOTV}.
This construction unifies and generalizes analogous constructions in~\cite{oh2005deformations} (symplectic case),~\cite{cattaneo2007relative} (Poisson case) and~\cite{le2012deformations} (l.c.s.~case), and also applies to the case of a coisotropic submanifold in a contact manifold.
The $L_\infty$-algebra controls the formal coisotropic deformation problem of $S$, but fails to convey any information about the non-formal coisotropic deformations of $S$, unless $J$ satisfies a certain entireness condition~\cite{LOTV}.
The main aim of this paper is to equip $S$ with an algebraic invariant well-suited to control (both the formal and) the non-formal coisotropic deformation problem.

On another side, the BRST formalism was originally introduced by Becchi, Rouet, Stora and Tyutin~\cite{Becchi1976,Tyutin1975} as a method to deal, both on the classical and the quantum level, with physical systems possessing gauge symmetries or Dirac (first class) constraints.
The Hamiltonian counterpart of this formalism was developed by Batalin, Fradkin and Vilkovisky~\cite{Batalin1983,Batalin1977}.
It was soon realized that, for systems with finitely many degrees of freedom, the BFV-method is intimately related to symplectic and Poisson reduction~\cite{kostant1987symplectic}.
The construction of the underlying BFV-complex was recast in the context of homological perturbation theory by Henneaux and Stasheff~\cite{Henneaux1992,stasheff1988constrained,Stasheff1997}.
More recently, simplified versions (without ``ghosts of ghosts'') of the BFV-complex of a coisotropic submanifold have been constructed in~\cite{Bordemann2000} (symplectic case) and~\cite{herbig2007,schatz2009bfv,schatz2009coisotropic} (Poisson case).

Our main result is the extension of the construction of the BFV-complex of a coisotropic submanifold from the Poisson setting to the far more general Jacobi setting.
The resulting BFV-complex, seen as a graded Jacobi bundle equipped with a homological Hamiltonian derivation, provides a homological resolution of the (non-graded) Gerstenhaber-Jacobi algebra obtained from $S$ by singular Jacobi reduction.
Our results are inspired by and encompass as special cases those of Herbig and Sch\"atz.
However, we stress that we do not follow Sch\"atz in all our proofs.
In fact, we fully rely on the Homological Perturbation Lemma and the ``step-by-step obstruction'' methods of homological perturbation theory rather than on homotopy transfer.
As a new case, our BFV-complex applies also to coisotropic submanifolds in contact and l.c.s.~manifolds.

Similarly as in the Poisson case, the $L_\infty$-algebra of $S$ can be reconstructed starting from the BFV-complex by homotopy transfer along suitable contraction data.
As a consequence, the BFV-complex and the $L_\infty$-algebra are $L_\infty$-quasi-isomorphic, and control equally well the formal coisotropic deformation problem of $S$.

The BFV-complex also controls the non-formal coisotropic deformation problem of $S$ and encodes the moduli spaces of coisotropic submanifolds (at least around $S$) under both Hamiltonian and Jacobi equivalence.
We follow Sch\"atz and single out a special class of ``geometric'' Maurer-Cartan (MC) elements of the BFV-complex.
In this way we are able to establish a 1--1 correspondence between coisotropic deformations of $S$ and geometric MC elements, modulo a certain equivalence.
Additionally, such 1--1 correspondence intertwines Hamiltonian/Jacobi equivalence of coisotropic deformations and Hamiltonian/Jacobi equivalence of geometric MC elements.

In the symplectic case, an interesting example of an obstructed coisotropic submanifold has been first considered by Zambon in~\cite{zambon2008example}.
Later Zambon's example has been reconsidered by Sch\"atz~\cite{schatz2009bfv} in terms of the associated BFV-complex.
Our construction is able to deal with a larger class of examples, indeed it applies, in particular, to the contact and the l.c.s.~case where Sch\"atz construction does not apply.
We illustrate this via an Example of an obstructed coisotropic submanifold in a contact manifold first presented in~\cite{tortorella2016rigidity} (see Section~\ref{sec:example}).
Actually this example can be reconsidered in terms of the associated BFV-complex.

The paper is organized as follows.
Section~\ref{sec:abstract_jac_mfd} collects known facts about algebraic and geometric structures attached to Jacobi manifolds and their coisotropic submanifolds which are used in what follows.
In Section~\ref{sec:lifting_Jacobi_structures}, by analogy with~\cite{Rothstein1991}, we establish existence and uniqueness results for suitable liftings of a Jacobi structure to a graded Jacobi structure.
In Section~\ref{sec:BRST-charges}, we introduce the BFV-bracket, and establish existence and uniqueness results for the corresponding BRST charges.
In Section~\ref{sec:BFV-complex}, using results from Section~\ref{sec:BRST-charges}, we equip any coisotropic submanifold with a BFV-complex.
The BFV-complex is canonical up to isomorphisms.
We also construct an $L_\infty$-quasi-isomorphism from the $L_\infty$-algebra to the BFV-complex.
Section~\ref{sec:coisotropic_deformation_problem} builds upon results from Section~\ref{sec:BRST-charges}, and describes how the BFV-complex controls the coisotropic deformation problem and encodes the moduli spaces under Hamiltonian and Jacobi equivalence.
As already mentioned, Section~\ref{sec:example} describes a non-trivial example of coisotropic submanifold in a contact manifold and shows how the associated BFV-complex controls its non-formal coisotropic deformation problem.

Finally, the paper contains three appendices.
The first one collects some necessary facts about (graded symmetric) multi-derivations and graded Jacobi structures on graded line bundles.
The second one provides a self-contained version of the step-by-step obstruction method well-suited for the aims of this paper. 
The third one contains two technical results which are needed in the main body of the paper but, if included there, would have delayed the full development of the principal ideas.


\section{Jacobi bundles and coisotropic submanifolds}
\label{sec:abstract_jac_mfd}

In this section we briefly recall the main notions and results concerning Jacobi bundles and coisotropic submanifolds which are important in what follows.
The interested reader can find a more comprehensive introduction to the subject in~\cite{LOTV,tortorella2017thesis} and references therein.
However note that some changes in terminology have been introduced in the present paper, and~\cite{tortorella2017thesis}, with respect to a previous version of~\cite{LOTV}.
Specifically, according to a terminology due to Marle~\cite{marle1991jacobi}, in the present paper we use the terms Jacobi bundles and Jacobi manifolds to refer to what were respectively called abstract Jacobi structures and abstract Jacobi manifolds in the first version of~\cite{LOTV}.

After introducing Jacobi bundles, and presenting important examples, we discuss the existence of a Jacobi algebroid structure on the first jet bundle $J^1 L$ of a Jacobi bundle $(L \to M,\{-,-\})$ (Proposition~\ref{prop:jacb2}), first discovered by Kerbrat and Souici-Benhammadi in the special case $L = M \times \bbR\to M$ \cite{kerbrat1993} (see \cite{crainic2013jacobi} for the general case).
Recall that a Jacobi algebroid is a Lie algebroid together with an action on a line bundle (see also below).
Going further we propose some equivalent characterizations of coisotropic submanifolds in a Jacobi manifold $(M, L, \{-,-\})$ (Definition \ref{def:cois}).
Finally, we establish a 1--1 correspondence between coisotropic submanifolds and certain Jacobi subalgebroids of the Jacobi algebroid $(J^1 L, L)$ (Proposition \ref{prop:conormal}).
In particular, this yields a way how to introduce the characteristic distribution of a coisotropic submanifold, which allows to perform singular Jacobi reduction.

\subsection{Jacobi bundles and their associated Jacobi algebroids}
\label{subsec:a_jac_mfd}

Let $M$ be a smooth manifold.

\begin{definition}
	\label{definition:abstractj}
	A \emph{Jacobi structure}, or a \emph{Jacobi bracket}, on a line bundle $L\to M$ is a Lie bracket $\{-,-\} :\Gamma(L)\times\Gamma (L)\rightarrow \Gamma(L)$, which is a first order differential operator, hence a derivation, in both entries.
	A \emph{Jacobi bundle} (over $M$) is a line bundle (over $M$) equipped with a Jacobi bracket.
	A \emph{Jacobi manifold} is a manifold equipped with a Jacobi bundle over it.
\end{definition}

\begin{remark}
	\label{rem:GJ_algebras_of_Jacobi_manifolds}
	A Jacobi structure on a line bundle $L\to M$ is exactly the same thing as a structure of Gerstenhaber-Jacobi algebra (concentrated in degree $0$) on $(C^\infty(M),\Gamma(L))$ (see Definition~\ref{def:GJ_algebra}).
\end{remark}

We want to emphasize that \emph{skew-symmetric} multi-derivations from $L$ to $L$ are in 1--1 correspondence with \emph{graded symmetric} multi-derivations from $L[1]$ to $L[1]$ via d\'ecalage isomorphism.
Further, as recalled in Proposition~\ref{prop:Jacobi_bi-do}, within this correspondence, Jacobi structures on $L\to M$  can be equivalently seen as \emph{Jacobi bi-derivations} on $L[1]$, i.e.~degree 1 graded symmetric bi-derivations $J$ from $L[1]$ to $L[1]$, such that $\ldsb J,J\rdsb=0$, where $\ldsb -,- \rdsb$ is the Schouten-Jacobi bracket (cf.~Proposition~\ref{prop:GJalgebra} for the construction of $\ldsb-,-\rdsb$).
Basic facts, including conventions and notation, about graded symmetric multi-derivations and, in particular, (graded) Jacobi structures on (graded) line bundles are collected in Appendix~\ref{app:graded_multi-do}.
In the following, we will freely use those conventions and notation, often without further comments.
Since this is not really standard material, we suggest the reader to go through Appendix~\ref{app:graded_multi-do} before proceeding.

\begin{example}
	\label{ex:1}
	Jacobi structures encompass several well-known geometric structures.
	Here are some first examples.
	\begin{enumerate}
		\item For any (non-necessarily coorientable) contact manifold $(M,C)$, the (non-necessarily trivial) line bundle $L:=TM/C$ is naturally equipped with a Jacobi structure (see~\cite[Section~5]{LOTV} and~\cite[Section 2.5.1]{tortorella2017thesis}).
		\item Any locally conformal symplectic (l.c.s.) manifold is naturally equipped with a Jacobi bundle (see~\cite[Appendix A]{vitagliano2014vector} and~\cite[Section 2.5.2]{tortorella2017thesis}).
	\end{enumerate}
\end{example}

\begin{example}
	\label{ex:2}
	We recover Lichnerowicz's notion of Jacobi pair on a manifold $M$ (see~\cite{Lich1978}) as the special case of Definition~\ref{definition:abstractj} when $L$ is the trivial line bundle $\bbR_M:=M\times\bbR\to M$.
	Recall that a \emph{Jacobi pair} $(\Lambda,\Gamma)$ on $M$ consists of a bivector $\Lambda$ and a vector field $\Gamma$ on $M$ such that
	\begin{equation*}
	\ldsb\Gamma,\Lambda\rdsb=0,\qquad \ldsb\Lambda,\Lambda\rdsb+2\Gamma\wedge\Lambda=0,
	\end{equation*} 
	with $\ldsb-,-\rdsb$ denoting here the Schouten-Nijenhuis bracket on multivector fields.
	So, in particular, if $\Gamma=0$ we recover the \emph{Poisson bivectors} $\Lambda$ on $M$.
	Jacobi pairs $(\Lambda,\Gamma)$ on $M$ identify with Jacobi structures $\{-,-\}$ on $\bbR_M\to M$ by means of the relation
	\begin{equation*}
	\{f,g\}=\Lambda(df,dg)+f\Gamma(g)-g\Gamma(f),
	\end{equation*}
	for all $f,g\in C^\infty(M)$.
	In particular, Poisson bivectors $\Lambda$ on $M$ identify with those Jacobi structures $\{-,-\}$ on $\bbR_M\to M$ such that $\{1,-\}=0$, by means of $\Lambda(df,dg)=\{f,g\}$.
\end{example}

Let $(M,L,\{-,-\})$ be a Jacobi manifold and $\lambda\in\Gamma(L)$.
The associated Hamiltonian derivation $\Delta_\lambda\in\Diff L$ is the first order differential operator, hence derivation of $L$, defined by $\Delta_\lambda:=\{\lambda,-\}$.
The associated Hamiltonian vector field $X_\lambda\in\frakX(M)$ is given by the symbol of $\Delta_\lambda$.
Here $DL$ is the space of derivations of $L$ (see Appendix~\ref{app:graded_multi-do} for details).

The following proposition (cf.~\cite[Theorem 1]{kerbrat1993}, \cite[(2.7)]{iglesias2001generalized}, \cite[Theorem 13]{GM2001}) shows how each Jacobi bundle determines a Jacobi algebroid.
Jacobi algebroids were first introduced in~\cite{GM2001} and~\cite{iglesias2001generalized} under the name generalized Lie algebroids.
A more general definition, adapted to the realm of non-necessarily trivial line bundles, is presented in~\cite[Definition~2.6]{LOTV} under the name abstract Jacobi algebroid.
We adopt the latter definition but use the simpler terminology Jacobi algebroid.
A \emph{Jacobi algebroid} is a pair $(A,L)$ where $A\to M$ is a Lie algebroid and $L\to M$ is a line bundle carrying a representation $\nabla$ of $A$.
A \emph{Jacobi subalgebroid} of $(A,L)$ is a pair $(B,\ell)$ where $B\to N$ is a Lie subalgebroid of $A$, and $\ell=L|_N\to N$.
Notice that, since $B\to N$ is a Lie subalgebroid of $A\to M$, $\nabla_b$ is tangent to $\ell=L|_N\to N$ for all $b\in\Gamma(B)$, and so $\nabla$ restricts to a representation of $B$ in $\ell$.
As a consequence a Jacobi subalgebroid $(B,\ell)$ inherits a natural structure of Jacobi algebroid from $(A,L)$.

\begin{proposition}
	\label{prop:jacb2}
	Let $(L,J \equiv \{-,-\})$ be a Jacobi bundle over $M$.
	Then there is a canonical Jacobi algebroid structure on $(J^1 L,L)$ whose anchor map $\rho_{\J}$, Lie bracket $[-,-]_{\J}$, and representation $\nabla^{\J}$ are uniquely determined by setting	\begin{equation*}
	\rho_{\J} (j^1 \lambda)=X_\lambda,\qquad
	[j^1 \lambda , j^1 \mu]_{\J}=j^1 \{\lambda,\mu\},\qquad
	\nabla^{\J}_{j^1 \lambda}\mu=\{\lambda,\mu\},
	\end{equation*}
	for all $\lambda, \mu \in \Gamma (L)$.
\end{proposition}

\begin{remark}
	Proposition~\ref{prop:jacb2} extends to Jacobi bundles a similar well-known result for Poisson manifolds.
	More details about Jacobi algebroids and explicit formulas for the structure maps $\rho_{\J}$, $[-,-]_{\J}$ and $\nabla^{\J}$ of the Jacobi algebroid associated to a Jacobi bundle can be found in~\cite[Section~2.2]{LOTV}.
\end{remark}

\begin{remark}
	\label{rem:jmap}
	Let $(M,L,\{-,-\})$ be a Jacobi manifold.
	The Hamiltonian vector fields generate a singular distribution $\calK \subset TM$, called \emph{the characteristic distribution} of $(M,L,\{-,-\})$.
	The Jacobi manifold is said to be \emph{regular}/\emph{transitive} when its characteristic distribution $\calK$ has constant/full rank.
	Since $\calK$ coincides with the image of $\rho_{\J}$, the anchor of the Jacobi algebroid of $(M,L,\{-,-\})$, the characteristic distribution $\calK$ is integrable in the sense of Stefan and Sussmann.
	Consequently, $\calK$ defines a singular foliation $\calF$ of $M$, whose leaves are called \emph{characteristic leaves} of the Jacobi manifold.
	For each characteristic leaf $\calC$, the ambient Jacobi structure on $L$ determines a unique transitive Jacobi structure $\{-,-\}_{\calC}$ on $L|_{\calC}$, such that the inclusion $L|_\calC \hookrightarrow L$ is a Jacobi map, or equivalently such that
	\begin{equation*}
	\{ \lambda |_{\calC}, \mu|_{\calC}\}_\calC= \{ \lambda, \mu \}|_{\calC},
	\end{equation*}
	for all $\lambda,\mu\in\Gamma(L)$.
	Moreover, a transitive Jacobi manifold $(M,L,\{-,-\})$ is either an abstract l.c.s.~manifold in the sense of \cite{vitagliano2014vector} (if $\dim M$ is even) or a contact manifold (if $\dim M$ is odd)~\cite{kirillov1976local}.
	More details about transitive Jacobi manifolds and the characteristic foliation of Jacobi manifolds can be found in~\cite[Sec.~2.5 and 2.6]{tortorella2017thesis}. 
\end{remark}

\begin{remark}
	\label{rem:d_J_ungraded}
	In view of Definition~\ref{def:Jacobi_bi-do} and Proposition~\ref{prop:Jacobi_bi-do}, a Jacobi structure $\J \equiv \{-,-\}$ on $L\to M$ determines differential graded Lie algebra $(\Diff^\star(L[1]),d_\J,\ldsb-,-\rdsb)$, with $d_\J:=\ldsb\J,-\rdsb$.
	Here $\ldsb-,-\rdsb$ denotes the Schouten-Jacobi bracket (cf.~Proposition~\ref{prop:GJalgebra} and Remark~\ref{rem:Gerstenhaber_product_and_SJ_bracket}).
	The cohomology of $(\Diff^\star(L[1]),d_\J,\ldsb-,-\rdsb)$ will be denoted by $H_{CE}(M,L,J)$ (cf.~Remark~\ref{rem:d_J}).
	Additionally, $d_{\J}$ coincides with  the de Rham differential $d_{J^1L,L}$
	of the Jacobi algebroid $(J^1L,L)$ associated with $(M,L,\J)$ (cf.~\cite[Remark~2.13]{LOTV}).
\end{remark}

\subsection{Coisotropic submanifolds and their subalgebroids}

There exists a notion of coisotropic submanifolds of Jacobi manifolds, unifying coisotropic submanifolds of Poisson, l.c.s., and contact manifolds.

Let $(M,L,\{-,-\})$ be a Jacobi manifold, and let $S \subset M$ be a submanifold.
Denote by $\Gamma_S$ the set of sections $\lambda\in\Gamma(L)$ such that $\lambda |_S = 0$.
\begin{definition}
	\label{def:cois}
	$S$ is said to be \emph{coisotropic} if the following equivalent conditions hold:
	\begin{enumerate}
		\item \label{2} $\Gamma_S$ is a Lie subalgebra in $\Gamma (L)$,
		\item \label{3} $X_\lambda$ is tangent to $S$, for all $\lambda \in \Gamma_S$.
	\end{enumerate}
\end{definition}

\begin{remark}
	Definition~\ref{def:cois} is well-posed~\cite[Lemma~3.1]{LOTV}.
	Analogous characterizations of coisotropic submanifolds of a Poisson manifold are well known~\cite[Section~2]{cattaneo2007relative} (see also~\cite[Lemma 13.3]{oh2005deformations} for the symplectic case).
	Moreover any coisotropic submanifold (in particular any Legendrian submanifold) in a contact manifold is coisotropic wrt~the associated Jacobi structure (see~\cite[Section~5.1]{LOTV} for more details).
\end{remark}

Denote by $NS:=TM|_S/TS$ the normal bundle of $S$ in $M$.
Then $N^\ast S:= (NS)^\ast$, the conormal bundle of $S$ in $M$, is canonically identified with $T^0 S\subset T^\ast M |_S$, the annihilator of $TS$ in $T^\ast M$.
Set $\ell :=L|_S\to S$, and $N_\ell S:= NS\otimes\ell^\ast$.
Moreover denote $N_\ell{}^\ast S:= (N_\ell S)^\ast = N^\ast S \otimes \ell$.
The vector bundle $N_\ell {}^\ast S$ will be also regarded as a vector subbundle of $(J^1 L)|_S$ via the vector bundle embedding
\begin{equation}
\label{eq:embedding_ell_adjoint}
N_\ell {}^\ast S \lhook\joinrel\longrightarrow (T^\ast M \otimes L)|_S \overset{\gamma}{\longrightarrow} J^1 L|_S,
\end{equation}
where the vector bundle map $\gamma:T^\ast M\otimes L\to J^1L$ is the \emph{co-symbol} map defined by $\gamma(df\otimes\lambda)=j^1(f\lambda)-fj^1\lambda$, for all $f\in C^\infty(M)$, and $\lambda\in\Gamma(L)$.
Here $J^1 L \to M$ is the first jet bundle of $L$.
Consequently, there exists a unique degree $0$ graded module epimorphism 
\begin{equation*}
P:\Diff^\star(L[1])\longrightarrow\Gamma(\wedge(N_\ell S)\otimes\ell)[1]=:\frakg(S)[1],
\end{equation*}
covering a degree $0$ graded algebra epimorphism $\underline{P}:\Diff^\star(L[1],\bbR_M)\to\Gamma(\wedge(N_\ell S))$, such that
\begin{itemize}
	\item $P$ acts on $\Gamma(L)$ by taking the restriction to $S$, and
	\item $P$ acts on $\Diff(L[1])$ by pulling-back along the composition~\eqref{eq:embedding_ell_adjoint}.
\end{itemize}
\begin{remark}
	\label{rem:derivation_along_morphism}
	Let us recall, for the reader convenience, that, if $M_i$ is a (graded) module over a (graded) algebra $A_i$, $i=0,1$, then a module morphism $\phi:M_0\to M_1$, covering an algebra morphism $\underline{\smash{\phi}}:A_0\to A_1$, is a linear map $\phi:M_0\to M_1$ such that $\phi(am)=\underline{\smash{\phi}}(a)\phi(m)$, for all $a\in A_0$, $m\in M_0$.
	For future use we remark here also what follows.
	Let $\phi:M_0\to M_1$ be a degree $0$ graded module morphism covering a degree $0$ algebra morphism $\underline{\smash{\phi}}:A_0\to A_1$.
	A \emph{degree $k$ graded derivation covering $\underline{\smash{\phi}}$} is a degree $k$ graded linear map $X:A_0\to A_1$ such that $X(aa')=X(a)\underline{\smash{\phi}}(a')+(-)^{k|a|}\underline{\smash{\phi}}(a)X(a')$ for all homogeneous $a,a'\in A_0$.
	Additionally, \emph{a degree $k$ graded derivation covering $\phi$}, with symbol $X$, is a degree $k$ graded linear map $\square:M_0\to M_1$ such that $\square(am)=X(a)\phi(m)+(-)^{k|a|}\underline{{\phi}}(a)\square(m)$, for all homogeneous $a\in A_0,m\in M_0$.
\end{remark}

Now we further assume that $S\subset M$ is a closed submanifold, and consider $(J^1L,L)$ equipped with the Jacobi algebroid structure determined by the Jacobi structure $\J$ on $L$ (cf.~Proposition~\ref{prop:jacb2}).
Then the following proposition (cf.~\cite[Proposition 5.2]{iglesias2003jacobi}) establishes a 1--1 correspondence between coisotropic submanifolds and certain Jacobi subalgebroids of $(J^1 L,L)$.

\begin{proposition}
	\label{prop:conormal}
	The following conditions are equivalent:
	\begin{enumerate}
		\item\label{enum:conormal:1} $S$ is a coisotropic submanifold,
		\item\label{enum:conormal:2} $(N_\ell {}^\ast S, \ell)$ is a Jacobi subalgebroid of $(J^1 L, L)$,
		\item\label{enum:conormal:3} there exists a (unique) homological derivation $d_{N_\ell{}^\ast S,\ell}$ of the graded module $\frakg(S)$, over the graded commutative algebra $\Gamma(\wedge (N_\ell S))$, such that
		\begin{equation*}
		d_{N_\ell{}^\ast S,\ell}\circ P=P\circ d_{J^1L,L}.
		\end{equation*}
	\end{enumerate}
	If conditions~\eqref{enum:conormal:1}-\eqref{enum:conormal:3} hold, $d_{N_\ell{}^\ast S,\ell}$ coincides with de Rham differential of the Jacobi algebroid $(N_\ell{}^\ast S,\ell)$.
\end{proposition}

\begin{remark}
	\label{rem:defcois}
	Proposition~\ref{prop:conormal} extends to the Jacobi setting a similar  well-known result for coisotropic submanifolds of Poisson manifolds~\cite[Proposition 3.1.3]{Weinstein1988}, \cite[Proposition 5.1]{cattaneo2004integration}, \cite[Theorem 10.4.2]{mackenzie2005general}.
	More details about Jacobi subalgebroids and a streamlined proof of Proposition~\ref{prop:conormal} can be found in~\cite[Section~3]{LOTV}.
\end{remark}

\begin{remark}
	\label{rem:coisotropic_characteristic_foliation}
	Let $(M,L,\{-,-\})$ be a Jacobi manifold, and let $S\subset M$ be a closed coisotropic submanifold.
	In view of Definition~\ref{def:cois}, the vector fields $\left.X_\lambda\right|_S\in\frakX(S)$, with $\lambda\in\Gamma_S$, generate a singular tangent distribution $\calK_S \subset TS$, called \emph{the characteristic distribution} of $S$.
	Since $\calK_S$ coincides with the image of the anchor map of $(N_\ell{}^\ast S,\ell)$, the Jacobi algebroid associated with $S$, the characteristic distribution $\calK_S$ is integrable in the sense of Stefan and Sussmann.
	Hence $\calK_S$ defines a singular foliation $\calF_S$ of $S$, whose leaves are called \emph{characteristic leaves} of the coisotropic submanifold.
	Finally note that, in particular, if $S=M$, then $\calK_S$ reduces to $\calK$,  the characteristic distribution of $(M,L,\{-,-\})$ as defined in Remark~\ref{rem:jmap}.
\end{remark}

\subsection{Singular Jacobi reduction}
\label{sec:Jacobi_reduction}

Let $(M,L,\{-,-\})$ be a Jacobi manifold, and let $S\subset M$ be an arbitrary coisotropic submanifold.
Denote by $I_S\subset C^\infty(S)$ the ideal of functions vanishing on $S$ and by $\Gamma_S\subset\Gamma(L)$ the $C^\infty(M)$-submodule and Lie subalgebra of sections vanishing on $S$.
Define the associative subalgebra $N(I_S)\subset C^\infty(M)$, and the Lie subalgebra $N(\Gamma_S)\subset\Gamma(L)$ by setting
\begin{align*}
N(\Gamma_S)&:=\{\nu\in\Gamma(L)\colon\{\Gamma_S,\nu\}\subset\Gamma_S\},\\
N(I_S)&:=\{f\in C^\infty(M)\colon X_\lambda(f)\in I_S,\ \textnormal{for all}\ \lambda\in\Gamma_S\}.
\end{align*}
Clearly, $N(\Gamma_S)$ is the normalizer of $\Gamma_S$ in $\Gamma (L)$, and consists of those sections $\lambda\in\Gamma(L)$ such that $X_\lambda$ is tangent to $S$.
Moreover $N(I_S)$ consists of those functions $f\in C^\infty(M)$ which are constant along the leaves of $\calF_S$.
The pair $(C^\infty(M_{\textnormal{red}}),\Gamma(L_{\textnormal{red}})):=(N(I_S)/I_S,N(\Gamma_S)/\Gamma_S)$ admits an obvious structure of Gerstenhaber-Jacobi algebra (concentrated in degree $0$), that we call the \emph{reduced Gerstenhaber-Jacobi algebra of $S$}.
The latter is morally the Gerstenhaber-Jacobi algebra of the ``singular'' Jacobi manifold $(M_{\textnormal{red}},L_{\textnormal{red}},\{-,-\}_{\textnormal{red}})$ obtained by performing a singular reduction of $S$ wrt its characteristic foliation $\calF_S$.

If $S$ is closed, then there exists a canonical module isomorphism $\phi:\Gamma(L_{\textnormal{red}})\to H^0(N_\ell{}^\ast S,\ell)$, covering an algebra isomorphism $\underline{\smash\phi}:C^\infty(M_{\textnormal{red}})\to H^0(N_\ell{}^\ast S)$, defined by $\underline{\smash\phi}(f+I_S)=[f|_S]$, and $\phi(\lambda+\Gamma_S)=[\lambda|_S]$, for all $f\in N(I_S)$, and $\lambda\in N(\Gamma_S)$.

\section{Lifted graded Jacobi structures}
\label{sec:lifting_Jacobi_structures}

In this section we will describe a procedure to lift a given Jacobi structure $\J$ on an ordinary line bundle $L\to M$ to a graded Jacobi structure $\hat{\J}$ on a certain graded line bundle $\hat{L}\to\hat{M}$.
The interest in this procedure is at least two-fold.
First, it extends, exploiting simplified techniques, similar lifting procedures in~\cite{Rothstein1991} (symplectic case) and~\cite{herbig2007,schatz2009bfv} (Poisson case).
Second, as we will show in Sections~\ref{sec:BRST-charges} and~\ref{sec:BFV-complex}, the lifting of Jacobi structures represents the first step towards the construction of the BFV-complex of a coisotropic submanifold (see also~\cite{schatz2009bfv}).

The lifting procedure is based on the set of contraction data that Proposition~\ref{prop:1contraction_data_1} associates to a $\Diff L$-connection.
In fact Theorem~\ref{theor:existence_BFV_brackets} will outline how to use these contraction data to inductively construct $\hat{\J}$ by implementing the ``step-by-step obstruction'' method of homological perturbation theory.
The same method can be used to construct a BRST charge (cf.~Theorem~\ref{theor:existence_BRST-charge}), and goes back to Stasheff~\cite{Stasheff1997}.

Notice that there is an alternative approach to the lifting of Jacobi structures.
Namely, the scheme adopted by Sch\"atz~\cite{schatz2009bfv}, based on the same set of contraction data~\eqref{eq:1contraction_data_1}, extends from the Poisson to the Jacobi setting.
To do this we need to use two main ingredients:
\begin{itemize}
	\item homotopy transfer along the contraction data~\eqref{eq:1contraction_data_1} to lift the quasi-isomorphism of co-chain complexes $i_\nabla$ to a quasi-isomorphism of $L_\infty$-algebras $\hat{i_\nabla}$ (see Section~\ref{subsec:first_relevant_contraction_data} for more details, including a definition, about $i_\nabla$), 
	\item transfer of formal Maurer-Cartan (MC) elements via $L_\infty$-quasi-isomorphisms to transform a Jacobi bi-derivation $\J$ on $L[1]$ into a new Jacobi bi-derivation $\hat{\J}$ on $\hat{L}[1]$.
\end{itemize}
However, we have preferred the first approach to the second one, because, in our opinion, it is simpler and does not involve unnecessary sophisticated tools.

\subsection{The initial setting}
\label{subsec:initial_setting}

Let $E\to M$ be a vector bundle, and let $L\to M$ be a line bundle.
Define the vector bundle $E_L\to M$ by setting $E_L:=E\otimes L^\ast$.

Denote by $\hat{M}$ the graded manifold, with support $M$, represented by the graded vector bundle $\pi:(E_L)^\ast[1]\oplus E[-1]\to M$, and by $\hat{L}$ the graded line bundle over $\hat{M}$ given by $\hat{L}:=\pi^\ast L\to\hat{M}$.
This means that $C^\infty(\hat{M})$, the unital associative $\bbZ$-graded commutative $C^\infty(M)$-algebra of functions on $\hat{M}$, and $\Gamma(\hat{L})$, the graded $C^\infty(\hat{M})$-module of sections of $\hat{L}\to\hat{M}$, are given by
\begin{gather*}
C^\infty(\hat{M})=S_{{\scriptscriptstyle C^\infty(M)}}\Gamma(E_L[-1]\oplus E^\ast[1]),\qquad
\Gamma(\hat{L})=S_{{\scriptscriptstyle C^\infty(M)}}\Gamma(E_L[-1]\oplus E^\ast[1])\underset{{\scriptscriptstyle C^\infty(M)}}{\otimes}\Gamma(L),
\end{gather*}
where, as in what follows, $S_{R}M$ denotes the graded commutative algebra of a graded module $M$ over a (non graded) algebra $R$.
\begin{remark}
	The algebra $C^\infty(\hat{M})=\bigoplus_{n\in\bbZ}C^\infty(\hat{M})^n$ is graded commutative wrt the $\bbZ$-grading  provided by the \emph{total ghost number} $n$, where
	\begin{equation*}
	C^\infty(\hat{M})^n:=\bigoplus_{\genfrac{}{}{0pt}{}{(h,k)\in\bbN^2}{h-k=n}}C^\infty(\hat{M})^{(h,k)},\qquad\textnormal{with}\quad C^\infty(\hat{M})^{(h,k)}:=\Gamma((\wedge^h E_L)\otimes(\wedge^k E^\ast)).
	\end{equation*}
	Hence its multiplication is also compatible with the finer $\bbN^2$-grading provided by the \emph{ghost/anti-ghost bi-degree} $(h,k)$.
	Similarly, the $C^\infty(\hat{M})$-module structure on $\Gamma(\hat{L})=\bigoplus_{n\in\bbZ}\Gamma(\hat{L})^n$ is graded wrt the $\bbZ$-grading provided by the \emph{total ghost number} $n$:
	\begin{equation*}
	\Gamma(\hat{L})^n:=\bigoplus_{\genfrac{}{}{0pt}{}{(h,k)\in\bbN^2}{h-k=n}}\Gamma(\hat{L})^{(h,k)},\qquad\textnormal{with}\quad \Gamma(\hat{L})^{(h,k)}:=\Gamma((\wedge^h E_L)\otimes(\wedge^k E^\ast)\otimes L).
	\end{equation*}
	Hence its $C^\infty(\hat{M})$-module structure is also compatible wrt the finer $\bbN^2$-grading provided by the \emph{ghost/anti-ghost bi-degree} $(h,k)$.
\end{remark}

\begin{remark}
	\label{rem:local_frames1}
	Let $x^i$ be an arbitrary local coordinate system on $M$.
	Fix a local frame $\xi^A$ on $E\to M$, and denote by $\xi^\ast_A$ the dual local frame on $E^\ast\to M$.
	Fix also a local frame $\mu$ on $L\to M$, and denote by $\mu^\ast$ the dual local frame on $L^\ast\to M$.
	Then the $C^\infty(M)$-algebra $C^\infty(\hat{M})$ is locally generated by
	\begin{equation*}
	\underbrace{\xi^A\otimes\mu^\ast}_{(1,0)}\qquad \underbrace{\xi^\ast_B}_{(0,1)},
	\end{equation*}
	where the subscripts denote the bi-degrees.
	Moreover, for all $(h,k)\in\bbN^2$, a local generating system of the $C^\infty(M)$-module $C^\infty(\hat{M})^{(h,k)}$ is given by
	\begin{equation*}
	(\xi^{A_1}\otimes\mu^\ast)\cdots(\xi^{A_h}\otimes\mu^\ast)\xi^\ast_{B_1}\cdots\xi^\ast_{B_k},
	\end{equation*}
	and a local generating system of the $C^\infty(M)$-module $\Gamma(\hat{L})^{(h,k)}$ is given by
	\begin{equation*}
	(\xi^{A_1}\otimes\mu^\ast)\cdots(\xi^{A_h}\otimes\mu^\ast)\xi^\ast_{B_1}\cdots\xi^\ast_{B_k}\otimes\mu.
	\end{equation*}
	Here, and in what follows, the graded commutative product is denoted by juxtaposition.
	Hence an arbitrary $f\in C^\infty(\hat{M})$ and an arbitrary $\lambda\in\Gamma(\hat{L})$ admit the following local expression
	\begin{equation*}
	f=f^{\bfC}_{\bfB}(\xi^{\bfB}\otimes\mu^\ast)\xi_{\bfC}^\ast,\qquad\lambda=f^{\bfC}_{\bfB}(\xi^{\bfB}\otimes\mu^\ast)\xi_{\bfC}^\ast\otimes\mu,
	\end{equation*}
	where, for any ordered $n$-tuple $\bfB=(B_1,\ldots,B_n)$, we understand the following abbreviations: $\xi^{\bfB}\otimes\mu^\ast$ for $(\xi^{B_1}\otimes\mu^\ast)\cdots(\xi^{B_n}\otimes\mu^\ast)$, and $\xi_{\bfB}^\ast$ for $\xi_{B_1}^\ast\cdots\xi_{B_n}^\ast$.
	This notation, together with $|\bfB|=n$, will be adopted below without further comments.
\end{remark}

\begin{remark}
	\label{rem:compatible_gradings}
	A graded symmetric $n$-ary derivation $\Delta\in\Diff^n(\hat{L}[1],\bbR_{\hat{M}})$ is said to have (ghost/anti-ghost) bi-degree $(h,k)\in\bbZ^2$ if
	\begin{equation*}
	\Delta(\Gamma(\hat{L})^{(p_1,q_1)}\times\cdots\times\Gamma(\hat{L})^{(p_n,q_n)})\subseteq C^\infty(\hat{M})^{(h+\sum_ip_i,k+\sum_iq_i)}.
	\end{equation*}
	Denote by $\Diff^n(\hat{L}[1],\bbR_{\hat{M}})^{(h,k)}\subseteq\Diff^n(\hat{L}[1],\bbR_{\hat{M}})$ the $C^\infty(M)$-submodule of graded symmetric $n$-ary derivations of bi-degree $(h,k)$.
	The associative algebra $\Diff^\star(\hat{L}[1],\bbR_{\hat{M}})$ is graded commutative wrt the $\bbZ$-grading provided by the \emph{total ghost number} $K$ as follows
	\begin{equation*}
	\Diff^\star(\hat{L}[1],\bbR_{\hat{M}}):=\bigoplus_{K\in\bbZ}\Diff^\star(\hat{L}[1],\bbR_{\hat{M}})^K,\quad\textnormal{with}\ \Diff^\star(\hat{L}[1],\bbR_{\hat{M}})^K=\bigoplus_{K=n+h-k}\Diff^n(\hat{L}[1],\bbR_{\hat{M}})^{(h,k)}.
	\end{equation*}
	Moreover, the product is compatible with both the arity $n$, and the ghost/anti-ghost bi-degree $(h,k)$, i.e.~$\Diff^{n_1}(\hat{L}[1],\bbR_{\hat{M}})^{(h_1,k_1)}\cdot\Diff^{n_2}(\hat{L}[1],\bbR_{\hat{M}})^{(h_2,k_2)}\subseteq\Diff^{n_1+n_2}(\hat{L}[1],\bbR_{\hat{M}})^{(h_1+h_2,k_1+k_2)}$.
\end{remark}

\begin{remark}
	\label{rem:local_generating_system_for_do_from_calL_to_calA}
	A local generating system of the $C^\infty(\hat{M})$-algebra $\Diff^\star(\hat{L}[1],\bbR_{\hat{M}})$ is provided by
	\begin{equation}\label{gen_2}
	\underbrace{\mu^\ast}_{1},\quad\underbrace{\slashed\Delta_i}_{1},\quad\underbrace{\slashed\Delta_A}_{0},\quad\underbrace{\slashed\Delta^A}_{2},
	\end{equation}
	with $\mu^\ast\in\Diff^1(\hat{L}[1],\bbR_{\hat{M}})^{(0,0)}$, $\slashed\Delta_i\in\Diff^1(\hat{L}[1],\bbR_{\hat{M}})^{(0,0)}$, $\slashed\Delta_A\in\Diff^1(\hat{L}[1],\bbR_{\hat{M}})^{(-1,0)}$, and $\slashed\Delta^A\in\Diff^1(\hat{L}[1],\bbR_{\hat{M}})^{(0,-1)}$ defined by:
	\begin{align*}
	\mu^\ast(f\mu)&=f,&\mu^\ast(\xi^B)={}&\xi^B\otimes\mu^\ast,&\mu^\ast(\xi_B^\ast\otimes\mu)={}&\xi_B^\ast,\\
	\slashed\Delta_i(f\mu)&=\frac{\partial f}{\partial x^i},&\slashed\Delta_i(\xi^B)={}&0,&\slashed\Delta_i(\xi_B^\ast\otimes\mu)={}&0,\\
	\slashed\Delta_A(f\mu)&=0,&\slashed\Delta_A(\xi^B)={}&\delta_A^B,&\slashed\Delta_A(\xi_B^\ast\otimes\mu)={}&0,\\
	\slashed\Delta^A(f\mu)&=0,&\slashed\Delta^A(\xi^B)={}&0,&\slashed\Delta^A(\xi_B^\ast\otimes\mu)={}&\delta^A_B.
	\end{align*}
	In (\ref{gen_2}) the subscripts denote the total ghost number.
	Hence an arbitrary $\slashed\Delta\in\Diff^1(\hat{L}[1],\bbR_{\hat{M}})$ will be locally given by:
	\begin{equation*}
	\slashed\Delta=f^{\bfC}_{\bfB}(\xi^{\bfB}\!\otimes\!\mu^\ast)\xi_{\bfC}^\ast\mu^\ast+f^{i\bfC}_{\phantom{i}\bfB}(\xi^{\bfB}\!\otimes\!\mu^\ast)\xi_{\bfC}^\ast\slashed\Delta_i+f^{A\bfC}_{\phantom{A}\bfB}(\xi^{\bfB}\!\otimes\!\mu^\ast)\xi_{\bfC}^\ast\slashed\Delta_A+f^{\phantom{A}\bfC}_{A\bfB}(\xi^{\bfB}\!\otimes\!\mu^\ast)\xi_{\bfC}^\ast\slashed\Delta^A.
	\end{equation*}
\end{remark}

\begin{remark}
	A graded symmetric $n$-ary derivation $\square\in\Diff^n(\hat{L}[1])$ is said to have (ghost/anti-ghost) bi-degree $(h,k)$ if
	\begin{equation*}
	\square(\Gamma(\hat{L})^{(p_1,q_1)}\times\cdots\times\Gamma(\hat{L})^{(p_n,q_n)})\subseteq\Gamma(\hat{L})^{(h+\sum_ip_i,k+\sum_iq_i)}.
	\end{equation*}
	Denote by $\Diff^n(\hat{L}[1])^{(h,k)}\subseteq\Diff^n(\hat{L}[1])$ the $C^\infty(M)$-submodule of graded symmetric $n$-ary derivations of bi-degree $(h,k)$.
	Similarly as in Remark~\ref{rem:compatible_gradings}, all the algebraic structures on $\Diff^\star(\hat{L}[1])$ are compatible with both the arity $n$, the ghost/anti-ghost bi-degree $(h,k)$, and the total ghost number $K$
	\begin{equation*}
	\Diff^\star(\hat{L}[1])=\bigoplus_{K\in\bbZ}\Diff^\star(\hat{L}[1])^K,\quad\textnormal{with}\ \Diff^\star(\hat{L}[1])^K=\!\!\!\bigoplus_{K=n-1+h-k}\!\!\!\Diff^n(\hat{L}[1])^{(h,k)}.
	\end{equation*}
	In the following, for every $(h,k)\in\bbZ^2$, we denote by $\operatorname{pr}^{(h,k)}:\Diff^\star(\hat{L}[1]) \to\Diff^\star(\hat{L}[1])^{(h,k)}$ the projection onto the homogeneous component of bi-degree $(h,k)$.
\end{remark}

\begin{remark}
	\label{rem:local_generating_system_for_do_from_calL_to_calL}
	A local generating system for the $C^\infty(\hat{M})$-module $\Diff(\hat{L}[1])$ is provided by
	\begin{equation*}
	\underbrace{\id}_{0},\quad\underbrace{\Delta_i}_{0},\quad\underbrace{\Delta_A}_{-1},\quad\underbrace{\Delta^A}_{1},
	\end{equation*}
	with $\id\in\Diff(\hat{L}[1])^{(0,0)}$, $\Delta_i\in\Diff(\hat{L}[1])^{(0,0)}$, $\Delta_A\in\Diff(\hat{L}[1])^{(-1,0)}$, and $\Delta^A\in\Diff(\hat{L}[1])^{(0,-1)}$ defined by:
	\begin{align*}
	\id(f\mu)&=f\mu,&\id(\xi^B)={}&\xi^B,&\id(\xi_B^\ast\otimes\mu)={}&\xi_B^\ast\otimes\mu,\\
	\Delta_i(f\mu)&=\frac{\partial f}{\partial x^i}\mu,&\Delta_i(\xi^B)={}&0,&\Delta_i(\xi_B^\ast\otimes\mu)={}&0,\\
	\Delta_A(f\mu)&=0,&\Delta_A(\xi^B)={}&\delta_A^B\mu,&\Delta_A(\xi_B^\ast\otimes\mu)={}&0,\\
	\Delta^A(f\mu)&=0,&\Delta^A(\xi^B)={}&0,&\Delta^A(\xi_B^\ast\otimes\mu)={}&\delta^A_B\mu.
	\end{align*}
	Hence an arbitrary $\square\in\Diff(\hat{L}[1])$ will be locally given by:
	\begin{equation*}
	\square=f^{\bfC\phantom{i}}_{\phantom{\bfA}\bfB}(\xi^{\bfB}\!\otimes\!\mu^\ast)\xi_{\bfC}^\ast\id+f^{i\bfC}_{\phantom{A}\bfB}(\xi^{\bfB}\!\otimes\!\mu^\ast)\xi_{\bfC}^\ast\Delta_i+f^{A\bfC}_{\phantom{A}\bfB}(\xi^{\bfB}\!\otimes\!\mu^\ast)\xi_{\bfC}^\ast\Delta_A+f^{\phantom{i}\bfC}_{A\bfB}(\xi^{\bfB}\!\otimes\!\mu^\ast)\xi_{\bfC}^\ast\Delta^A.
	\end{equation*}
\end{remark}

In the following, \emph{with no risk of confusion}, $\mu^\ast,\slashed\Delta_i$  will also denote the local generating system of the $C^\infty(M)$-algebra $\Diff^\star(L[1],\bbR_M)$ defined by
\begin{equation*}
\mu^\ast(f\mu)=f,\quad\slashed\Delta_i(f\mu)=\frac{\partial f}{\partial x^i},
\end{equation*}
similarly $\id,\Delta_i$ will also denote the local generating system of the $C^\infty(M)$-module $\Diff(L[1])$ defined by
\begin{equation*}
\id(f\mu)=f\mu,\quad\Delta_i(f\mu)=\frac{\partial f}{\partial x^i}\mu.
\end{equation*}

\subsection{Existence and uniqueness of the liftings: the statements}
\label{subsec:lifted_graded_Jacobi}

The graded line bundle $\hat{L}\to\hat{M}$ admits a tautological graded Jacobi structure, given by the canonical bi-degree $(-1,-1)$ Jacobi bi-derivation $G$, with corresponding Jacobi bracket $\{-,-\}_G$.
The latter can be seen as the natural extension, from the Poisson setting to the Jacobi one, of the so-called big bracket (cf.~\cite{kosmann1996poisson}).

\begin{definition}
	\label{def:G}
	The \emph{tautological Jacobi bi-derivation} on $\hat{L}[1]$ is the unique $G\in\Diff^2(\hat{L}[1])^{(-1,-1)}$ such that
	\begin{equation}
	\label{eq:def:G}
	\{e,\alpha\}_G=\{\alpha,e\}_G=\alpha(e),
	\end{equation}
	for all $e\in\Gamma(\hat{L}[1])^{(1,0)}=\Gamma(E)$, and $\alpha\in\Gamma(\hat{L}[1])^{(0,1)}=\Gamma(E^\ast\otimes L)$.
\end{definition}

\begin{remark}
	\label{rem:d_G}
	Set $d_G:=\ldsb G,-\rdsb$.
	Then $d_G$ is a homological derivation of $(\Diff^\star(\hat{L}[1]),\ldsb-,-\rdsb)$, and its bi-degree is $(-1,-1)$ as well, i.e.
	\begin{equation*}
	d_G(\Diff^n(\hat{L}[1])^{(h,k)})\subseteq\Diff^{n+1}(\hat{L}[1])^{(h-1,k-1)}.
	\end{equation*}
	In particular, each section $\lambda\in\Gamma(\hat{L}[1])^{(0,0)}=\Gamma(L)[1]$ is a co-cycle wrt $d_G$, i.e.~$d_G\lambda=0$.
\end{remark}

\begin{remark}
	\label{rem:local_expression_G}
	Clearly, $G$ is locally given by $G=\slashed\Delta_A\slashed\Delta^A\otimes\mu$.
	Moreover, according to Remark~\ref{rem:Gerstenhaber_product_and_SJ_bracket}, $d_G$ is completely determined by:
	\begin{gather*}
	d_G(f\mu)=0,\quad d_G(\xi^A)=\Delta^A,\quad d_G(\xi_A^\ast\otimes\mu)=\Delta_A,\\
	d_G(\id)=G,\quad d_G(\Delta_i)=d_G(\Delta_A)=d_G(\Delta^A)=0.
	\end{gather*}
\end{remark}

\begin{definition}
	\label{def:BFV_brackets}
	A Jacobi bi-derivation $\hat{\J}$ on $\hat{L}[1]$ is said to be a \emph{lifting} of a Jacobi bi-derivation $J$ on $L[1]$ if
	\begin{enumerate}
		\item ${\operatorname{pr}^{(0,0)}}\circ\{-,-\}_{\hat{\J}}$ agrees with $\{-,-\}_G$ on $\Gamma(\hat{L})^{(1,0)}\oplus\Gamma(\hat{L})^{(0,1)}$,
		\item ${\operatorname{pr}^{(0,0)}}\circ\{-,-\}_{\hat{\J}}$ agrees with $\{-,-\}_{\J}$ on $\Gamma(\hat{L})^{(0,0)}$.
	\end{enumerate}
\end{definition}

Every Jacobi structure $\J$ on $L$ admits a unique lifting $\hat{\J}$ (up to isomorphisms).

\begin{theorem}[Existence]
	\label{theor:existence_BFV_brackets}
	Every Jacobi structure $\J$ on the line bundle $L\to M$ admits a lifting to a graded Jacobi structure $\hat{\J}$ on the graded line bundle $\hat{L}\to\hat{M}$.
	Specifically, for every fixed Jacobi structure $\J$ on $L\to M$, there exists a canonical map
	\begin{equation*}
	\{\textnormal{$\Diff L$-connections in $E\to M$}\}\longrightarrow\{\textnormal{liftings of $\J$ to $\hat{L}\to\hat{M}$}\},\quad \nabla\longmapsto\hat{\J}^\nabla.
	\end{equation*}
\end{theorem}

\begin{proposition}
	\label{prop:existence_BFV_brackets}
	If the $DL$-connection $\nabla$ in $E\to M$ is flat, then $\hat{\J}^\nabla=G+i_\nabla\J$.
\end{proposition}

\begin{theorem}[Uniqueness]
	\label{theor:uniqueness_BFV_brackets}
	Fix an arbitrary Jacobi structure $\J$ on $L\to M$.
	Let $\hat{\J}$ and $\hat{\J}'$ be Jacobi structures on $\hat{L}\to\hat{M}$.
	If $\hat{\J}$ and $\hat{\J}'$ are both liftings of $\J$, then there exists a degree $0$ graded automorphism $\phi$ of the graded line bundle $\hat{L}\to\hat{M}$ such that
	\begin{equation*}
	\hat{\J}=\phi^\ast\hat{\J}'.
	\end{equation*}
	Moreover such $\phi$ can be chosen so to have the additional property that
	\begin{equation}
	\label{eq:theor:uniqueness_BFV_brackets}
	\phi(\Omega)-\Omega\in\bigoplus_{k\geq 1}\Gamma(\hat{L})^{(p+k,q+k)},\qquad\Omega\in\Gamma(\hat{L})^{(p,q)},\ (p,q)\in\bbN^2.
	\end{equation}
\end{theorem}

The following proposition, describing an interesting property of liftings of Jacobi structures, generalizes a similar result obtained in the Poisson setting by Herbig (cf.~\cite[Theorem 3.4.5]{herbig2007}).
\begin{proposition}
	\label{prop:lifting_CE_cohom}
	If a Jacobi structure $\hat{\J}$ on the graded line bundle $\hat{L}\to\hat{M}$ is a lifting of a Jacobi structure $\J$ on $L\to M$, then the Chevalley--Eilenberg cohomologies of $\J$ and $\hat{\J}$ are isomorphic:
	\begin{equation*}
	H_{CE}^\bullet(M,L,\J)\simeq H_{CE}^\bullet(\hat{M},\hat{L},\hat{\J}).
	\end{equation*}
	More precisely, every $\Diff L$-connection in $E\to M$ determines a quasi-isomorphism from $(\Diff^\star(L[1]),d_{\J})$ to $(\Diff^\star(\hat{L}[1]),d_{\hat{\J}})$.
\end{proposition}

In order to develop the necessary technical tools first, we postpone the proofs of Theorems~\ref{theor:existence_BFV_brackets} and~\ref{theor:uniqueness_BFV_brackets}, and Propositions~\ref{prop:existence_BFV_brackets} and~\ref{prop:lifting_CE_cohom} to the end of this section.

\subsection{A first relevant set of contraction data}
\label{subsec:first_relevant_contraction_data}

\begin{definition}
	\label{def:contraction_data}
	A set of \emph{contraction data} between co-chain complexes $(\calK, \partial)$ and $(\underline{\smash{\calK}}, \underline{\smash{\partial}})$ consists of:
	\begin{enumerate}[label=(\roman*)]
		\item\label{item:1} a surjective co-chain map $q:(\calK,\partial)\to(\underline{\smash{\calK}},\underline{\smash{\partial}})$ that we simply call the \emph{projection},
		\item\label{item:2} an injective co-chain map $j:(\underline{\smash{\calK}},\underline{\smash{\partial}})\to(\calK,\partial)$, that we call the \emph{immersion}, such that $q\circ j=\id_{\underline{\smash{\calK}}}$,
		\item\label{item:3} a \emph{homotopy} $h : (\calK, \partial) \to (\calK, \partial)$ between $j\circ q$ and $\id_{\calK}$.
	\end{enumerate}
	Additionally, $q,j,h$ satisfy the following \emph{side conditions}: $h^2=0$, $h\circ j=0$, $q\circ h=0$.
\end{definition}

\subsubsection*{The projection}
\label{subsubsec:p}
There is a degree $0$ graded module morphism $p:\Diff^\star(\hat{L}[1])\to\Diff^\star(L[1])$, covering a degree $0$ graded $\bbR$-algebra morphism $\underline{\smash{p}}:\Diff^\star(\hat{L}[1],\bbR_{\hat{M}})\to\Diff^\star(L[1],\bbR_M)$, given by
\begin{equation}
\label{eq:p}
(p\square)(\lambda_1,\ldots,\lambda_k)=\operatorname{pr}^{(0,0)}(\square(\lambda_1,\ldots,\lambda_k)),
\end{equation}
for all $\square\in\Diff^k(\hat{L}[1])$, and $\lambda_1,\ldots,\lambda_k\in\Gamma(L[1])$.
The following proposition lists some properties of $p$.

\begin{proposition}
	\label{prop:p} \ 
	
	(1) $p$ preserves the arity, i.e.~$p(\Diff^k(\hat{L}[1]))\subseteq\Diff^k(L[1])$.
	
	(2) $p$ annihilates $\Diff^\star(\hat{L}[1])^{(h,k)}$, for all $(h,k)\in\bbZ^2\setminus\{(0,0)\}$, and induces a degree $0$ graded Lie algebra morphism from $\Diff^\star(\hat{L}[1])^{(0,0)}$ onto $\Diff^\star(L[1])$, i.e.
	\begin{equation*}
	p(\ldsb\square_1,\square_2\rdsb)=\ldsb p\square_1,p\square_2\rdsb,\qquad\square_1,\square_2\in\Diff^\star(\hat{L}[1])^{(0,0)}.
	\end{equation*} 
	
	(3) $p$ is a co-chain map from $(\Diff^\star(\hat{L}[1]),d_G)$ to $(\Diff^\star(L[1]),0)$, i.e.~$p\circ d_G=0$.
\end{proposition}

\begin{proof}
	Properties (1) and (2) are immediate consequences of~\eqref{eq:p} and Remark~\ref{rem:Gerstenhaber_product_and_SJ_bracket}.
	Moreover, from Remarks~\ref{rem:d_G} and~\ref{rem:Gerstenhaber_product_and_SJ_bracket}, it follows that
	\begin{align*}
	(d_G\square)(\lambda_1,\ldots,\lambda_{k+1})&=\ldsb \ldsb \ldots\ldsb d_G\square,\lambda_1\rdsb ,\ldots\rdsb ,\lambda_{k+1}\rdsb \\
	&=d_G\underbrace{\ldsb \ldsb \ldots\ldsb \square,\lambda_1\rdsb ,\ldots\rdsb ,\lambda_{k+1}\rdsb }_{=0}\\
	&\phantom{=}+\sum_{i=1}^{k+1}(-)^{|\square|-i}\ldsb \ldsb \ldots\ldsb \ldsb \ldots\ldsb \square,\lambda_1\rdsb ,\ldots\rdsb ,\underbrace{d_G\lambda_i}_{=0}\rdsb ,\ldots\rdsb ,\lambda_{k+1}\rdsb =0,
	\end{align*}
	for all homogeneous $\square\in\Diff^k(\hat{L}[1])$, and $\lambda_1,\ldots,\lambda_{k+1}\in\Gamma(L)[1]$.
	This concludes the proof.
\end{proof}

\begin{remark}
	\label{rem:local_expression_p}
	The module morphism $p:\Diff^\star(\hat{L}[1])\to\Diff^\star(L[1])$ is locally given by
	\begin{gather*}
	p(f\mu)=f\mu,\quad p(\xi^A)=0,\quad p(\xi_A^\ast\otimes\mu)=0,\\
	p(\id)=\id,\quad p(\Delta_i)=\Delta_i,\quad p(\Delta_A)=0,\quad p(\Delta^A)=0.
	\end{gather*}
\end{remark}

\begin{remark}
	\label{rem:BFV_brackets}
	Let $\J$ be a Jacobi structure on $L\to M$.
	An arbitrary $\hat{\J}\in\Diff^2(\hat{L}[1])^1$ decomposes as follows
	\begin{equation*}
	\hat{\J}=\sum_{k=0}^\infty\hat{\J}_k,
	\end{equation*}
	where $\hat{\J}_k\in\Diff^2(\hat{L}[1])^{(k-1,k-1)}$, for all $k\in\bbN$.
	It follows that the liftings of $\J$ are given by the degree $1$ graded symmetric bi-derivations $\hat{\J}$ such that
	\begin{gather}
	\label{eq:rem:BFV_brackets1}
	\hat{\J}_0=G\\
	\label{eq:rem:BFV_brackets2}
	p(\hat{\J}_1)=\J\\
	\label{eq:rem:BFV_brackets3}
	2d_G\hat{\J}_k+\sum_{i=1}^{k-1}\ldsb \hat{\J}_i,\hat{\J}_{k-i}\rdsb = 0 ,\quad\textnormal{for all }k>0.
	\end{gather}
\end{remark}

The line bundle $L\to M$ comes equipped with a canonical flat $\Diff L$-connection, i.e.~the tautological representation given by the identity map $\id:\Diff L\to\Diff L$.
Accordingly each choice of a $\Diff L$-connection $\nabla$ in $E\to M$ determines a $\Diff L$-connection in $E^\ast\otimes L\to M$, again denoted by $\nabla$.
This shows that every connection $\nabla:\Diff L\to \Diff E$ admits a canonical extension $\nabla:\Diff L\to(\Diff \hat{L})^{(0,0)}\subset\Diff \hat{L}$.
Note here that $\Diff L=\Diff(L[1])$ and $\Diff \hat{L}=\Diff(\hat{L}[1])$.

\begin{lemma}
	\label{lem:metric_connection}
	Let $\nabla$ be a $\Diff L$-connection in $E\to M$.
	Its canonical extension $\nabla:\Diff L\to\Diff \hat{L}$ takes values in $\mathfrak{aut}(\hat{M},\hat{L},G)$, i.e., for all $\square\in\Diff L$, the following equivalent conditions hold:
	\begin{itemize}
		\item\label{en:lem:metric_connection_1}
		$\nabla_\square$ is a Jacobi derivation, or, which is the same, an infinitesimal Jacobi automorphism of $(\hat{M},\hat{L},G)$, i.e.
		\begin{equation*}
		\nabla_\square\{\lambda_1,\lambda_2\}_G=\{\nabla_\square\lambda_1,\lambda_2\}_G+\{\lambda_1,\nabla_\square\lambda_2\}_G,\qquad\lambda_1,\lambda_2\in\Gamma(\hat{L}),
		\end{equation*}
		\item\label{en:lem:metric_connection_2}
		$\nabla_\square$ is a co-cycle of the co-chain complex $(\Diff^\star(\hat{L}[1]),d_G)$, i.e.~$d_G(\nabla_\square)=0$.
	\end{itemize}
\end{lemma}

\begin{proof}
	Fix an arbitrary $\square\in\Diff L$.
	From Remark~\ref{rem:Gerstenhaber_product_and_SJ_bracket}, for all $\alpha\in\Gamma(E^\ast\otimes L)$ and $e\in\Gamma(E)$, a straightforward computation shows that
	\begin{equation*}
	\ldsb G,\nabla_\square\rdsb (\alpha,e)=0,
	\end{equation*}
	and so, since $\ldsb G,\nabla_\square\rdsb \in\Diff^2(\hat{L}[1])^{(-1,-1)}$, we get that $d_G(\nabla_\square)=0$.
\end{proof}

\subsubsection*{The immersion}
There is a degree $0$ graded module morphism $i_\nabla:\Diff^\star(L[1])\to\Diff^\star(\hat{L}[1])$, covering a degree $0$ graded algebra morphism $\underline{\smash{i_\nabla}}:\Diff^\star(L[1],\bbR_M)\to\Diff^\star(\hat{L}[1],\bbR_{\hat{M}})$, completely determined by
\begin{equation}
\label{eq:i_nabla}
i_\nabla(\lambda)=\lambda,\qquad i_\nabla(\square)=\nabla_\square,
\end{equation}
for all $\lambda\in\Gamma(L)$, and $\square\in\Diff(L[1])$.
In the following proposition we list some properties of $i_\nabla$.

\begin{proposition}
	\label{prop:i_nabla} \ 
	
	(1) $i_\nabla$ pr\underline{}eserves the arity, and takes values of bi-degree $(0,0)$, i.e.
	\begin{equation*}
	i_\nabla(\Diff^k(L[1]))\subseteq\Diff^k(\hat{L}[1])^{(0,0)}.
	\end{equation*}
	
	(2) $i_\nabla$ is a section of $p$, i.e.~$p\circ i_\nabla=\id$ on $\Diff^\star(L[1])$.
	
	(3) $i_\nabla$ is a co-chain map from $(\Diff^\star(L[1]),0)$ to $(\Diff^\star(\hat{L}[1]),d_G)$, i.e.~$d_G\circ i_\nabla=0$.
\end{proposition}

\begin{proof}\ 
	
	(1) It is an immediate consequence of~\eqref{eq:i_nabla}.
	
	(2) Since $p\circ i_\nabla:\Diff^\star(L[1])\to\Diff^\star(L[1])$ is a module morphism, covering the algebra morphism $\underline{\smash{p}}\circ\underline{\smash{i_\nabla}}:\Diff^\star(L[1],\bbR_M)\to\Diff^\star(L[1],\bbR_M)$, it is enough to check that $p\circ i_\nabla$ agrees with the identity on $\Gamma(L)[1]$ and $\Diff(L[1])$.
	This is exactly the case because of~\eqref{eq:p} and~\eqref{eq:i_nabla}.
	
	(3) Since $d_G\circ i_\nabla:\Diff^\star(L[1])\to\Diff^\star(\hat{L}[1])$ is a derivation along the module morphism $i_\nabla:\Diff^\star(L[1])\to\Diff^\star(\hat{L}[1])$ (cf.~Remark~\ref{rem:derivation_along_morphism}), it is enough to check that $d_G\circ i_\nabla$ vanishes on $\Gamma(L)[1]$ and $\Diff(L[1])$, and indeed this is the case because of~\eqref{eq:i_nabla}, Remark~\ref{rem:d_G} and Lemma~\ref{lem:metric_connection}.
\end{proof}

\begin{proposition}
	\label{prop:flatness}
	The following conditions are equivalent:
	\begin{itemize}
		\item the $DL$-connection $\nabla$ in $E\to M$ is flat,
		\item the immersion $i_\nabla:D^\star(L[1])\to D^\star(\hat{L}[1])$ is a Lie algebra morphism.
	\end{itemize}
\end{proposition}

\begin{proof}
	It reduces to a straightforward computation on generators of the module $D^\star(L[1])$ over the algebra $D^\star(L[1],\bbR_{M})$.
\end{proof}

\begin{remark}
	\label{rem:local_expression_i_nabla}
	Let us keep the same notations of Remark~\ref{rem:local_frames1}.
	The tautological $\Diff L$-connection in $L$ is locally given by:
	\begin{align*}
	\nabla_{\id}\mu&=\mu,&\nabla_{\Delta_i}\mu&=0,
	\intertext{Clearly, if the $\Diff L$-connection in $E$ is locally given by:}
	\nabla_{\id}\xi^A&=\Gamma^A_B\xi^B,&\nabla_{\Delta_i}\xi^A&=\Gamma^{\phantom{i}A}_{iB}\xi^B,
	\end{align*}
	then the $\Diff L$-connection in $E^\ast\otimes L$, obtained by tensor product, is locally given by:
	\begin{align*}
	\nabla_{\id}(\xi^\ast_A\otimes\mu)&=(\delta_A^B-\Gamma_A^B)\xi^\ast_B\otimes\mu,&\nabla_{\Delta_i}(\xi^\ast_A\otimes\mu)&=-\Gamma^{\phantom{i}B}_{iA}\xi^\ast_B\otimes\mu.
	\end{align*}
	Accordingly its extension $\nabla:\Diff L\to\Diff\hat{L}$ is locally given by:
	\begin{align*}
	\nabla_{\id}&=\id-(\delta_B^A-\Gamma_B^A)(\xi^B\otimes\mu^\ast)\Delta_A-\Gamma_A^B\xi_B^\ast\Delta^A,\\
	\nabla_{\Delta_i}&=\Delta_i+\Gamma_{iB}^{\phantom{i}A}(\xi^B\otimes\mu^\ast)\Delta_A-\Gamma_{iA}^{\phantom{i}B}\xi_B^\ast\Delta^A.
	\end{align*}  	
	Hence, the action of the module morphism $i_\nabla:\Diff^\star(L[1])\to\Diff^\star(\hat{L}[1])$ is locally given by
	\begin{align*}
	i_\nabla(f\mu)&=f\mu,\\
	i_\nabla(\id)&=\id+(\Gamma_B^A-\delta_B^A)(\xi^B\otimes\mu^\ast)\Delta_A-\Gamma_A^B\xi_B^\ast\Delta^A,\\
	i_\nabla(\Delta_i)&=\Delta_i+\Gamma_{iB}^{\phantom{i}A}(\xi^B\otimes\mu^\ast)\Delta_A-\Gamma_{iA}^{\phantom{i}B}\xi_B^\ast\Delta^A.
	\end{align*}
\end{remark}

\subsubsection*{The homotopy}
We construct the homotopy in several steps.
First of all, each $\Diff L$-connection $\nabla$ in $E$ determines a degree $0$ graded module isomorphism
\begin{equation*}
\psi_\nabla:
\Diff^\star(\hat{L}[1])\overset{\simeq}{\longrightarrow} S_{{\scriptscriptstyle C^\infty(M)}}(\Gamma(E_L[-1]\oplus E^\ast[1]\oplus E_L[-2]\oplus E^\ast[0]))\underset{{\scriptscriptstyle C^\infty(M)}}{\otimes}\Diff^\star(L[1]),
\end{equation*} 
covering a degree $0$, $C^\infty(\hat{M})$-linear, graded algebra isomorphism
\begin{equation*}
\underline{\smash{\psi_\nabla}}:
\Diff^\star(\hat{L}[1],\bbR_{\hat{M}})\overset{\simeq}{\longrightarrow} S_{{\scriptscriptstyle C^\infty(M)}}(\Gamma(E_L[-1]\oplus E^\ast[1]\oplus E_L[-2]\oplus E^\ast[0]))\underset{{\scriptscriptstyle C^\infty(M)}}{\otimes}\Diff^\star(L[1],\bbR_M),
\end{equation*}
see Appendix~\ref{app:graded_multi-do}, Section~\ref{subsec:psi_nabla}.
Using the same notations of Remark~\ref{rem:local_frames1}, $\psi_\nabla$ is locally given by
\begin{gather*}
\psi_\nabla(f\mu)=f\mu,\quad\psi_\nabla(\xi^A)=\xi^A,\qquad\psi_\nabla(\xi_A^\ast\otimes\mu)=\xi_A^\ast\otimes\mu,\\
\psi_{\nabla}(\nabla_{\id})=\id,\quad \psi_\nabla(\nabla_{\Delta_i})=\Delta_i,\\ \psi_{\nabla}(\Delta_A)=\xi_A^\ast\otimes\mu,\quad \psi_{\nabla}(\Delta^A)=\xi^A.
\end{gather*}
Now we can use $\psi_\nabla$ to define a degree $0$ graded derivation $\weight:\Diff^\star(\hat{L}[1])\to\Diff^\star(\hat{L}[1])$ as follows: $\weight$ is completely determined by the following conditions:	\begin{itemize}
	\item $\psi_\nabla\circ\weight\circ\psi_\nabla^{-1}$ vanishes on $\Diff^\star(L[1])$, and
	\item $\underline{\smash{\psi_\nabla}}\circ\underline{\smash{\weight}}\circ\underline{\smash{\psi_\nabla}}^{-1}$ agrees with the identity map on $\Gamma(E_L[-1]\oplus E^\ast[1]\oplus E_L[-2]\oplus E^\ast[0])$,
\end{itemize}
where $\underline{\smash{\weight}}$ is the symbol of $\weight$.
In particular, $\weight$ measures the tensorial degree of multi-derivations w.r.t.~$E_L[-1]$, $E^\ast[1]$, $E_L[-2]$ and $E^\ast[0]$.
\begin{remark}
	Because of its very definition, $\weight$ is $C^\infty(M)$-linear, and preserves both the arity and the bi-degree, i.e.
	\begin{equation*}
	\weight(\Diff^n(\hat{L}[1])^{(h,k)})\subseteq\Diff^n(\hat{L}[1])^{(h,k)}.
	\end{equation*}
\end{remark}

\begin{remark}
	Derivation $\weight$ is locally given by:
	\begin{gather*}
	\weight(f\mu)=\weight(i_\nabla(\id))=\weight(i_\nabla(\Delta_i))=0,\\
	\weight(\xi^A)=\xi^A,\quad\weight(\xi_A^\ast\otimes\mu)=\xi_A^\ast\otimes\mu,\\
	\weight(\Delta_A)=\Delta_A,\quad\weight(\Delta^A)=\Delta^A.
	\end{gather*}
\end{remark}
We also define a degree $(-1)$ graded derivation $\tilde H_\nabla$ of $\Diff^\star(\hat{L}[1])$ as follows: $\tilde H_\nabla$ is completely determined by the following conditions:
\begin{itemize}
	\item $\psi_\nabla\circ{\tilde H}_\nabla\circ\psi_\nabla^{-1}$ vanishes on $\Diff^\star(L[1])$,
	\item $\underline{\smash{\psi_\nabla}}\circ\underline{\smash{{\tilde H}_\nabla}}\circ\underline{\smash{\psi_\nabla}}^{-1}$ vanishes on $\Gamma(E_L[-1]\oplus E^\ast[1])$, is $C^\infty(\hat{M})$-linear, and maps $\Gamma(E_L[-2]\oplus E^\ast[0])$ to $\Gamma(E_L[-1]\oplus E^\ast[1])$ acting as the desuspension map,
\end{itemize}
where $\underline{\smash{{\tilde H}_\nabla}}$ is the symbol of $\tilde H_\nabla$.

\begin{remark}
	\label{rem:H_nabla}
	From its very definition $\tilde H_\nabla$ is also graded $C^\infty(\hat{M})$-linear, of bi-degree $(1,1)$, so that
	\begin{equation*}
	\tilde H_\nabla(\Diff^n(\hat{L}[1])^{(h,k)})\subseteq\Diff^{n-1}(\hat{L}[1])^{(h+1,k+1)}.
	\end{equation*}
\end{remark}

\begin{remark}
	\label{def:H^tilde}
	Derivation $\tilde H_\nabla$ is locally given by:
	\begin{gather*}
	\tilde H_\nabla(f\mu)= \tilde H_\nabla(i_\nabla(\id))= \tilde H_\nabla(i_\nabla(\Delta_i))=0,\\
	\tilde H_\nabla(\xi^A)=\tilde H_\nabla(\xi_A^\ast\otimes\mu)=0,\quad \tilde H_\nabla(\Delta_A)=\xi_A^\ast\otimes\mu,\quad \tilde H_\nabla(\Delta^A)=\xi^A.
	\end{gather*}
\end{remark}

\begin{lemma}
	\label{lem:1contraction_data_1a}
	The following identities hold:
	\begin{equation}
	\label{eq:1contraction_data_1a}
	p\circ\tilde H_\nabla=0,\qquad \tilde H_\nabla\circ i_\nabla=0,\qquad \tilde H_\nabla^2=0,\qquad [\tilde{H}_\nabla,d_G]=\weight.
	\end{equation}
\end{lemma}

\begin{proof}
	The first three identities follow immediately from the local coordinate expressions for $p$, $i_\nabla$, and $\tilde{H}_\nabla$.
	Moreover, from $d_G\circ i_\nabla=\tilde{H}_\nabla\circ i_\nabla=0$, a straightforward computation in local coordinates shows that the graded derivations $[\widetilde H_\nabla, d_G]$ and $\weight$ agree on generators.
	Hence they coincide.
\end{proof}

\begin{remark}
	\label{rem:eigenspaces_deg}
	Since both $d_G$ and $\tilde{H}_{\nabla}$ commute with $\weight$, the eigenspaces of $\weight$ are invariant under $d_G$ and ${\tilde H}_\nabla$, i.e., for every $k\in\bbN$,
	\begin{align*}
	d_G(\ker(\weight-k\id))&\subseteq\ker(\weight-k\id),\\
	{\tilde H}_\nabla(\ker(\weight-k\id))&\subseteq\ker(\weight-k\id).
	\end{align*}
	Moreover, we have the spectral decomposition of $\weight$, namely
	\begin{equation*}
	\Diff^\star(\hat{L}[1])=\bigoplus_{k\geq 0}\ker(\weight-k\id).
	\end{equation*}
	Actually, $\Diff^\star(\hat{L}[1])=\ker p\oplus\im i_\nabla$, with $\im i_\nabla=\ker\weight$,\ and $\ker p=\bigoplus_{k>0}\ker(\weight-k\id)$.
\end{remark}

\begin{lemma}
	\label{lem:1contraction_data_1b}
	Let $H_\nabla:\Diff^\star(\hat{L}[1])\to\Diff^\star(\hat{L}[1])$ be the degree $(-1)$ graded $C^\infty(M)$-linear map, of bi-degree $(1,1)$, defined by setting:
	\begin{equation*}
	H_\nabla=
	\begin{cases}
	0,&\textnormal{ on }\ker\weight,\\
	-k^{-1}{\tilde H}_\nabla,&\textnormal{ on }\ker(\weight-k\id),\textnormal{ for all }k>0.
	\end{cases}
	\end{equation*}
	Then
	\begin{equation}
	\label{eq:1contraction_data_1b}
	i_\nabla\circ p-\id=[d_G,H_\nabla].
	\end{equation}
	Additionally, $p,i_\nabla$ and $H_\nabla$ satisfy the side conditions $H_\nabla^2=0$, $H_\nabla\circ i_\nabla=0$, $p\circ H_\nabla=0$.
\end{lemma}

\begin{proof}
	It is an immediate consequence of Lemma~\ref{lem:1contraction_data_1a} and Remark~\ref{rem:eigenspaces_deg}.
\end{proof}

The following proposition summarizes the above discussion.

\begin{proposition}
	\label{prop:1contraction_data_1}
	Every  $\Diff L$-connection $\nabla$ in $E$ determines the following set of contraction data:
	\begin{equation}
	\label{eq:1contraction_data_1}
	\begin{tikzpicture}[>= stealth,baseline=(current bounding box.center)]
	\node (u) at (0,0) {$(\Diff^\star(\hat{L}[1]),d_G)$};
	\node (d) at (5,0) {$(\Diff^\star(L[1]),0)$};
	\draw [transform canvas={yshift=-0.5ex},->] (d) to node [below] {\footnotesize $i_\nabla$} (u);
	\draw [transform canvas={yshift=+0.5ex},<-] (d) to node [above] {\footnotesize $p$} (u);
	\draw [->] (u.south west) .. controls +(210:1) and +(150:1) .. node[left=2pt] {\footnotesize $H_\nabla$} (u.north west);
	\end{tikzpicture}
	\end{equation}
	In particular, $p$ is a quasi-isomorphism, so that $H_{CE}(\hat{M},\hat{L},G)\simeq\Diff^\ast(L[1])$, in a canonical way.
\end{proposition}

\subsection{Existence and uniqueness of the liftings: the proofs}
\label{subsec:existence_uniqueness_lifted_graded_Jacobi}

\begin{proof}[Proof of Theorem~\ref{theor:existence_BFV_brackets}]
	Fix a $\Diff L$-connection $\nabla$ in $E$.
	The corresponding $\hat{J}^\nabla$ is constructed by applying Proposition~\ref{prop:SBSO_existence}.
	It is enough to use contraction data~\eqref{eq:1contraction_data_1} for contraction data~\eqref{eq:homotopy_equivalence}, and set $\calF_n:=\bigoplus_{j\geq n}\Diff^\star(\hat{L}[1])^{(i,j)}$, $N=0$, and $\bar{Q}:=G+i_\nabla(J)$.
	In this special case, the necessary and sufficient condition~\eqref{eq:prop:SBSO_existence} is trivially satisfied.
\end{proof}

\begin{proof}[Proof of Proposition~\ref{prop:existence_BFV_brackets}]
	Because of Proposition~\ref{prop:i_nabla}(3), and $\ldsb G,G\rdsb=0$, we just have
	\begin{equation*}
	\ldsb G+i_\nabla\J,G+i_\nabla\J\rdsb=\ldsb i_\nabla\J,i_\nabla\J\rdsb.
	\end{equation*}
	Hence, when $\nabla$ is flat, Proposition~\ref{prop:flatness} guarantees that $\ldsb G+i_\nabla\J,G+i_\nabla\J\rdsb=0$.
	In such case the step-by-step obstruction method does not add any perturbative correction to $G+i_\nabla\J$, and so the output is $\hat{\J}^\nabla=G+i_\nabla\J$.
\end{proof}

\begin{proof}[Proof of Theorem~\ref{theor:uniqueness_BFV_brackets}]
	It follows from Proposition~\ref{prop:SBSO_uniqueness} and Corollary~\ref{cor:SBSO_uniqueness}.
\end{proof}

\begin{proof}[Proof of Proposition~\ref{prop:lifting_CE_cohom}]
	\label{proof:prop:lifting_CE_cohom}
	Let $\J$ and $\hat{\J}$ be Jacobi structures on $L\to M$ and $\hat{L}\to\hat{M}$ respectively.
	Assume that $\hat{\J}$ is a lifting of $\J$, and fix a $\Diff L$-connection $\nabla$ in $E\to M$.
	
	Using the same terminology as in~\cite{Crainic_perturbation}, $\delta:=d_{\hat{\J}}-d_G$ provides a small perturbation of the contraction data~\eqref{eq:1contraction_data_1} determined by $\nabla$.
	Actually, from Remarks~\ref{rem:BFV_brackets} and~\ref{rem:H_nabla}, it follows that
	\begin{align*}
	\delta(\Diff^n(\hat{L}[1])^{(p,q)})&\subseteq\bigoplus_{k\geq 0}\Diff^{n+1}(\hat{L}[1])^{(p+k,q+k)},&
	(\delta H_\nabla)(\Diff^n(\hat{L}[1])^{(p,q)})&\subseteq\bigoplus_{k\geq 1}\Diff^n(\hat{L}[1])^{(p+k,q+k)},
	\end{align*}
	so that $\delta H_\nabla$ is nilpotent and $(\id-\delta H_\nabla)$ is invertible with $(\id-\delta H_\nabla)^{-1}=\sum_{k=0}^\infty(\delta H_\nabla)^k$.
	Hence the Homological Perturbation Lemma (see, e.g., \cite{brown1967twisted,Crainic_perturbation}) applies with the contraction data~\eqref{eq:1contraction_data_1} and their small perturbation $\delta$ as input.
	The output is a new (deformed) set of contraction data
	\begin{equation*}
	\begin{tikzpicture}[>= stealth,baseline=(current bounding box.center)]
	\node (d) at (0,0) {$(\Diff^\star(L[1]),d')$};
	\node (u) at (-5,0) {$(\Diff^\star(\hat{L}[1]),d_{\hat{\J}})$};
	\draw [transform canvas={yshift=-0.5ex},->] (d) to node [below] {\footnotesize $i_\nabla'$} (u);
	\draw [transform canvas={yshift=0.5ex},<-] (d) to node [above] {\footnotesize $p'$} (u);
	\draw [->] (u.south west) .. controls +(210:1) and +(150:1) .. node[left=2pt] {\footnotesize $H_\nabla'$} (u.north west);
	\end{tikzpicture}
	\end{equation*}
	given by
	\begin{equation*}
	i_\nabla'=\sum_{k=0}^\infty(H_\nabla\delta)^ki_\nabla,\quad p'=\sum_{k=0}^\infty p(\delta H_\nabla)^k,\quad H_\nabla'=\sum_{k=0}^\infty H_\nabla(\delta H_\nabla)^k,\quad d'=\sum_{k=0}^\infty p\delta(H_\nabla\delta)^k i_\nabla.
	\end{equation*}
	It follows from Propositions~\ref{prop:p} and~\ref{prop:i_nabla} that $p\delta(H_\nabla\delta)^ki_\nabla=0$, for all $k\geq 1$.
	Moreover, using the same notations as in Remark~\ref{rem:BFV_brackets}, we can write
	\begin{equation*}
	d'\square=p\ldsb \hat{\J}-G,i_\nabla\square\rdsb=p\ldsb\hat{\J}_1,i_\nabla\square\rdsb=\ldsb\J,\square\rdsb=d_{\J}\square,
	\end{equation*}
	for all $\square\in\Diff^\star(L[1])$.
	Hence $d'=d_{\J}$, and $i_\nabla'$ is the desired quasi-isomorphism.
\end{proof}

\section{BRST charges}

\label{sec:BRST-charges}

Let $S$ be a coisotropic submanifold of a Jacobi manifold $(M,L,J)$.
In analogy with the Poisson case~\cite{herbig2007,schatz2009bfv}, we will attach to $S$ an algebraic invariant, the BFV-complex.
The BFV-complex provides a homological resolution of the reduced Gerstenhaber-Jacobi algebra of $S$, and encodes its coisotropic deformations and their local moduli spaces.
Since we are only interested in small deformations of $S$, we can restrict to work within a tubular neighborhood of $S$ in $M$.
Accordingly, we will also need a suitable tubular neighborhood of the restricted line bundle $\ell:=L|_S\to S$ in $L\to M$.
Recall, from~\cite{LOTV}, that a fat tubular neighborhood $(\tau,\underline{\smash{\tau}})$ of $\ell\to S$ in $L\to M$ consists of two layers:
\begin{itemize}
	\item a tubular neighborhood $\underline{\smash{\tau}}:NS\to M$ of $S$ in $M$,
	\item an embedding $\tau:L_{NS}\to L$ of line bundles, over $\underline{\smash{\tau}}:NS\to M$, such that $\tau=\id$ on $L_{NS}|_S\simeq\ell$,
\end{itemize}
where $\pi:NS\to S$ is the normal bundle to $S$ in $M$, and $L_{NS}:=\pi^\ast\ell\to NS$.
By transferring Jacobi structures along a fat tubular neighborhood, we end up with the following local model for a Jacobi manifold $(M,L,J)$ around an arbitrary submanifold $S\subset M$.
\begin{itemize}
	\item The manifold $M$ is modeled on the total space $\calE$ of a vector bundle $\pi:\calE\to S$, and $S$ is identified with the image of the zero section of $\pi$.
	\item The line bundle $L\to M$ is modeled on $\pi^\ast\ell\to\calE$, for some line bundle $\ell\to S$.
\end{itemize}
In this section, working within such local model, we will apply the lifting procedure of the previous section to the case when $E\to M$ is $V\calE \simeq \pi^\ast\calE\to\calE$, the vertical bundle of $\calE\to S$.
In particular, $E\to M = \calE$ admits a tautological section, that we denote by $\Omega_E$, mapping $x \in \calE$ to $(x,x) \in \pi^\ast \calE = \calE \times_S \calE$.

\subsection{Existence and uniqueness of the BRST charges: the statements}
\label{subsec:BRST-charges}

Let $u^i$ be a system of local coordinates on $S$, $\eta^A$ a local frame of $\pi:\calE\to S$, and $\mu$ a local frame of $\ell\to S$.
Denote by $\eta_A^\ast$ the local frame of $\calE^\ast\to S$ dual to $\eta^A$, and by $y_A$ the corresponding fiber-wise linear functions on $\calE$.
Then $\xi^A:=\pi^\ast\eta^A$ is a local frame of $E\to\calE$, and $\xi_A^\ast=\pi^\ast\eta_A^\ast$ is the dual frame of $E^\ast\to\calE$.
Furthermore a local frame of $E_L\to\calE$ is given by $\xi^A\otimes \pi^\ast\mu^\ast=\pi^\ast(\eta^A\otimes\mu^\ast)$, with $\xi_A^\ast\otimes \pi^\ast\mu=\pi^\ast(\eta_A^\ast\otimes\mu)$ the dual local frame of $(E_L)^\ast\to\calE$.

\begin{proposition}
	\label{prop:delta_s}
	Let $s$ be an arbitrary section of $\pi:\calE\to S$.
	The section $\Omega_E[s]:=\Omega_E-\pi^\ast s\in\Gamma(E)$ is a MC element of $(\Gamma(\hat{L}),\{-,-\}_G)$.
	In particular $d[s]:=\{\Omega_E[s],-\}_G$ is a bi-degree $(0,-1)$ homological Hamiltonian derivation of the graded Jacobi manifold $(\hat{M},\hat{L},\{-,-\}_G)$.
\end{proposition}

\begin{proof}
	It is straightforward for ghost/anti-ghost bi-degree reasons.
\end{proof}

\begin{remark}
	\label{rem:local_expression_delta_s}
	The tautological section $\Omega_E$, and an arbitrary $s\in\Gamma(\pi)$  are locally given by $\Omega_E=y_A\xi^A$, and $s=g_A(u^i)\eta^A$.
	Hence $\Omega_E[s]$, and the associated homological derivation $d[s]$ are locally given by
	\begin{equation*}
	\Omega_E[s]=(y_A-g_A(u^i))\xi^A,\quad\textnormal{and}\quad d[s]=(y_A-g_A(u^i))\Delta^A.
	\end{equation*}
\end{remark}

Now, let $\J$ be a Jacobi structure on $L\to M$, and let $\hat{\J}$ be a lifting of $\J$ to $\hat{L}\to\hat{M}$.
Fix an arbitrary $ s\in\Gamma(\pi)$.
In general $\Omega_E[s]$ fails to be a MC element of $(\Gamma(\hat{L}),\{-,-\}_{\hat{\J}})$.
The aim of this section is to find conditions on $s$ so that $\Omega_E [s]$ can be deformed into a suitable MC element of $(\Gamma(\hat{L}),\{-,-\}_{\hat{\J}})$.
The latter will be called an \emph{$s$-BRST charge}.
It turns out that an $s$-BRST charge exists precisely when the image of $s$ is a coisotropic submanifold of $(M,L,J)$.
Now, suppose $S$ is coisotropic itself.
There are two reasons why $s$-BRST charges are interesting.
First of all, as it will be shown in Section~\ref{sec:BFV-complex}, the choice of a $0$-BRST charge represents the second and last step in the construction of the BFV-complex of $S$.
Moreover, as it will be shown in Section~\ref{sec:coisotropic_deformation_problem}, the small coisotropic deformations of $S$ are encoded into its BFV-complex through the BRST charges.
\begin{definition}
	\label{def:BRST-charge}
	An \emph{$s$-BRST charge} wrt $\hat{\J}$ is a MC element $\Omega$ of $(\Gamma(\hat{L}),\{-,-\}_{\hat{\J}})$ having $\Omega_E[s]$ as its bi-degree $(1,0)$ component.
	Explicitly, $\Omega\in\Gamma(\hat{L})^1$, $\{\Omega,\Omega\}_{\hat{\J}}=0$, and $\operatorname{pr}^{(1,0)}\Omega=\Omega_E[s]$.
\end{definition}

\begin{remark}
	Our $s$-BRST charges are analogous to what Sch\"atz calls \emph{normalized MC elements}. In particular, $0$-BRST charges are analogous to Sch\"atz's \emph{BFV-charges}~\cite{schatz2009bfv}. We adopted the terminology ``BRST charge'' because it seems to be more standard in the Physics literature on the subject.
\end{remark}

\begin{remark}
	\label{rem:BRST_charge}
	Assume that $\hat{\J}=\sum_{k=0}^\infty\hat{\J}_k$, a lifting of $\J$ to $\hat{L}\to\hat{M}$, has been decomposed as in Remark~\ref{rem:BFV_brackets}.
	Every $\Omega\in\Gamma(\hat{L})^1$ decomposes as follows
	\begin{equation*}
	\Omega=\sum_{i=0}^\infty\Omega_i,
	\end{equation*}
	with $\Omega_i\in\Gamma(\hat{L})^{(i+1,i)}$.
	Accordingly, an $s$-BRST charge wrt $\hat{\J}$ is a degree $1$ section $\Omega$ of $\hat{L}\to\hat{M}$ such that
	\begin{gather}
	\label{eq:rem:BRST_charge1}
	\Omega_0=\Omega_E[s],\\
	\label{eq:rem:BRST_charge2}
	2d[s]\Omega_h+\sum_{\genfrac{}{}{0pt}{}{i,j< h,k\geq 0}{i+j+k=h}}\{\Omega_i,\Omega_j\}_{\hat\J_k}=0,\quad\textnormal{for all } h\geq 1.
	\end{gather}
\end{remark}

Given $s\in\Gamma(\pi)$, next Theorem~\ref{theor:existence_BRST-charge} shows that, as already anticipated, an $s$-BRST charge wrt $\hat{\J}$ exists precisely when the image of $s$ is coisotropic.
In this case, the $s$-BRST charge is also unique up to isomorphisms (Theorem \ref{theor:uniqueness_BRST-charge}).

\begin{theorem}[Existence]
	\label{theor:existence_BRST-charge}
	Let $\J$ be a Jacobi structure on $L\to M$, let $\hat{\J}$ be a lifting of $\J$ to $\hat{L}\to\hat{M}$, and let $s\in\Gamma(\pi)$.
	Then there exists an $s$-BRST charge wrt $\hat{\J}$ iff  the image of $s$ is coisotropic in $(M,L,\J)$.
\end{theorem}

\begin{remark}
	\label{rem:filtration_Ham}
	Let $\{\calL^\bullet_{\geq n}\}_{n\geq 0}$ be the \emph{finite} decreasing filtration of $\Gamma(\hat{L})^\bullet$ by the graded $C^\infty(\hat{M})$-submodules $\calL^\bullet_{\geq n}$ defined as the sum of those $\Gamma(\hat{L})^{(h,k)}$, with $k\geq n$.
	Then there is a \emph{finite} decreasing filtration $\{\Ham_{\geq n}(\hat{M},\hat{L},\hat{\J})\}_{n\geq 0}$ of the group $\Ham(\hat{M},\hat{L},\hat{\J})$ of Hamiltonian automorphisms of $(\hat{M},\hat{L},\hat{\J})$.
	Namely, for $n\geq 0$, $\Ham_{\geq n}(\hat{M},\hat{L},\hat{\J})$ consists of those $\Phi\in\Ham(\hat{M},\hat{L},\hat{\J})$ such that $\Phi=\Phi_1$ for a smooth path of Hamiltonian automorphisms $\{\Phi_t\}_{t\in I}$ integrating $\{\lambda_t,-\}_{\hat{\J}}$, with $\{\lambda_t\}_{t\in I}\subset\calL^0_{\geq n}$ (cf.~Definitions~\ref{def:Hamiltonian_family} and~\ref{def:Hamiltonian_single}).
	Here, as in what follows, $I$ denotes the closed interval $[0,1]$. 
\end{remark}

\begin{theorem}[Uniqueness]
	\label{theor:uniqueness_BRST-charge}
	Let $\J$ be a Jacobi structure on $L\to M$, and let $\hat{\J}$ be a lifting of $\J$ to $\hat{L}\to\hat{M}$.
	Moreover, let $s\in\Gamma(\pi)$, and let $\Omega$ and $\Omega'$ be $s$-BRST charges wrt $\hat{\J}$.
	Then there exists $\phi\in\Ham_{\geq 2}(\hat{M},\hat{L},\hat{\J})$ such that $\phi^\ast\Omega'=\Omega$.
\end{theorem}

In order to develop the necessary technical tools first, we postpone the proofs of Theorems~\ref{theor:existence_BRST-charge} and~\ref{theor:uniqueness_BRST-charge} to the end of this section.

\subsection{A second relevant set of contraction data}
\label{subsec:second_relevant_contraction_data}

Fix a section $s\in\Gamma(\pi)$ locally given by $s=g_A(u^i)\eta^A$.
In the following, for any vector bundle $F\to S$, we will understand the canonical isomorphism $F\overset{\simeq}{\longrightarrow}(\pi^\ast F)|_{\im s}$, given by $F_x \ni v\mapsto(s(x),v)\in(\pi^\ast F)_{s(x)}$, for all $x\in S$.
\subsubsection*{The projection}
There is a degree $0$ graded module epimorphism $\wp[s]:\Gamma(\hat{L})\to\Gamma( (\wedge\calE_\ell)\otimes\ell)$, covering a degree $0$ graded algebra morphism $\underline{\smash{\wp[s]}}:C^\infty(\hat{M})\to\Gamma(\wedge\calE_\ell)$, completely determined by:
\begin{equation*}
\wp[s]\lambda=\lambda|_{\im s},\qquad \wp[s](e)=e|_{\im s},\qquad \wp[s](\alpha)=0,
\end{equation*}
for all $\lambda\in\Gamma(L)$, $e\in\Gamma(E)$, and $\alpha\in\Gamma(E^\ast\otimes L)$.
This means that $\wp[s]$ is obtained by restricting to $\im s$ and killing the components with non-zero anti-ghost degree.
Locally
\begin{equation*}
\wp[s]\left(f_{\bfA}^{\bfB}(u^i,y_C)\xi^\ast_{\bfB}(\xi^{\bfA}\otimes \pi^\ast\mu^\ast)\otimes\pi^\ast\mu\right)=f_{\bfA}(u^i,g_C(u^i))(\eta^{\bfA}\otimes\mu^\ast)\otimes\mu.
\end{equation*}
\subsubsection*{The immersion}
There is a degree $0$ graded module monomorphism $\iota:\Gamma( (\wedge\calE_\ell)\otimes\ell)\to\Gamma(\hat{L})$, covering a degree $0$ graded algebra morphism $\underline{\smash{\iota}}:\Gamma(\wedge\calE_\ell)\to C^\infty(\hat{M})$, completely determined by:
\begin{equation*}
\iota\lambda'=\pi^\ast\lambda',\qquad \iota\eta=\pi^\ast\eta,
\end{equation*}
for all $\lambda'\in\Gamma(\ell)$, and $\eta\in\Gamma(\calE)$.
Locally
\begin{equation*}
\iota(f_{\bfA}(u^i)(\eta^{\bfA}\otimes\mu^\ast)\otimes\mu)=f_{\bfA}(u^i)(\xi^{\bfA}\otimes \pi^\ast\mu^\ast)\otimes\pi^\ast\mu.
\end{equation*}

It immediately follows from the definitions of $d[s]$, $\wp[s]$ and $\iota$, that
\begin{itemize}
	\item $\wp[s]$ and $\iota$ are differential graded module morphisms between $(\Gamma(\hat{L}),d[s])$ and $(\Gamma((\wedge\calE_\ell)\otimes\ell),0)$, i.e.
	\begin{equation*}
	\wp[s]\circ d[s]=0,\qquad d[s]\circ \iota=0,
	\end{equation*}
	\item $\iota$ is a section of $\wp[s]$, i.e.~$\wp[s]\circ \iota=\id$.
\end{itemize}
Conversely $\iota \circ \wp[s]=\id$ holds only up to a homotopy of differential graded modules that we construct now.
\subsubsection*{The homotopy}
Let $\{\psi^\ast_t[s]\}_{t\in I}$ be the smooth path of bi-degree $(0,0)$ graded module endomorphisms of $\Gamma(\hat{L})$, covering a smooth path $\{\underline{\smash{\psi}}^\ast_t[s]\}_{t\in I}$ of bi-degree $(0,0)$ graded algebra endomorphisms of $C^\infty(\hat{M})$, locally given by
\begin{multline*}
\label{eq:local_psi_t}
\psi_t^\ast[s]\left(f_{\bfA}^{\bfB}(u^i,y_C)\xi^\ast_{\bfB}(\xi^{\bfA}\otimes \pi^\ast\mu^\ast)\otimes\pi^\ast\mu\right)=\\=f_{\bfA}^{\bfB}(u^i,y_C-t(y_C-g_C(u^i)))(1-t)^{|\bfB|}\xi^\ast_{\bfB}(\xi^{\bfA}\otimes\pi^\ast\mu^\ast)\otimes\pi^\ast\mu.
\end{multline*}
We remark that $\{\psi^\ast_t[s]\}_{t\in I}$ connects $\id$ to $\iota\circ \wp[s]$.
Define a smooth path $\{\square_t[s]\}_{t\in I}$ of bi-degree $(0,0)$ graded derivations of $\hat{L}\to\hat{M}$ along $\{\psi_t^\ast[s]\}_{t\in I}$ by setting
\begin{equation*}
\square_t[s]:=\frac{d}{dt}\psi_t^\ast[s],
\end{equation*}
(see~Remark~\ref{rem:derivation_along_morphism} for the meaning of derivation along a module morphism).
There is a smooth path $\{j_t[s]\}_{t\in I}$ of bi-degree $(0,1)$ graded derivations of $\hat{L}\to\hat{M}$, along $\{\psi_t^\ast[s]\}_{t\in I}$, completely determined by
\begin{gather*}
j_t[s](f(u^i,y_C)\pi^\ast\mu)=-(\partial_{y_A}f)(u^i,y_C-t(y_C-g_C(u^i)))\xi_A^\ast\otimes\pi^\ast\mu,\\
j_t[s](\xi^A)=0,\qquad j_t[s](\xi_A^\ast\otimes\pi^\ast\mu)=0.
\end{gather*}
Since $d[s]$ and $\psi_t^\ast[s]$ commute, $\{[d[s],j_t[s]]\}_{t\in I}$ is a smooth path of bi-degree $(0,0)$ graded derivations, along $\{\psi_t^\ast[s]\}_{t\in I}$, as well.
Actually, a straightforward computation in local coordinates shows that $[d[s], j_t[s]]$ and $\square_t$ agree on generators.
Hence they coincide.

Finally, define a bi-degree $(0,1)$ graded $C^\infty(M)$-linear map $h[s]:\Gamma(\hat{L})\to\Gamma(\hat{L})$ by setting
\begin{equation*}
h[s]:=\int_0^1 j_t[s]dt.
\end{equation*}
The map $h[s]$ is a homotopy between the co-chain morphisms $\id,\iota\circ \wp[s]:(\Gamma(\hat{L}),d[s])\to(\Gamma(\hat{L}),d[s])$.
Indeed
\begin{align*}
\iota\circ \wp[s]-\id=\int_0^1\frac{d}{dt}\psi_t^\ast[s]dt=\int_0^1\square_t[s]dt=\int_0^1[d[s],j_t[s]]dt=[d[s],h[s]].
\end{align*}
In addition $\iota$, $\wp[s]$ and $h[s]$ satisfy the side conditions $h[s]^2=0$, $h[s]\circ\iota=0$, $\wp[s]\circ h[s]=0$.

The above discussion is summarized in the following.

\begin{proposition}
	\label{prop:2contraction_data}
	Every section $s\in\Gamma(\pi)$ determines a set of contraction data
	\begin{equation}
	\label{eq:2contraction_data_1}
	\begin{tikzpicture}[>= stealth,baseline=(current bounding box.center)]
	\node (u) at (0,0) {$(\Gamma(\hat{L}),d[s])$};
	\node (d) at (5,0) {$(\Gamma((\wedge\calE_\ell)\otimes\ell),0)$};
	\draw [transform canvas={yshift=-0.5ex},->] (d) to node [below] {\footnotesize $\iota$} (u);
	\draw [transform canvas={yshift=0.5ex},<-] (d) to node [above] {\footnotesize $\wp[s]$} (u);
	\draw [->] (u.south west) .. controls +(210:1) and +(150:1) .. node[left=2pt] {\footnotesize $h[s]$} (u.north west);
	\end{tikzpicture}
	\end{equation}
	In particular $\wp[s]$ is a quasi-isomorphism, and $H^\bullet(\Gamma(\hat{L}),d[s])\simeq\Gamma((\wedge\calE_\ell)\otimes\ell)$ in a canonical way.
\end{proposition}

\begin{remark}
	\label{rem:2contraction__data}
	All the above constructions and Proposition~\ref{prop:2contraction_data} hold true even after replacing the section $s\in\Gamma(\pi)$ with a smooth path $\{s_\tau\}_{\tau\in I}$ in $\Gamma(\pi)$.
	The obvious details are left to the reader.
\end{remark}

\subsection{Existence and uniqueness of the BRST charges: the proofs}
\label{subsec:existence_uniqueness_BRST-charges}

\begin{lemma}
	\label{lem:BRST_coisotropic}
	Let $\J$ be a Jacobi structure on $L\to M$, and $s\in\Gamma(\pi)$.
	For any lifting $\hat{\J}$ of $\J$ to $\hat{L}\to\hat{M}$, we have that  $\{\Omega_E[s],\Omega_E[s]\}_{\hat{\J}}\in\ker\wp[s]$ iff the image of $s$ is coisotropic in $(M,L,J)$.
\end{lemma}

\begin{proof}
	Since $\hat{\J}$ is a bi-derivation with $p(\hat{\J})=\J$, it is straightforward to check that, locally,
	\begin{equation*}
	\wp[s]\vphantom{\xi^A}\{\Omega_E[s],\Omega_E[s]\}_{\hat{\J}}=\left.(\eta^B\otimes\mu^\ast)(\eta^A\otimes\mu^\ast)\{(y_A-g_A(u^i))\mu,(y_B-g_B(u^i))\mu\}_{\J}\right|_{\im s}.
	\end{equation*}
	So $\wp[s]\{\Omega_E[s],\Omega_E[s]\}_{\hat{\J}}=0$ iff $\{(y_A-g_A(u^i))\mu,(y_B-g_B(u^i))\mu\}_{\J}$ vanishes on $\im s$, for all $A$ and $B$.
	The last condition means exactly that $\im s$ is coisotropic in $(M,L,\J)$.
\end{proof}

\begin{proof}[Proof of Theorem~\ref{theor:existence_BRST-charge} (resp.~Theorem~\ref{theor:uniqueness_BRST-charge})]
	It follows immediately as a special case of Proposition~\ref{prop:SBSO_existence} (resp.~Proposition~\ref{prop:SBSO_uniqueness}).
	It will be enough to use the contraction data~\eqref{eq:2contraction_data_1} for the contraction data~\eqref{eq:homotopy_equivalence}, and set $\calF_n:=\calL_{\geq n+1}$, $N=-1$, and $\bar{Q}:=\Omega_E[s]$.
	Indeed, in this case, from Lemma~\ref{lem:BRST_coisotropic} the necessary and sufficient condition~\eqref{eq:prop:SBSO_existence} coincides with $\im(s)$ being coisotropic.
\end{proof}

\section{The BFV-complex of a coisotropic submanifold}

\label{sec:BFV-complex}

Let $(M,L,\J)$ be a Jacobi manifold, and let $S\subset M$ be a coisotropic submanifold.
Recall that $\pi:NS\to S$ denotes the normal bundle to $S$ in $M$, $\ell:=L|_S\to S$ is the restricted line bundle, and we have set $L_{NS}:=\pi^\ast\ell\to NS$.
We will use a fat tubular neighborhood $(\tau,\underline{\smash{\tau}})$ of $\ell\to S$ in $L\to M$ to identify $S$ with the image of the zero section of $\pi$, and replace the Jacobi manifold $(M,L,J)$ with its local model $(NS,L_{NS},\tau^\ast J)$ around $S$.
We will then use the lifting procedure of Section~\ref{sec:lifting_Jacobi_structures}, and the results of Section~\ref{sec:BRST-charges} with the r\^ole of $M=\calE\to S$ and $E\to M=\calE$ being now played by, respectively, the normal bundle $\pi:NS\to S$ and the vertical bundle $V(NS):=\pi^\ast(NS)\to NS$.
\begin{definition}
	\label{def:BFV_complex}
	A \emph{BFV-complex} (attached to $S$ via the fat tubular neighborhood $(\tau,\underline{\smash{\tau}})$) is a differential graded Lie algebra $(\Gamma(\hat{L}),\{-,-\}_{\BFV},d_{\BFV})$ such that:
	\begin{itemize}
		\item $\{-,-\}_{\BFV}=\hat{\J}$, for some lifting $\hat{\J}$ of $\J$ to a graded Jacobi structure on $\hat{L}\to\hat{M}$,
		\item $d_{\BFV}=\{\Omega_{\BRST},-\}_{\BFV}$, where $\Omega_{\BRST}$ is some $0$-BRST charge wrt $\hat{\J}$.
	\end{itemize}
\end{definition}
\begin{remark}
	\label{rem:BFV_complex}
	In more geometric terms, a BFV-complex can be seen, in particular, as the graded Jacobi manifold $(\hat{M},\hat{L},\hat{\J}\equiv\{-,-\}_{\BFV})$ further equipped with the homological Hamiltonian derivation $d_{\BFV}$.
\end{remark}
This section aims at showing that in the Jacobi setting, as already in the Poisson setting~\cite{schatz2010invariance,schatz2011moduli}, the BFV-complex is actually independent, to some extent, and up to isomorphisms, of the fat tubular neighborhood, it is a homological resolution of the reduced Gerstenhaber-Jacobi algebra of $S$, and encodes the moduli space of formal coisotropic deformations of $S$ under Hamiltonian equivalence.

\subsection{Gauge invariance of the BFV-complex}
\label{subsec:gauge_invariance_BFV}

Let $(M,L,\J)$ be a Jacobi manifold, and let $S\subset M$ be a coisotropic submanifold.
The BFV-complex of $S$ is actually independent of the choice of the (fat) tubular neighborhood, at least around $S$, as pointed out by the following.

\begin{theorem}
	\label{theor:gauge_invariance_BFV_complex}
	Let $(\tau_0,\underline{\smash{\tau}}_0)$ and $(\tau_1,\underline{\smash{\tau}}_1)$ be fat tubular neighborhoods of $\ell\to S$ in $L\to M$, and set $\J^0:=\tau_0^\ast\J$ and $\J^1:=\tau_1^\ast\J$.
	Pick liftings $\hat{\J}^i$ of $\J^i$ to $\hat{L}\to\hat{M}$, and let $\Omega^i_{\BRST}$ be a $0$-BRST charge wrt $\hat\J^i$, with $i=0,1$.
	Then there exist open neighborhoods $U_0$ and $U_1$ of $S$ in $NS$, and a degree $0$ graded Jacobi bundle isomorphism $\phi:(\hat M,\hat L,\hat\J^0)|_{U_0}\longrightarrow (\hat M,\hat L,\hat\J^1)|_{U_1}$, such that $\phi^\ast(\Omega^1_{\BRST})=\Omega^0_{\BRST}$, and a fortiori $\phi^\ast d_{\BFV}^1=d_{\BFV}^0$.
\end{theorem}

\begin{proof}
	The main idea of the proof is to show the existence of open neighborhoods $U_0$ and $U_1$ of $S$ in $NS$, and a bi-degree $(0,0)$ graded line bundle isomorphism $\phi$ from $\hat L|_{U_0}\to\hat M|_{U_0}$ to $\hat L|_{U_1}\to\hat M|_{U_1}$, such that $\phi^\ast\hat{\J}^1$ is a lifting of $\J^0$ to $\hat L|_{U_0}\to\hat M|_{U_0}$, and $\phi^\ast(\Omega^1_{\BRST})$ is a $0$-BRST charge wrt $\phi^\ast\hat{\J}^1$.
	After doing this, the proof will be completed using Theorems~\ref{theor:uniqueness_BFV_brackets} and~\ref{theor:uniqueness_BRST-charge}.
	
	The standard uniqueness, up to isotopy, of tubular neighborhoods (cf.~\cite[Chapter 4, Theorem 5.3]{hirsch}) can be adapted to fat tubular neighborhoods (cf.~\cite[Lemma 3.20]{LOTV}).
	Accordingly it will be enough to consider the following two special cases:
	\begin{enumerate}
		\item $\tau_1\circ F=\tau_0$, for some automorphism $F$ of the line bundle $L_{NS}\to NS$, covering an automorphism $\underline{\smash{F}}$ of the normal bundle $NS\overset{\pi}{\to} S$, such that $F=\id$ on $L_{NS}|_S\simeq\ell$,
		\item $\tau_0=\calT_0$ and $\tau_1=\calT_1$, for some smooth path $\{(\calT_t,\underline{\smash{\calT}}_t)\}_{t\in I}$ of fat tubular neighborhoods of $\ell\to S$ in $L\to M$.
	\end{enumerate}
	
	{\sc First case.}
	Let $\underline{F}^\dagger:N^\ast S\to N^\ast S$ be the inverse of the transpose of the vector bundle automorphism $\underline{F}:NS\to NS$.
	There is a bi-degree $(0,0)$ automorphism $\calF$ of the graded line bundle $\hat{L}\to\hat{M}$ uniquely determined by
	\begin{equation}
	\label{eq:proof:gauge_invariance1}
	\calF^\ast\lambda=F^\ast\lambda,\quad\calF^\ast(\pi^\ast\eta)=\pi^\ast(\underline{\smash{F}}^\ast\eta),\quad\calF^\ast((\pi^\ast\alpha)\otimes\lambda)=\pi^\ast((\underline{\smash{F^\dagger}})^\ast\alpha)\otimes F^\ast\lambda,
	\end{equation}
	for all $\lambda\in\Gamma(\hat{L})^{(0,0)}=\Gamma(L_{NS})$, $\eta\in\Gamma(NS)$, and $\alpha\in\Gamma(N^\ast S)$.
	By its very construction, $\calF$ satisfies:
	\begin{equation*}
	\calF^\ast\Omega_{NS}=\Omega_{NS},\quad \calF^\ast G=G,\quad p\circ\calF^\ast=F^\ast\circ p.
	\end{equation*}
	It follows that $\calF^\ast\hat{\J}^1$ is a lifting of $\J^0$ to $\hat{L}\to\hat{M}$, and $\calF^\ast\Omega_{\BRST}^1$ is a $0$-BRST charge wrt $\calF^\ast\hat{\J}^1$.
	
	{\sc Second case.}
	We can find an open neighborhood $V$ of $S$ in $NS$, and a smooth path $\{(F_t,\underline{\smash{F}}_t)\}_{t\in I}$ of line bundle embeddings of $L_{NS}|_V\to V$ into $L_{NS}\to NS$ such that
	\begin{itemize}
		\item $\calT_0=\calT_t\circ F_t$, so that, in particular, $F_0=\id$ on $L_{NS}|_V$,
		\item $F_t$ agrees with the identity map on $L_{NS}|_S\simeq\ell$.
	\end{itemize}
	Consequently, $\J^t:=\calT_t^\ast\J$ is a Jacobi structure, and the image of the zero section of $\pi:NS\to S$ is coisotropic wrt $\J^t$.
	Additionally, $(F_t)_\ast\J^0=\J^t$ and $\underline{\smash{F}}_t(\im 0)=\im 0$, for all $t\in I$.
	Hence, in view of Proposition~\ref{prop:main_result}, $\{F_t\}_{t\in I}$ can be lifted to a smooth path $\{\calF_t\}_{t\in I}$ of bi-degree $(0,0)$ graded line bundle embeddings of $\hat{L}|_V\to\hat{M}|_V$ into $\hat{L}\to\hat{M}$ such that, for all $t\in I$,
	\begin{itemize}
		\item $(\calF_t)_\ast\hat{\J}^0$ is a lifting of $\J^t$ to $\hat{L}|_{F_t(V)}\to\hat{M}|_{F_t(V)}$,
		\item $(\calF_t)_\ast\Omega_{\BRST}^0$ is a $0$-BRST charge wrt $(\calF_t)_\ast\hat{\J}^0$.
	\end{itemize}
	In particular, $U_0:=V$ and $U_1:=F_1(V)$ are open neighborhoods of $S$ in $NS$, and $\phi:=\calF_1$ is a bi-degree $(0,0)$ graded line bundle isomorphism from $\hat{L}|_{U_0}\to\hat{M}|_{U_0}$ to $\hat{L}|_{U_1}\to\hat{M}|_{U_1}$ with all the desired properties.
\end{proof}

\subsection{The BFV-complex and the homological Jacobi reduction of a coisotropic submanifold}
\label{subsec:BFV_homological_resolution}

Let $\J$ be a Jacobi structure on the line bundle $L\to M$, and let $S\subset M$ be a coisotropic submanifold wrt $\J$.
According to~\cite[Prop.~3.6]{LOTV}, this means that $((N_\ell S)^\ast,\ell)$ is a Jacobi subalgebroid of $(J^1L,L)$.
Set $\frakg(S):=\Gamma(\wedge(N_\ell S)\otimes\ell)$, and denote by $d_{dR}:\frakg(S)\to\frakg(S)$ the de Rham differential of the Jacobi algebroid $((N_\ell S)^\ast,\ell)$.
Differential $d_{dR}$ is completely determined by
\begin{equation*}
d_{dR}\circ P=P\circ d_{\J},
\end{equation*}
where the degree $0$ graded module epimorphism  $P:\Diff^\star(L[1])\to\frakg(S)[1]$ is the canonical projection defined by setting $P(\lambda)=\lambda|_S$, $\langle P(\square),(df\otimes\lambda)|_S\rangle=\square(f\lambda)|_S$, for all $\lambda\in\Gamma(L)$, and $f\in I_S$.
In the following we will understand the module isomorphism $\Gamma(L_{\textnormal{red}})\overset{\simeq}{\longrightarrow}H^0(\frakg(S),d_{dR})$, $\lambda+\Gamma_S\longmapsto[\lambda|_S]$, introduced in Section~\ref{sec:Jacobi_reduction}.

\begin{proposition}
	\label{prop:BFV_homological_resolution}
	For every BFV-complex $(\Gamma(\hat{L}),\{-,-\}_{\BFV},d_{\BFV})$ of $S$, its cohomology is canonically isomorphic to the de Rham cohomology of the Jacobi algebroid of $S$
	\begin{equation*}
	H^\bullet(\Gamma(\hat{L}),d_{\BFV})\simeq H^\bullet(\frakg(S),d_{dR}).
	\end{equation*}
	Specifically, from Section~\ref{subsec:second_relevant_contraction_data}, the map $\wp[0]$ is a quasi-isomorphism from $(\Gamma(\hat{L}),d_{\BFV})$ to $(\frakg(S),d_{dR})$.
\end{proposition}

\begin{proof}
	Let $(\Gamma(\hat{L}),\{-,-\}_{\BFV},d_{\BFV})$ be a BFV-complex of $S$.
	Assume that this latter has been constructed by choosing a fat tubular neighborhood $(\tau,\underline{\smash{\tau}})$ of $\ell\to S$ into $L\to M$, a lifting of $\J$ to a Jacobi structure $\hat{\J}=\sum_{k=0}^\infty\hat{\J}_k$ on $\hat{L}\to\hat{M}$, and a $0$-BRST charge $\Omega_{\BRST}=\sum_{k=0}^\infty\Omega_k$ wrt $\hat{\J}$.
	
	Using again the terminology of~\cite{Crainic_perturbation}, $\delta:=d_{\BFV}-d[0]$ is a small perturbation of the contraction data~\eqref{eq:2contraction_data_1}, with $s=0$.
	Indeed
	\begin{align}
	\label{eq:prop:lifting_CE_cohom1}
	\delta\in\bigoplus_{k\geq 0}(\Diff \hat{L})^{(k+1,k)},\qquad
	\delta h[0]\in\bigoplus_{k\geq 1}\End(\hat{L},\hat{L})^{(k,k)},
	\end{align}
	so that $\delta h[0]$ is nilpotent and $\id-\delta h[0]$ is invertible with $(\id-\delta h[0])^{-1}=\sum_{k=0}^\infty(\delta h[0])^k$.
	Hence the Homological Perturbation Lemma~\cite{brown1967twisted,Crainic_perturbation} can be applied taking as input the contraction data~\eqref{eq:2contraction_data_1} and its small perturbation $\delta$.
	The resulting output is given by a new deformed set of contraction data
	\begin{equation*}
	\begin{tikzpicture}[>= stealth,baseline=(current bounding box.center)]
	\node (u) at (0,0) {$(\Gamma(\hat{L}),d_{\BFV})$};
	\node (d) at (5,0) {$(\frakg(S),d')$};
	\draw [transform canvas={yshift=-0.5ex},->] (d) to node [below] {\footnotesize $\iota'$} (u);
	\draw [transform canvas={yshift=0.5ex},<-] (d) to node [above] {\footnotesize $\wp[0]'$} (u);
	\draw [->] (u.south west) .. controls +(210:1) and +(150:1) .. node[left=2pt] {\footnotesize $h[0]'$} (u.north west);
	\end{tikzpicture}.
	\end{equation*}
	From the explicit formulas for $d',\wp[0]',\iota',h[0]'$, and the very definition of $\wp[0]$, it follows that
	\begin{equation*}
	\wp[0]'=\wp[0],\qquad d'=\wp[0]\delta\iota.
	\end{equation*}
	Now it remains to prove that $d'=d_{dR}$.
	Both $d_{dR}$ and $d'$ are homological derivations of the graded module $\frakg(S)$, hence it is enough to check that they coincide on (local) generators, i.e.~on:
	\begin{enumerate}
		\item arbitrary sections $\lambda$ of $\ell\to S$,
		\item elements of a local frame $\eta^A$ of $NS\to S$.
	\end{enumerate}
	Since $\{\Omega_E,-\}_{\hat{\J}_1}+\{\Omega_1,-\}_G$ is the $(1,0)$ bi-degree component of $\delta$, 
	it follows that
	\begin{multline*}
	d'\lambda=(\wp[0]\delta\iota)\lambda=(\wp[0]\delta)(\pi^\ast\lambda)=\left.\{\pi^\ast\lambda,\Omega_E\}_{\hat{\J}_1}\right|_S=\\=\left.(\eta^A\otimes\mu^\ast)\{\pi^\ast\lambda,y_A\mu\}_{\J}\right|_S=(P\circ d_{\J})(\pi^\ast\lambda)=d_{dR}\lambda.
	\end{multline*}
	for all $\lambda\in\Gamma(\ell)$.
	Moreover,
	\begin{equation}
	\label{eq:proof:BFV_homological_resolution1}
	d'\eta^A=(\wp[0]\delta\iota)\eta^A=(\wp[0]\delta)\xi^A=\left.\left(\{\xi^A,\Omega_E\}_{\hat{\J}_1}+\{\xi^A,\Omega_1\}_G\right)\right|_S.
	\end{equation}
	From~\eqref{eq:rem:BRST_charge2}, with $s=0$ and $h=1$, it follows that, locally,
	\begin{equation}
	\label{eq:proof:BFV_homological_resolution2}
	\left.\Delta_{y_A}\{\Omega_1,\Omega_E\}_{\hat{\J}_1}\right|_S=-2\left.\{\Omega_1,\xi^A\}_G\right|_S.
	\end{equation}
	Finally, plugging~\eqref{eq:proof:BFV_homological_resolution2} into~\eqref{eq:proof:BFV_homological_resolution1}, we get
	\begin{equation*}
	d'\eta^A=\frac{1}{2}\left.\ldsb\hat{\J}_1,\Delta_{y_A}\rdsb(\Omega_E,\Omega_E)\right|_S=\frac{1}{2}\eta^A\eta^B\left.\ldsb\J,\Delta_{y_A}\rdsb(y_A\mu,y_B\mu)\right|_S=(P\circ d_{\J})\Delta_{y_A}=d_{dR}\eta_A.\qedhere
	\end{equation*}
\end{proof}

It is now straightforward to see that each BFV-complex of $S$ provides a homological resolution of the reduced Gerstenhaber-Jacobi algebra of $S$.

\begin{corollary}
	\label{cor:BFV_homological_resolution}
	The degree $0$ graded module isomorphism $\wp[0]_\ast:H^\bullet(\Gamma(\hat{L}),d_{\BFV})\to H^\bullet(\frakg(S),d_{dR})$ intertwines, in degree $0$, the bracket induced by $\{-,-\}_{\BFV}$ on $H^0(\Gamma(\hat{L}),d_{\BFV})$ and the reduced Jacobi bracket $\{-,-\}_{\textnormal{red}}$ on $H^0(\frakg(S),d_{dR})\simeq\Gamma(L_{\textnormal{red}})$.
\end{corollary}

\begin{proof}
	Pick arbitrary $d_{\BFV}$-closed degree $0$ sections $\lambda_i=\sum_{k=1}^\infty\lambda_i^k\in\Gamma(\hat{L})^0$, with $\lambda_i^k\in\Gamma(\hat{L})^{(k,k)}$, for all $k\in\bbN$, and $i=1,2$.
	From the construction of $\wp:=\wp[0]$ and the BFV-complex, it follows that
	\begin{equation*}
	\wp_\ast\left[\{\lambda_1,\lambda_2\}_{\BFV}\right]\!=\!\left[\wp\left(\{\lambda_1,\lambda_2\}_{\hat{\J}}\right)\right]\!=\!\left[\left.\{\lambda^0_1,\lambda^0_2\}_{\J}\right|_S\right]\!=\!\left\{\left[\left.\lambda^0_1\right|_S\right]\!,\!\left[\left.\lambda^0_2\right|_S\right]\right\}_{\textnormal{red}}\!=\!\left\{\wp_\ast[\lambda_1],\wp_\ast[\lambda_2]\right\}_{\textnormal{red}}.\qedhere
	\end{equation*}	
\end{proof}

Hence the reduced Gerstenhaber-Jacobi algebra of a coisotropic submanifold $S$ admits two different homological resolutions: one provided by the BFV-complex and another one given by the $L_\infty$-algebra.
In fact, as shown in the next subsection, these two resolutions are strictly related.

\subsection{The BFV-complex and the \texorpdfstring{$L_\infty$}{L-infinity}-algebra of a coisotropic submanifold}
\label{subsec:L_infty_qi}

Let $\J$ be a Jacobi structure on a line bundle $L\to M$, and let $S\subset M$ be a coisotropic submanifold of $(M,L,\J)$.
Additionally, let $(\Gamma(\hat{L}),\{-,-\}_{\BFV},d_{\BFV})$ be a BFV-complex associated to $S$ via the choice of:
\begin{itemize}
	\item a fat tubular neighborhood $(\tau,\underline{\tau})$ of $\ell\to S$ into $L\to M$,
	\item a lifting of $\J$ to a Jacobi structure $\hat{\J}$ on $\hat{L}\to\hat{M}$,
	\item a $0$-BRST charge $\Omega_{\BRST}$ wrt $\hat{\J}$.
\end{itemize}
In view of the proof of Proposition~\ref{prop:BFV_homological_resolution}, after a choice of $(\tau,\underline{\tau})$ and $(\Gamma(\hat{L}),\{-,-\}_{\BFV},d_{\BFV})$, the corresponding set of contraction data~\eqref{eq:2contraction_data_1}, with $s=0$, gets deformed into a new set of contraction data
\begin{equation}
\label{eq:deformed_second_contraction_data}
\begin{tikzpicture}[>= stealth,baseline=(current bounding box.center)]
\node (u) at (0,0) {$(\Gamma(\hat{L}),d_{\BFV})$};
\node (d) at (5,0) {$(\frakg(S),d_{dR})$};
\draw [transform canvas={yshift=-0.5ex},->] (d) to node [below] {\footnotesize $\iota'$} (u);
\draw [transform canvas={yshift=0.5ex},<-] (d) to node [above] {\footnotesize $\wp[0]$} (u);
\draw [->] (u.south west) .. controls +(210:1) and +(150:1) .. node[left=2pt] {\footnotesize $h[0]'$} (u.north west);
\end{tikzpicture},
\end{equation}
where $\iota'=\sum_{k=0}^\infty(h[0]\delta)^k\iota$ and $h[0]'=\sum_{k=0}^\infty h[0](\delta h[0])^k$, with $\delta=d_{\BFV}-d[0]$.

As described in~\cite[Section~3.3]{LOTV}, a fat tubular neighborhood $(\tau,\underline{\tau})$ allows to get a right inverse of the canonical projection $P$ as the unique degree $0$ graded module morphism $I:\frakg(S)[1]\to\Diff^\star(L[1])$ such that $I(\lambda)=\pi^\ast\lambda$, and $I(\eta)(f\otimes\lambda)=(\eta^v f)\otimes\lambda$, for all $\lambda\in\Gamma(\ell)$, $\eta\in\Gamma(NS)$, and $f\in C^\infty(NS)$.
Here $\eta^v$ is the vertical lift of $\eta$: the unique vertical vector field on $NS$ which is constant along the fibers and agrees with $\eta$ along $S$.
The quadruple $(\Diff^\star(L[1]),P,I,\J)$ is a set of $V$-data, as defined in~\cite[Section 1.2]{fregier2014simultaneous}, so that, according to Voronov~\cite{Voronov2005higher1}, there is a (flat) $L_\infty$-algebra structure $\{\frakm_k\}_{k\geq 1}$ on $\frakg(S)$ given by higher derived brackets 
\begin{equation*}
\frakm_k(g_1,\ldots,g_k)=(-)^\sharp P\ldsb\ldsb\ldots\ldsb\J,Ig_1\rdsb,\ldots\rdsb,Ig_k\rdsb,
\end{equation*}
where $(-)^\sharp$ denotes a certain sign coming from d\'ecalage isomorphism (cf., e.g., \cite[Section 1]{fiorenza2007structures}).
It is easy to see that $\frakm_1$ coincides with the de Rham differential of the Jacobi algebroid $((N_\ell S)^\ast,\ell)$ associated to $S$.
Hence, from Section~\ref{sec:Jacobi_reduction}, it follows that the $L_\infty$-algebra is a homological resolution of the reduced Gerstenhaber-Jacobi algebra of $S$.

The following theorem constructs, by homotopy transfer (cf., e.g., \cite[Section 10.3]{loday2012algebraic}), an $L_\infty$-quasi-isomorphism between the two homological resolutions of the reduced Gerstenhaber-Jacobi algebra of $S$, i.e.~the BFV-complex and the $L_\infty$-algebra.
In this way we extend, from the Poisson to the Jacobi case, a result by Sch\"atz~\cite{schatz2009bfv}.
\begin{theorem}
	\label{theor:L_infty_qi}
	For every BFV-complex $(\Gamma(\hat{L}),\{-,-\}_{\BFV},d_{\BFV})$ of $S$ there exists an $L_\infty$-quasi-isomorphism
	\begin{equation*}
	\iota'_\infty:(\frakg(S),\{\frakm_k\}_{k\geq 1})\longrightarrow(\Gamma(\hat{L}),\{-,-\}_{\BFV},d_{\BFV}).
	\end{equation*}
\end{theorem}

\begin{proof}
	The proof is an adaptation of the proof of the analogous result in \cite[Theorem 5]{schatz2009bfv} and we omit it.
\end{proof}

\begin{theorem}
	\label{theor:BFV-complex_formal_moduli}
	Let $(M,L,J \equiv \{-,-\})$ be a Jacobi manifold, and let $S\subset M$ be a coisotropic submanifold.
	The BFV-complex of $S$ controls the formal coisotropic deformation problem of $S$ under Hamiltonian equivalence.
	Indeed there exists a 1--1 correspondence between the moduli space of formal coisotropic deformations of $S$, under Hamiltonian equivalence, and the moduli space of formal MC elements of the BFV-complex, under gauge equivalence.
\end{theorem}

\begin{proof}
	The $L_\infty$-algebra of $S$ controls the formal coisotropic deformation problem of $S$ under Hamiltonian equivalence, (see~\cite[Section 4.4]{LOTV}).
	Furthermore two $L_\infty$-quasi-isomorphic $L_\infty$-algebras control the same formal coisotropic deformation problem (see, e.g.,~\cite[Section 7]{doubek2007deformation}).
	Hence the statement is a corollary of Theorem~\ref{theor:L_infty_qi}.
\end{proof}

In the next section, we will show that the BFV-complex controls the coisotropic deformation problem also at the non formal level, under both Hamiltonian and Jacobi equivalence.

\section{The BFV-complex and the coisotropic deformation problem}
\label{sec:coisotropic_deformation_problem}

In this section we will show that the BFV-complex associated, via a fat tubular neighborhood, to a coisotropic submanifold $S$ of a Jacobi manifold encodes the whole information about the small coisotropic deformations of $S$ lying within the tubular neighborhood, and their moduli spaces under Hamiltonian and Jacobi equivalence.

Let $(M,L,\J)$ be a Jacobi manifold, and let $S\subset M$ be a coisotropic submanifold.
Let $(\Gamma(\hat{L}),\hat{\J}=\{-,-\}_{\BFV},d_{\BFV})$ be a BFV-complex, with $d_{\BFV}=\{\Omega_{\BRST},-\}_{\BFV}$, attached to $S$ via a fat tubular neighborhood $(\tau,\underline{\smash{\tau}})$.
Recall that, in particular, $(\tau,\underline{\smash{\tau}})$ is used to identify $S$ with the image of the zero section of the normal bundle $\pi:NS\to S$, introduce $L_{NS}:=\pi^\ast (L|_S)\to NS$ and $E:=\pi^\ast(NS)\to NS$, and replace the Jacobi manifold $(M,L,\J)$ with its local model $(NS,L_{NS},\tau^\ast\J)$ around $S$.
Let us fix some notation which will be used in the following.
We will denote by $C(L,\J)$ the set of those $s\in\Gamma(\pi)$ whose image is a coisotropic submanifold of $(M,L,\J)$.
The elements of $C(L,\J)$, called \emph{coisotropic sections}, can be seen as the coisotropic submanifolds, lying within the tubular neighborhood, which are small deformations of $S$.
We will denote by $\BRST(\hat{L},\hat{\J})$ the set of those $\Omega\in\Gamma(\hat{L})$ which are $s$-BRST charges wrt $\hat{\J}$ for some arbitrary $s\in\Gamma(\pi)$ (cf.~Definition~\ref{def:BRST-charge}).
The elements of $\BRST(\hat{L},\hat{\J})$ will be simply called \emph{BRST charges}.

\begin{proposition}
	\label{prop:coisotropic_deformation_space} \ 
	\newline\noindent
	1) The space $\BRST(\hat{L},\hat{\J})$ is invariant under the natural action of $\Ham_{\geq 2}(\hat{M},\hat{L},\hat{\J})$ on $\Gamma(\hat{L})$ (see Remark~\ref{rem:filtration_Ham} for the meaning of $\Ham_{\geq 2}(\hat{M},\hat{L},\hat{\J})$).
	\newline\noindent
	2) For any $\Omega\in\BRST(\hat{L},\hat{\J})$, there is a unique $s_\Omega\in C(L,\J)$, such that $\Omega$ is an $s_\Omega$-BRST charge.
	Section $s_\Omega$ is implicitly determined by the following relation
	\begin{equation*}
	\im(s_\Omega)=\textnormal{``zero locus of $\operatorname{pr}^{(1,0)}\Omega$''}.
	\end{equation*}
	3) There is a 1--1 correspondence between $\BRST(\hat{L},\hat{\J})/\Ham_{\geq 2}(\hat{M},\hat{L},\hat{\J})$ and $C(L,\J)$ mapping $\Ham_{\geq 2}(\hat{M},\hat{L},\hat{\J}).\Omega$ to $s_\Omega$.
\end{proposition}

\begin{proof}
	It is a straightforward consequence of Definition~\ref{def:BRST-charge}, and Theorems~\ref{theor:existence_BRST-charge} and~\ref{theor:uniqueness_BRST-charge}.
\end{proof}

Now we introduce the notion of geometric MC element of the BFV-complex by slightly adapting the analogous one given by Sch\"atz in the Poisson case~\cite[Sec.~3.4]{schatz2011moduli}.

\begin{proposition}
	\label{prop:geometric_MC}
	Let $\Omega$ be a MC element of $(\Gamma(\hat{L}),\{-,-\}_{\BFV},d_{\BFV})$.
	For any $s\in C(L,\J)$, the following conditions are equivalent:
	\begin{enumerate}
		\item $\Phi^\ast(\Omega_{\BRST}+\Omega)$ is an $s$-BRST charge wrt $\hat{\J}$, for some $\Phi\in\Ham_{\geq 1}(\hat{M},\hat{L},\hat{\J})$,
		\item $A^\ast(\operatorname{pr}^{(1,0)}(\Omega_{\BRST}+\Omega))=\Omega_E[s]$, for some section $A\in\Gamma(\operatorname{GL}_+(E))$. 
	\end{enumerate}
	If the equivalent conditions (1)-(2) hold, then $s$, also denoted by $s_\Omega$, is completely determined by $\Omega$ through the relation:
	\begin{equation*}
	\im s_\Omega=\textnormal{``zero locus of $\operatorname{pr}^{(1,0)}(\Omega_{\BRST}+\Omega)$''}.
	\end{equation*}
\end{proposition}

\begin{proof}
	As a preliminary step we recall that the relation $\Omega'=\Omega_{\BRST}+\Omega$ establishes a one-to-one correspondence between MC elements $\Omega'$ of the graded Lie algebra $(\Gamma(\hat L),\{-,-\}_{\BFV})$ and MC elements $\Omega$ of the differential graded Lie algebra $(\Gamma(\hat L),\{-,-\}_{\BFV},d_{\BFV})$.
	Additionally the natural action of $\Ham_{\geq 1}(\hat M,\hat L,\hat J)$ on $\Gamma(\hat L)$ preserves the space of MC elements of  $(\Gamma(\hat L),\{-,-\}_{\BFV})$.
	
	Now we prove the equivalence of (1) and (2).
	Let $\{\lambda_t\}_{t\in I}\subset\calL^0_{\geq 1}$ and $\{a_t\}_{t\in I}\subset\Gamma(\hat{L})^{(1,1)}=\Gamma(\End(E))$ be smooth paths such that $a_t=\operatorname{pr}^{(1,1)}\lambda_t$.
	Since the $(0,0)$ bi-degree component of $\{\lambda_t,-\}_{\hat{\J}}$ reduces to $\{a_t,-\}_G$, whose symbol is zero, Lemma~\ref{lem:integrating_graded_do} guarantees that $\{\{\lambda_t,-\}_{\hat{\J}}\}_{t\in I}$ and $\{\{a_t,-\}_G\}_{t\in I}$ integrate to smooth paths $\{\Phi_t\}_{t\in I}\subset\Ham_{\geq 1}(\hat{M},\hat{L},\hat{\J})$ and $\{\Psi_t\}_{t\in I}\subset\Ham(\hat{M},\hat{L},G)^{(0,0)}$, respectively, so that
	\begin{equation}
	\label{eq:proof:geometric_MC_1}
	\Phi_t^\ast\lambda=\Psi_t^\ast\lambda\Mod\bigoplus_{k\geq 1}\Gamma(\hat{L})^{(p+k,q+k)},\quad\textnormal{for all}\ (p,q)\in\bbN^2,\ \lambda\in\Gamma(\hat{L})^{(p,q)}.
	\end{equation}
	Furthermore, from the very definition of the tautological Jacobi structure $G$ on $\hat{L}$, it follows that
	\begin{equation}
	\label{eq:proof:geometric_MC_2}
	\Psi_t= S_{{\scriptscriptstyle C^\infty(M)}}((A_t\otimes\id_{L^\ast})\oplus A_t^\ast)\otimes\id_L,
	\end{equation}
	where $\{A_t\}_{t\in I}\subset\Gamma(\operatorname{GL}(E))$ is the smooth path, with $A_0=\id_E$, which integrates $\{a_t\}_{t\in I}\subset\Gamma(\End(E))$, and is explicitly given by the time-ordered exponential
	\begin{equation*}
	A_t=\calT\exp\left(\int_0^t a_sds\right):=\sum_{n=0}^\infty\frac{1}{n!}\int_0^t\int_0^{s_{n-1}}\dots\int_0^{s_1}\left(a_{s_{n-1}}\circ a_{s_n}\circ\ldots\circ a_{s_0}\right)ds_{n-1}\dots ds_1ds_0.
	\end{equation*}
	
	Finally, \eqref{eq:proof:geometric_MC_1} and~\eqref{eq:proof:geometric_MC_2} imply that $\operatorname{pr}^{(1,0)}(\Phi_t^\ast\Omega')=A_t^\ast(\operatorname{pr}^{(1,0)}\Omega')$, for every $\Omega'\in\Gamma(\hat{L})^1$, which is enough to conclude the proof.
\end{proof}

\begin{definition}
	\label{def:geometric_MC}
	A \emph{geometric MC element of the BFV-complex} is a MC element $\Omega$ of $(\Gamma(\hat{L}),\{-,-\}_{\BFV},d_{\BFV})$ which satisfies, for some $s\in C(L,\J)$, the equivalent conditions (1)-(2) in Proposition~\ref{prop:geometric_MC}.
\end{definition}

In the following, $\MC_{\textnormal{geom}}(\BFV)$ will denote the set of all geometric MC elements of the BFV-complex.
In view of Proposition~\ref{prop:geometric_MC}, this set identifies with the orbit described by $\BRST(\hat{L},\hat{\J})$ under the natural action of $\Ham_{\geq 1}(\hat{M},\hat{L},\hat{\J})$ on $\Gamma(\hat{L})$, so that the induced action of $\Ham_{\geq 1}(\hat{M},\hat{L},\hat{\J})$ on $\MC_{\textnormal{geom}}(\BFV)$ is given by $\Phi\cdot\Omega=\Phi^\ast(\Omega_{\BRST}+\Omega)-\Omega_{\BRST}$.
The following theorem shows that, as already in the Poisson case~\cite[Theorem 2]{schatz2011moduli}, also in the Jacobi setting the BFV-complex of a coisotropic submanifold $S$ encodes the small coisotropic deformations of $S$.

\begin{theorem}
	\label{theor:coisotropic_def_space}
	There is a 1--1 correspondence between $\MC_{\textnormal{geom}}(\BFV)/\Ham_{\geq 1}(\hat{M},\hat{L},\hat{\J})$ and $C(L,\J)$ mapping $\Ham_{\geq 1}(\hat{M},\hat{L},\hat{\J}).\Omega$ to $s_\Omega$.
\end{theorem}

\begin{proof}
	It is a straightforward consequence of Propositions~\ref{prop:coisotropic_deformation_space} and~\ref{prop:geometric_MC}.
\end{proof}

The BFV-complex does also encode the information about the (local) moduli space of coisotropic deformations of $S$ under Hamiltonian equivalence.

\begin{definition}
	\label{def:Ham_homotopy_geometric_MC}
	A \emph{Hamiltonian homotopy} of geometric MC elements of the BFV-complex consists of
	\begin{itemize}
		\item a smooth path $\{\Omega_t\}_{t\in I}$ of geometric MC elements of $(\Gamma(\hat{L}),\{-,-\}_{\BFV},d_{\BFV})$,
		\item a smooth path $\{\Phi_t\}_{t\in I}$ of automorphisms of $(\hat{M},\hat{L},\hat{\J})$, with $\Phi_0=\id_{\hat{L}}$, which integrates $\{\lambda_t,-\}_{\hat{\J}}$, for some smooth path $\{\lambda_t\}_{t\in I}\subset\Gamma(\hat{L})^0$ (cf.~Definition~\ref{def:Hamiltonian_family}),
	\end{itemize}
	such that they are related by the compatibility condition $\Phi_t^\ast(\Omega_{\BRST}+\Omega_t)=\Omega_{\BRST}+\Omega_0$.
	Such Hamiltonian homotopy is said to \emph{interpolate} the geometric MC elements $\Omega_0$ and $\Omega_1$.
	If geometric MC elements $\Omega_0$ and $\Omega_1$ are interpolated by an Hamiltonian homotopy, then they are called \emph{Hamiltonian equivalent}, and we write $\Omega_0\sim_{\Ham}\Omega_1$.
	Indeed $\sim_{\Ham}$ is an equivalence on $\MC_{\textnormal{geom}}(\BFV)$.
\end{definition}

\begin{lemma}
	\label{lem:coisotropic_moduli_space}
	Let $s$ be an arbitrary coisotropic section of $\pi$ in $(L,\J)$.
	\newline\noindent
	1) Any two $s$-BRST charges wrt $\hat{\J}$ are Hamiltonian equivalent.
	\newline\noindent
	2) For any $\Omega_0,\Omega_1\in\MC_{\textnormal{geom}}(\BFV)$, if $s_{\Omega_0}=s_{\Omega_1}=s$ then $\Omega_0\sim_{\Ham}\Omega_1$.
\end{lemma}

\begin{proof}
	It is a straightforward consequence of Theorem~\ref{theor:uniqueness_BRST-charge} and Proposition~\ref{prop:geometric_MC}.
\end{proof}

An analogous notion of Hamiltonian equivalence for coisotropic sections of a Jacobi manifold, within a fat tubular neighborhood, appears already in~\cite{LOTV}.
For the reader's convenience we will present it again here.

\begin{definition}
	\label{def:Ham_homotopy_coisotropic_sections}
	A \emph{Hamiltonian homotopy} of coisotropic sections consists of
	\begin{itemize}
		\item a smooth path $\{s_t\}_{t\in I}\subset C(L,\J)$ of coisotropic sections of $(L,\J)$, and
		\item a smooth path $\{(F_t,\underline{\smash{F_t}})\}_{t\in I}\subset\Aut(M,L,\J)$ of automorphisms of $(M,L,\J)$, with $F_0=\id_L$, which integrates $\{\lambda_t,-\}_{\J}$, for some smooth path $\{\lambda_t\}_{t\in I}\subset\Gamma(L)$,
	\end{itemize}
	such that they are related by the compatibility condition $\im s_t=\underline{\smash{F_t}}(\im s_0)$.
	Such an Hamiltonian homotopy $\{(F_t,\underline{\smash{F_t}})\}_{t\in I}$ is said to \emph{interpolate} the coisotropic sections $s_0$ and $s_1$.
	If coisotropic sections $s_0$ and $s_1$ are interpolated by an Hamiltonian homotopy, then they are called \emph{Hamiltonian equivalent}, and we write $s_0\sim_{\Ham}s_1$.
	Indeed $\sim_{\Ham}$ is an equivalence relation on $C(L,\J)$.
\end{definition}

The following theorem shows that, as already in the Poisson case~\cite[Theorem 4]{schatz2011moduli}, also in the Jacobi setting the BFV-complex of a coisotropic submanifold $S$ encodes the local moduli space of coisotropic deformations of $S$ under Hamiltonian equivalence.

\begin{theorem}
	\label{theor:Mod_Ham}
	There is a 1--1 correspondence between $\MC_{\textnormal{geom}}(\BFV)/{\sim_{\Ham}}$ and $C(L,\J)/{\sim_{\Ham}}$ mapping $[\Omega]_{\Ham}$ to $[s_\Omega]_{\Ham}$.
\end{theorem}

\begin{proof}
	We have to prove that, for all $\Omega_0,\Omega_1\in\MC_{\textnormal{geom}}(\BFV)$, the following conditions are equivalent:
	\begin{itemize}
		\item[a)] $\Omega_0$ and $\Omega_1$ are Hamiltonian equivalent,
		\item[b)] $s_0:=s_{\Omega_0}$ and $s_1:=s_{\Omega_1}$ are Hamiltonian equivalent.
	\end{itemize}
	
	a) $\Longrightarrow$ b).
	Assume there is a Hamiltonian homotopy of geometric MC elements given by $\{\Omega_t\}_{t\in I}\subset\MC_{\textnormal{geom}}(\BFV)$ and $\{\Phi_t\}_{t\in I}\subset\Aut(\hat{M},\hat{L},\hat{\J})$, with
	\begin{equation*}
	\Phi_0=\id_{\hat{L}},\quad\frac{d}{dt}\Phi_t^\ast=\{\lambda_t,-\}_{\hat{\J}}\circ\Phi_t^\ast,
	\end{equation*}
	for some $\{\lambda_t\}_{t\in I}\subset\Gamma(\hat{L})^0$.
	The latter can be canonically projected onto a Hamiltonian homotopy of coisotropic sections of $(L,\J)$, given by $\{s_t\}_{t\in I}\subset C(L,\J)$ and $\{(F_t,\underline{\smash{F_t}})\}_{t\in I}\subset\Aut(M,L,\J)$, with
	\begin{equation*}
	F_0=\id_{L},\quad\frac{d}{dt}F_t^\ast=\{\underline\lambda_t,-\}_{\J}\circ F_t^\ast,
	\end{equation*}
	for some $\{\underline\lambda_t\}_{t\in I}\subset\Gamma(L)$.
	Such projection is defined by setting
	\begin{equation*}
	s_t:=s_{\Omega_t},\quad \underline\lambda_t:=\operatorname{pr}^{(0,0)}\lambda_t,\quad F_t^\ast=\operatorname{pr}^{(0,0)}\circ \left.\Phi_t^\ast\right|_{\Gamma(\hat{L})^{(0,0)}}.
	\end{equation*}
	
	b) $\Longrightarrow$ a).
	Assume there is a Hamiltonian homotopy of coisotropic sections of $(L,\J)$ given by $\{s_t\}_{t\in I}\subset C(L,\J)$ and $\{(F_t,\underline{\smash{F_t}})\}_{t\in I}\subset\Aut(M,L,\J)$.
	In view of Lemma~\ref{lem:coisotropic_moduli_space}, it is enough to lift the latter to a Hamiltonian homotopy of geometric MC elements of $(\Gamma(\hat L),\hat J=\{-,-\}_{\BFV},d_{\BFV})$ intertwining $\Omega_0'-\Omega_{\BRST}$ and  $\Omega_1'-\Omega_{\BRST}$, where $\Omega_i'$ is an $s_i$-BRST charge wrt $\hat J$ for $i=0,1$.
	
	So, let $\nabla$ be a $\Diff L$-connection in $L\to M$ obtained by pulling back, along $NS\overset{\pi}{\to}S$, a $\Diff\ell$-connection in $NS\overset{\pi}{\to}S$.
	For an arbitrary $s_0$-BRST charge $\tilde\Omega_0$ wrt $\hat{\J}^\nabla$, Proposition~\ref{prop:main_result_bis} allows us to lift $\{F_t\}_{t\in I}$ to a smooth path $\{\calF_t\}_{t\in I}\subset\Aut(\hat{M},\hat{L},\hat{\J}^\nabla)$ such that
	\begin{equation*}
	\calF_0=\id_{\hat{L}},\quad\frac{d}{dt}\calF_t^\ast=\{\lambda_t,-\}_{\hat{\J}^\nabla}\circ\calF_t^\ast,
	\end{equation*}
	for some smooth path $\{\lambda_t\}_{t\in I}\subset\Gamma(\hat{L})^0$, and additionally $\tilde\Omega_t:=(\calF_t)_\ast\tilde\Omega_0$ is an $s_t$-BRST charge wrt $\hat{\J}^\nabla$.
	In view of Theorem~\ref{theor:uniqueness_BFV_brackets}, there is an automorphism $\calG$ of the graded line bundle $\hat{L}\to\hat{M}$ such that $\calG^\ast(\hat{\J}^\nabla)=\hat{\J}$, and $\Omega'_t:=\calG^\ast\tilde\Omega_t$ is an $s_t$-BRST charge wrt $\hat{\J}$.
	Hence $\{\Phi_t:=\calG^{-1}\circ\calF_t\circ\calG\}_{t\in I}$ and $\{\Omega_t:=\Omega'_t-\Omega_{\BRST}\}_{t\in I}$ provide the required Hamiltonian homotopy of geometric MC elements.
\end{proof}

The BFV-complex does also encode the information about the (local) moduli space of coisotropic deformations of $S$ under Jacobi equivalence.

\begin{definition}
	\label{def:Jac_homotopy_geometric_MC}
	A \emph{Jacobi homotopy} of geometric MC elements of the BFV-complex consists of
	\begin{itemize}
		\item a smooth path $\{\Omega_t\}_{t\in I}$ of geometric MC elements of $(\Gamma(\hat{L}),\{-,-\}_{\BFV},d_{\BFV})$,
		\item a smooth path $\{\Phi_t\}_{t\in I}$ of Jacobi automorphisms of $(\hat{M},\hat{L},\hat{\J})$ with $\Phi_0=\id_{\hat L}$ (cf.~Definition~\ref{def:Jacobi_automorphism}),
	\end{itemize}
	such that they are related by the compatibility condition $\Phi_t^\ast(\Omega_{\BRST}+\Omega_t)=\Omega_{\BRST}+\Omega_0$.
	Such Jacobi homotopy is said to \emph{interpolate} the geometric MC elements $\Omega_0$ and $\Omega_1$.
	If geometric MC elements $\Omega_0$ and $\Omega_1$ are interpolated by a Jacobi homotopy, then they are called \emph{Jacobi equivalent}, and we write $\Omega_0\sim_{\Jac}\Omega_1$.
	Indeed $\sim_{\Jac}$ is an equivalence on $\MC_{\textnormal{geom}}(\BFV)$ coarser than $\sim_{\Ham}$.
\end{definition}

We present now the analogous notion of Jacobi equivalence for coisotropic sections within a fat tubular neighborhood.

\begin{definition}
	\label{def:Jac_homotopy_coisotropic_sections}
	A \emph{Jacobi homotopy} of coisotropic sections of $(L,\J)$ consists of
	\begin{itemize}
		\item a smooth path $\{s_t\}_{t\in I}\subset C(L,\J)$ of coisotropic sections of $(L,\J)$, and
		\item a smooth path $\{(F_t,\underline{\smash{F_t}})\}_{t\in I}\subset\Aut(M,L,\J)$, with $F_0=\id_L$, of Jacobi automorphisms of $(M,L,\J)$ (cf.~\cite[Definition~2.19]{LOTV}),
	\end{itemize}
	such that they are related by compatibility condition $\im s_t=\underline{\smash{F_t}}(\im s_0)$.
	Such Jacobi homotopy is said to \emph{interpolate} coisotropic sections $s_0$ and $s_1$.
	If coisotropic sections $s_0$ and $s_1$ are interpolated by a Jacobi homotopy, they are called \emph{Jacobi equivalent}, and we write $s_0\sim_{\Jac}s_1$.
	Indeed $\sim_{\Jac}$ is an equivalence relation on $C(L,\J)$ coarser than $\sim_{\Ham}$.
\end{definition}

The following theorem shows that, in the Jacobi setting, and a fortiori in the Poisson setting, the BFV-complex of a coisotropic submanifold $S$ encodes the local moduli space of coisotropic deformations of $S$ under Jacobi equivalence.
\begin{theorem}
	\label{theor:Mod_Jac}
	There is a 1--1 correspondence between $\MC_{\textnormal{geom}}(\BFV)/{\sim_{\Jac}}$ and $C(L,\J)/{\sim_{\Jac}}$ mapping $[\Omega]_{\Jac}$ to $[s_\Omega]_{\Jac}$.
\end{theorem}

\begin{proof}
	We have to prove that, for all $\Omega_0,\Omega_1\in\MC_{\textnormal{geom}}(\BFV)$, the following conditions are equivalent:
	\begin{itemize}
		\item[a)] $\Omega_0$ and $\Omega_1$ are Jacobi equivalent,
		\item[b)] $s_0:=s_{\Omega_0}$ and $s_1:=s_{\Omega_1}$ are Jacobi equivalent.
	\end{itemize}
	a) $\Longrightarrow$ b).
	Assume there is a Jacobi homotopy of geometric MC elements given by $\{\Omega_t\}_{t\in I}\subset\MC_{\textnormal{geom}}(\BFV)$ and $\{\Phi_t\}_{t\in I}\subset\Aut(\hat{M},\hat{L},\hat{\J})$.
	The latter can be canonically projected onto a Jacobi homotopy of coisotropic sections given by $\{s_t\}_{t\in I}\subset C(L,\J)$ and $\{(F_t,\underline{\smash{F_t}})\}_{t\in I}\subset\Aut(M,L,\J)$.
	Such projection is defined by setting
	\begin{equation*}
	s_t:=s_{\Omega_t},\qquad F_t^\ast:=({\operatorname{pr}^{(0,0)}}\circ{\Phi_t^\ast})|_{\Gamma(\hat{L})^{(0,0)}}.
	\end{equation*}
	
	b) $\Longrightarrow$ a).
	Assume there is a Jacobi homotopy of coisotropic sections of $(L,\J)$ given by $\{s_t\}_{t\in I}\subset C(L,\J)$ and $\{(F_t,\underline{\smash{F_t}})\}_{t\in I}\subset\Aut(M,L,\J)$.
	In view of Lemma~\ref{lem:coisotropic_moduli_space}, it is enough to lift the latter to a Jacobi homotopy of geometric MC elements of $(\Gamma(\hat L),\hat J=\{-,-\}_{\BFV},d_{\BFV})$ intertwining $\Omega_0'-\Omega_{\BRST}$ and  $\Omega_1'-\Omega_{\BRST}$, where $\Omega_i'$ is an $s_i$-BRST charge wrt $\hat J$ for $i=0,1$.
	
	So, fix an $s_0$-BRST charge $\tilde\Omega_0$ wrt $\hat{\J}$.
	Proposition~\ref{prop:main_result} allows us to lift $\{F_t\}_{t\in I}$ to a smooth path $\{\calF_t\}_{t\in I}$ of bi-degree $(0,0)$ graded automorphisms of $\hat{L}\to\hat{M}$, with $\calF_0=\id_{\hat L}$, such that
	\begin{itemize}
		\item $\hat{\J}_t:=(\calF_t)_\ast\hat{\J}$ is a lifting of $\J$ to a Jacobi structure on $\hat{L}\to\hat{M}$,
		\item $\tilde\Omega_t:=(\calF_t)_\ast\tilde\Omega_0$ is a $s_t$-BRST charge wrt $\hat{\J}_t$.
	\end{itemize}
	Hence, in view of Theorem~\ref{theor:uniqueness_BFV_brackets}, there is also a smooth path $\{\calG_t\}_{t\in I}$ of automorphisms of the graded line bundle $\hat{L}\to\hat{M}$, with $\calG_0=\id_{\hat{L}}$, such that $(\calG_t)_\ast\hat{\J}_t=\hat{\J}$, and $\Omega'_t:=(\calG_t)_\ast\tilde\Omega_t$ is an $s_t$-BRST charge wrt $\hat{\J}$.
	Finally, $\{\Phi_t:=\calG_t\circ\calF_t\}_{t\in I}$ and $\{\Omega_t:=\Omega'_t-\Omega_{\BRST}\}_{t\in I}$ provide the desired Jacobi homotopy of geometric MC elements.
\end{proof}

\section{An example}
\label{sec:example}

In this section we briefly describe an example of obstructed coisotropic submanifold in a contact manifold in terms of the associated BFV-complex.
This example was first considered in~\cite[Examples 3.5 and 3.8]{tortorella2016rigidity} and it is inspired by an analogous one in the symplectic setting by Zambon~\cite{zambon2008example}.
For more details see~\cite{tortorella2017thesis} where the same example is also re-interpreted in terms of the associated $L_\infty$-algebra.

We consider the trivial vector bundle $\calE:=\bbT^5\times\bbR^2\overset{\pi}{\longrightarrow} S:=\bbT^5$, and the contact structure on $\calE$ defined by the following global contact form on $\calE$
\begin{equation*}
\theta_\calE:=y_1d\phi_1+y_2d\phi_2+\sin\phi_3d\phi_4+\cos\phi_3d\phi_5.
\end{equation*}
where $\phi_1,\ldots,\phi_5$ are the angular coordinates on $\bbT^5$, and $y_1,y_2$ are the Euclidean coordinates on $\bbR^2$.
It is easy to see that $S\simeq\bbT^5\times\{(0,0)\}$ is a coisotropic submanifold of $\calE$.
We will denote by $J$, or $\{-,-\}_J$, the Jacobi structure on the trivial line bundle $L=\calE\times\bbR\to\calE$ which is canonically determined by the contact structure on $\calE$.

Let us start constructing the BFV-complex of $S$.
Denote by $\eta^1,\eta^2$ the canonical global frame of $\calE\overset{\pi}{\to}S$, so that $\xi^1:=\pi^\ast\eta^1,\xi^2:=\pi^\ast\eta^2$ is the canonical global frame of $E:=\pi^\ast\calE=\calE\times\bbR^2\to\calE$.
Then we have that
\begin{equation*}
C^\infty(\hat{M})
=C^\infty(\calE)\otimes_{C^\infty(S)}\Gamma(S^\bullet(\calE[-1]\oplus \calE^\ast[1]))=C^\infty(\calE)\otimes_{\bbR}\wedge^\bullet(\bbR^2\oplus(\bbR^2)^\ast),
\end{equation*}
and $\Gamma(\hat{L})$ coincides with $C^\infty(\hat{M})$ seen as a left module over itself.
Since $E\to\calE$ is the trivial vector bundle, we can pick the trivial flat $DL$-connection $\nabla$ in $E$.
Then, according to Proposition~\ref{prop:existence_BFV_brackets}, the lifted Jacobi structure defining the BFV-bracket $\{-,-\}_{\BFV}$ reduces to $\hat J^\nabla=G+\iota_\nabla(J)$.
We can easily compute
\begin{equation}
\label{eq:lifted_Jacobi_example}
\begin{aligned}
\hat{\J}^\nabla&=\frac{\partial}{\partial\phi_3} X+Y\left(y_1\frac{\partial}{\partial y_1}+y_2\frac{\partial}{\partial y_2}\right)-\frac{\partial}{\partial\phi_1}\frac{\partial}{\partial y_1}-\frac{\partial}{\partial\phi_2}\frac{\partial}{\partial y_2}\\
&\phantom{=}-Y\left(\id-\xi^1\frac{\partial}{\partial\xi^1}-\xi^2\frac{\partial}{\partial\xi^2}\right)+\frac{\partial}{\partial\xi^1}\frac{\partial}{\partial\xi^\ast_1}+\frac{\partial}{\partial\xi^2}\frac{\partial}{\partial\xi^\ast_2}.
\end{aligned}
\end{equation}
where $X:=\cos\phi_3\frac{\partial}{\partial \phi_4}-\sin\phi_3\frac{\partial}{\partial \phi_5}$ and $Y:=\sin\phi_3\frac{\partial}{\partial \phi_4}+\cos\phi_3\frac{\partial}{\partial \phi_5}$.
Further $\Omega_E=y_1\xi^1+y_2\xi^2$ satisfies the MC-equation $\{\Omega_E,\Omega_E\}_{\BFV}=0$.
Hence the procedure for the construction of the $s$-BRST-charge, described in the proof of Theorem~\ref{theor:existence_BRST-charge}, for $s=0$ produces as its output $\Omega_{\BRST}=\Omega_E$.
So $d_{\BFV}:=\{\Omega_{\BRST},-\}_{\BFV}$ is given by
\begin{equation*}
d_{\BFV}=y_1\frac{\partial}{\partial\xi_1^\ast}+y_2\frac{\partial}{\partial\xi_2^\ast}-\xi^1\frac{\partial}{\partial\phi_1}-\xi^2\frac{\partial}{\partial\phi_2}+(y_1\xi^1+y_2\xi^2)Y.
\end{equation*}

Finally we describe the space $C(L,J)$ of coisotropic sections of $\pi:\calE\to S$.
Let $s$ be an arbitrary section of $\pi$.
Hence $s=f_1\eta^1+f_2\eta^2$, and $\Omega_E[s]=(y_1-f_1)\xi^1+(y_2-f_2)\xi^2$, for arbitrary $f_1,f_2\in C^\infty(\bbT^5)$.
From~\eqref{eq:lifted_Jacobi_example} it follows that
\begin{equation*}
\label{eq:obstructed_example_contact_BFV_coisotropic1}
\{\Omega_E[s],\Omega_E[s]\}_{\BFV}=2\left(\frac{\partial f_1}{\partial \phi_3}Xf_2-\frac{\partial f_2}{\partial \phi_3}Xf_1+\frac{\partial f_1}{\partial\phi_2}-\frac{\partial f_2}{\partial\phi_1}+y_1Yf_2-y_2Yf_1\right)\xi^1\xi^2.
\end{equation*}
Applying $\wp[s]$ to the latter, Lemma~\ref{lem:BRST_coisotropic} provides us with a complete description of $C(L,J)$, i.e.~$s=f_1\eta^1+f_2\eta^2$ is a coisotropic section iff $f_1,f_2\in C^\infty(\bbT^5)$ satisfy the following non-linear first order pde
\begin{equation*}
\label{eq:obstructed_example_contact_BFV_nonlinear_pde}
\frac{\partial f_1}{\partial \phi_3}Xf_2-\frac{\partial f_2}{\partial \phi_3}Xf_1+\frac{\partial f_1}{\partial\phi_2}-\frac{\partial f_2}{\partial\phi_1}+f_1Yf_2-f_2Yf_1=0,
\end{equation*}
which duly agrees with that one found in~\cite[Eq.~3.1]{tortorella2016rigidity} by analytical methods.

\appendix
\renewcommand*{\thesection}{\Alph{section}}

\section{Graded symmetric multi-derivations on graded line bundles}
\label{app:graded_multi-do}

In this Appendix we collect basic facts, including conventions and notations, concerning graded symmetric multi-derivations on graded line bundles, and, in particular, graded Jacobi structures and Jacobi bi-derivations.
In doing so we will use the language of $\bbZ$-graded differential geometry (see, e.g.,~\cite{Mehta2006}).

\subsection{Graded symmetric multi-derivations}

Let $\scrM$ be a $\bbZ$-graded manifold and let $\scrP,\scrQ$ be graded vector bundles over it.
A \emph{degree $k$ graded first order differential operator} from $\scrP$ to $\scrQ$ is a degree $k$ graded $\bbR$-linear map $\square:\Gamma(\scrP)\to\Gamma(\scrQ)$, such that, for all $a_1,a_2\in C^\infty(\scrM)$, 
\begin{equation*}
[[\square,a_1],a_2]=0.
\end{equation*}
A \emph{degree $k$ graded derivation} of $\scrP$ is a degree $k$ graded $\bbR$-linear map $\square:\Gamma(\scrP)\to\Gamma(\scrP)$, such that there is a (necessarily unique) vector field $X_\square\in\frakX(\scrM)$, the \emph{symbol} of $\scrM$, satisfying the following graded Leibniz rule\begin{equation*}
\square(a p)=X_\square(a) p+(-)^{|a|k}a\square(p),
\end{equation*}
for all homogeneous $a\in C^\infty(\scrM)$, and $p\in\Gamma(\scrP)$.
Here and throughout the paper we denote by $|v|$ the degree of a homogeneous element $v$ in a graded vector space.

\begin{example}
	We denote by $\bbR_{\scrM} := \scrM \times \bbR \to\scrM$ the trivial line bundle over $\scrM$, so that $\Gamma(\bbR_{\scrM})$ coincides with $C^\infty(\scrM)$ seen as a left module over itself.
	Then degree $k$ graded derivations $\square$ from $\bbR_{\scrM}$ to $\bbR_{\scrM}$ identify with the pairs $(X,f)$ of degree $k$ homogeneous $X\in\frakX(\scrM)$ and $f\in C^\infty(\scrM)$ by means of the relation
	\begin{equation*}
	\square(a)=X(a)+(-)^{|a|k}af,
	\end{equation*}
	for all homogeneous $a\in C^\infty(\scrM)$.
\end{example}

\begin{remark}
	A derivation of a graded vector bundle $\scrP$ is, in particular, a first order differential operator.
	In general, the converse is not true, unless $\scrP\to\scrM$ is a (graded) line bundle.
	In this case derivations of $\scrP$ are the same as first order differential operators $\Gamma (\scrP) \to \Gamma (\scrP)$.
\end{remark}

\begin{remark}
	\label{rem:LRalgebra}
	Denote by $\Diff(\scrP)^k$ the space of degree $k$ graded derivations of $\scrP$.
	Then the whole space of graded derivations of $\scrP$
	\begin{equation*}
	\Diff(\scrP)^\bullet:=\bigoplus_{k\in\bbZ}\Diff(\scrP)^k
	\end{equation*}
	is a graded $C^\infty(\scrM)$-module, and also a graded Lie algebra, with the Lie bracket given by the usual graded commutator $[-,-]$.
	Additionally, the \emph{symbol map} $\Diff(\scrP)\to\frakX(\scrM),\ \square\mapsto X_\square$, is both a $C^\infty(\scrM)$-linear map and a Lie algebra morphism, and satisfies the following compatibility condition   
	\begin{equation*}
	[\square_1,a\square_2]=X_{\square_1}(a)\square_2+(-)^{|a||\square_1|}a[\square_1,\square_2],
	\end{equation*}
	for all homogeneous $a\in C^\infty(\scrM)$, and $\square_1,\square_2\in\Diff(\scrP)$, i.e.~the pair $(C^\infty(\scrM),\Diff(\scrP))$ is a graded Lie-Rinehart algebra (see, e.g.,~\cite{huebschmann2004lie} and~\cite{vitagliano2015homotopy}).
	Hence $D(\scrP)$ is the module of sections of a graded Lie algebroid over $\scrM$, called the \emph{gauge algebroid} (or the \emph{Atiyah algebroid}) of $\scrP$ (for more details about the Atiyah algebroid, at least in the non-graded case, see, e.g.,~\cite{mackenzie2005general, KM2002,vitagliano2015generalized}).
	Abusing the notation, we will often denote the gauge algebroid by the same symbol $D(\scrP)$ as for its sections.
	Accordingly, we will speak about, e.g., $D(\scrP)$-connections, representations of $D(\scrP)$, etc., without further comments.
	Notice that $\scrP$ carries a canonical representation of the gauge algebroid $D(\scrP)$, the \emph{tautological representation}, given by the action of derivations on sections.
	For more information about Lie algebroid, Lie algebroid connections, and, in particular, Lie algebroid representations, we refer the reader to \cite{mackenzie2005general, CF2010} and references therein.
\end{remark}

A \emph{degree $k$ graded symmetric first order $n$-ary differential operator} from $\scrP$ to $\scrQ$ is a degree $k$ graded symmetric $\bbR$-multilinear map
\begin{equation}
\square:\underbrace{\Gamma(\scrP)\times\cdots\times\Gamma(\scrP)}_{n-\textnormal{times}}\longrightarrow\Gamma(\scrQ)
\end{equation}
which is a graded differential operator from $\scrP$ to $\scrQ$ in each entry.

\begin{example}
	Degree $k$ graded symmetric first order n-ary differential operators $\square$ from $\bbR_{\scrM}$ to $\bbR_{\scrM}$ identify with pairs $(\Lambda,\Gamma)$ of degree $k$ homogeneous multivectors $\Lambda\in\frakX^k(\scrM)$ and $\Gamma\in\frakX^{k-1}(\scrM)$ by means of the relation
	\begin{equation*}
	\square(a_1,\ldots,a_n)=\Lambda(a_1,\ldots,a_n)+\sum_{i=1}^n(-)^{|a_i|(|a_{i+1}|+\ldots+|a_n|)}\Gamma(a_1,\ldots,\widehat{a_i},\ldots,a_n)a_i,
	\end{equation*}
	for all homogeneous $a_1,\ldots,a_n\in C^\infty(\scrM)$.
\end{example}

\begin{remark}
	Let $\scrL\to\scrM$ be a graded line bundle.
	We denote by $\Diff^n(\scrL,\bbR_{\scrM})^\bullet:=\bigoplus_{k\in\bbZ}\Diff^n(\scrL,\bbR_{\scrM})^k$ the space of graded symmetric $n$-ary first order differential operators from $\scrL$ to $\bbR_{\scrM}$.
	The whole space of graded symmetric multi-differential operators
	\begin{equation*}
	\Diff^\star(\scrL,\bbR_{\scrM})^\bullet:=\bigoplus_{n\in\bbZ}\Diff^n(\scrL,\bbR_{\scrM})^\bullet,
	\end{equation*}
	is a graded left $C^\infty(\scrM)$-module in the obvious way (we set $\Diff^0(\scrL,\bbR_{\scrM}):=C^\infty(\scrM)$, and $\Diff^n(\scrL,\bbR_{\scrM}):=0$, for all $n<0$).
	Here, and throughout the paper, the superscripts ${}^\star$ and ${}^\bullet$ refer to, respectively, the arity and the total degree.
	The $C^\infty(\scrM)$-module $\Diff^\star(\scrL,\bbR_{\scrM})$ is, additionally, a unital associative graded commutative $\bbR$-algebra whose product is given by
	\begin{equation}
	\label{eq:product_of_multi-do}
	(\Delta\cdot\square)(\lambda_1,\ldots,\lambda_{h+k})=\sum_{\tau\in S_{h,k}}(-)^{\chi}\epsilon(\tau,\boldsymbol\lambda)\Delta(\lambda_{\tau(1)},\ldots,\lambda_{\tau(h)})\cdot\square(\lambda_{\tau(h+1)},\ldots,\lambda_{\tau(h+k)}),
	\end{equation}
	for all homogeneous $\Delta\in\Diff^h(\scrL,\bbR_{\scrM}),\square\in\Diff^k(\scrL,\bbR_{\scrM})$, and $\lambda_1,\ldots,\lambda_{h+k}\in\Gamma(\scrL)$.
	In~\eqref{eq:product_of_multi-do}, $\chi:=\sum_{i=1}^{h}|\square||\lambda_{\tau(i)}|$, and $\epsilon(\tau,\boldsymbol{\lambda})$ denotes the symmetric Koszul symbol corresponding to the un-shuffle permutation $\tau \in S_{h,k}$ and the $(h+k)$-tuple $\boldsymbol{\lambda}:=(\lambda_1,\ldots,\lambda_{h+k})$.
	Accordingly, we have a canonical (degree $0$) graded $C^\infty(\scrM)$-algebra isomorphism
	\begin{equation*}
	\Diff^\star(\scrL,\bbR_{\scrM})\simeq S_{{\scriptscriptstyle C^\infty(\scrM)}}\Diff^1(\scrL,\bbR_{\scrM}),
	\end{equation*}
	where $S_{{\scriptscriptstyle C^\infty(\scrM)}}$ denotes the graded symmetric algebra over $C^\infty(\scrM)$.
\end{remark}

\begin{remark}
	Let $\scrL\to\scrM$ be a graded line bundle.
	We denote by $\Diff^n(\scrL)^\bullet:=\bigoplus_{k\in\bbZ}\Diff^n(\scrL)^k$ the space of graded symmetric $n$-ary first order differential operators from $\scrL$ to $\scrL$.
	Since $\scrL$ is a line bundle, the operators in $\Diff^n(\scrL)^\bullet$ are, in fact, derivations in each entry.
	Accordingly, we will also call them (graded symmetric) \emph{multi-derivations}.
	In particular, $\Diff^1(\scrL)^\bullet=\Diff(\scrL)^\bullet$.
	The whole space of graded symmetric multi-derivations
	\begin{equation*}
	\Diff^\star(\scrL)^\bullet:=\bigoplus_{n\in\bbZ}\Diff^n(\scrL)^\bullet,
	\end{equation*}
	is a graded left $C^\infty(\scrM)$-module in the obvious way (we set $\Diff^0(\scrL):=\Gamma(\scrL)$, and $\Diff^n(\scrL):=0$, for all $n<0$).
	The $C^\infty(\scrM)$-module $\Diff^\star(\scrL)$ is, additionally, a graded $\Diff^\star(\scrL,\bbR_{\scrM})$-module whose product is given by a formula similar to~\eqref{eq:product_of_multi-do}.
	Accordingly, we have a canonical (degree $0$) isomorphism of graded $\Diff^\star(\scrL,\bbR_{\scrM})$-modules
	\begin{equation*}
	\Diff^\star(\scrL)\simeq\Diff^\star(\scrL,\bbR_{\scrM})\otimes_{{\scriptscriptstyle C^\infty(\scrM)}}\Gamma(\scrL).
	\end{equation*}
\end{remark}

Now we recall the notion of Gerstenhaber-Jacobi algebra (\cite[Definition~B.1]{LOTV} and \cite[Definition 1.9]{tortorella2017thesis}).
Our definition slightly generalizes the notion of Gerstenhaber-Jacobi algebra as defined in~\cite{GM2001} (see also Example~\ref{ex:GJ_algebra}).
\begin{definition}
	\label{def:GJ_algebra}
	A \emph{Gerstenhaber-Jacobi algebra} consists of a unital associative graded commutative algebra $\calA$ and a graded $\calA$-module $\calL$, equipped with a graded Lie bracket $[-,-]$ on $\calL$ and an action by derivations of $\calL$ on $\calA$, i.e.~a (degree $0$) graded Lie algebra morphism $X_{(-)}:\calL\to\Der(\calA)$, such that
	\begin{equation*}
	[\lambda,a\mu]=X_\lambda(a)\mu+(-)^{|a||\lambda|}a[\lambda,\mu],\qquad a\in\calA,\quad \lambda,\mu\in\calL.
	\end{equation*}
\end{definition}

\begin{example}
	\label{ex:GJ_algebra}
	In the special case $\calL=\calA[1]$, we recover the notion of Gerstenhaber-Jacobi algebra as defined in~\cite{GM2001}.
	Additionally, if $\calL=\calA[1]$, and $X_a=[a,-]$ for all $a\in\calA[1]$, then we recover the standard notion of Gerstenhaber algebra.
	For further examples of Gerstenhaber-Jacobi algebras we refer the reader to~\cite[Sec.~1.3]{tortorella2017thesis}.
\end{example}

\begin{remark}
	We will speak about Gerstenhaber-Jacobi algebras even in the case when both $\calA$ and $\calL$ are concentrated in degree 0.
	For instance, we will speak about the \emph{Gerstenhaber-Jacobi} algebra $(C^\infty (M), \Gamma (L))$ of a Jacobi manifold $(M,L,\{-,-\})$.
	This allows us to distinguish from the more popular (but less general) \emph{Jacobi algebras} which are non-graded Gerstenhaber-Jacobi algebras such that $\calL = \calA$.
\end{remark}

The following proposition provides a canonical example of a Gerstenhaber-Jacobi algebra of key importance for this paper.

\begin{proposition}
	\label{prop:GJalgebra}
	For every graded line bundle $\scrL\to\scrM$, there is a natural Gerstenhaber-Jacobi algebra structure $(\ldsb-,-\rdsb,X_{(-)})$ on $(\Diff^\star(\scrL,\bbR_{\scrM}),\Diff^\star(\scrL))$, uniquely determined by
	\begin{equation}\label{eq:SJbrackets}
	\ldsb \square,\square'\rdsb=[\square,\square'],\qquad
	\ldsb \square,\lambda\rdsb=\square(\lambda),\qquad
	\ldsb \lambda,\mu\rdsb=0,
	\end{equation}
	for all $\square,\square'\in\Diff(\scrL)$, and $\lambda,\mu\in\Gamma(\scrL)$.
	The Lie bracket $\ldsb-,-\rdsb$ is called \emph{the Schouten-Jacobi bracket}.
\end{proposition}

Since $\scrL$ carries a representation of the gauge algebroid $\Diff(\scrL)$ (Remark~\ref{rem:LRalgebra}), Proposition~\ref{prop:GJalgebra} (as the graded version of Proposition 1.14 in~\cite{tortorella2017thesis}) is an immediate corollary of Proposition 1.12 in~\cite{tortorella2017thesis}.

\begin{remark}
	\label{rem:Gerstenhaber_product_and_SJ_bracket}
	When computing with multi-derivations it is very helpful to have an explicit expression for the Schouten-Jacobi bracket.
	It is easy to see that
	\begin{equation}
	\label{eq:Gerstenhaber_product_and_SJ_bracket}
	\ldsb\square,\square'\rdsb=\square\bullet \square'-(-)^{|\square||\square'|}\square'\bullet\square,
	\end{equation}
	for all homogeneous $\square,\square'\in\Diff^\star(\scrL)$.
	In~\eqref{eq:Gerstenhaber_product_and_SJ_bracket}, we have denoted by $\bullet$ the \emph{Gerstenhaber product}
	(of multi-derivations).
	The latter is defined by 
	\begin{equation*}
	\square_1\bullet\square_2(\lambda_1,\ldots,\lambda_{k_1+k_2+1})=\!\!\!\!\sum_{\tau\in S_{k_2+1,k_1}}\!\!\!\!\epsilon(\tau,\boldsymbol\lambda)\square_1(\square_2(\lambda_{\tau(1)},\ldots,\lambda_{\tau(k_2+1)}),\lambda_{\tau(k_2+2)},\ldots,\lambda_{\tau(k_1+k_2+1)}),
	\end{equation*}
	for all homogeneous $\square_i\in\Diff^{k_i+1}(\scrL)$, with $i=1,2$, and $\lambda_1,\ldots,\lambda_{k_1+k_2+1}\in\Gamma(\scrL)$.
	It is also helpful to point out that, it follows from~\eqref{eq:Gerstenhaber_product_and_SJ_bracket} that
	\begin{equation}
	\label{eq:Gerstenhaber_product_and_SJ_bracket_bis}
	\square(\lambda_1,\ldots,\lambda_n)=\ldsb\ldsb\ldots\ldsb\square,\lambda_1\rdsb,\ldots\rdsb,\lambda_n\rdsb,\qquad\lambda_1,\ldots,\lambda_n\in\Gamma(\scrL).
	\end{equation}
	for all $\square\in\Diff^n(\scrL)$.
\end{remark}

\subsection{Graded Jacobi bundles}

In this subsection we introduce Jacobi structures on graded line bundles, and their equivalent description in terms of Jacobi bi-derivations.

\begin{definition}
	\label{def:graded_Jacobi_structure}
	A \emph{graded Jacobi structure} on a graded line bundle $\scrL\to\scrM$ is given by a \emph{graded Jacobi bracket} $\{-,-\} :\Gamma(\scrL)\times\Gamma (\scrL)\rightarrow\Gamma(\scrL)$, i.e.~a (degree $0$) graded Lie bracket which is a first order differential operator, hence a derivation, in each entry.
	A \emph{graded Jacobi bundle} (over $\scrM$) is a graded line bundle (over $\scrM$) equipped with a graded Jacobi structure.
	A \emph{graded Jacobi manifold} is a graded manifold equipped with a graded Jacobi bundle over it.
\end{definition}
A Jacobi structure on a graded line bundle $\scrL\to\scrM$ is the same as a Gerstenhaber-Jacobi algebra structure on $(C^\infty(\scrM),\Gamma(\scrL))$.

\begin{remark}\label{rem:shift}
	According to Definition \ref{def:graded_Jacobi_structure}, a graded Jacobi bracket is a \emph{skew-symmetric} bi-derivation.
	However, it turns out that formulas get much simplified if we understand graded Jacobi structures as \emph{symmetric} bi-derivations on a shifted line bundle, via \emph{d\'ecalage}.
	This is made precise below.
\end{remark}

\begin{definition}
	\label{def:Jacobi_bi-do}
	A \emph{graded Jacobi bi-derivation} on a graded line bundle $\scrL\to\scrM$ is a degree $1$ graded symmetric bi-derivation $J\in\Diff^2(\scrL)^1$ such that $\ldsb J,J\rdsb=0$.
\end{definition}

Let $\scrL\to\scrM$ be a graded line bundle.
The following proposition identifies, in a canonical way, Jacobi bi-derivations $J$ on $\scrL[1]$ with Jacobi structures $\{-,-\}$ on $\scrL$, thus clarifying the content of Remark \ref{rem:shift}.
Within this identification, adopted throughout the paper, the Jacobi brackets corresponding to $\J$ are also denoted by $\{-,-\}_{\J}$.
\begin{proposition}
	\label{prop:Jacobi_bi-do}
	There is a canonical 1--1 correspondence between graded Jacobi brackets $\{-,-\}$ on $\scrL$ and graded Jacobi bi-derivations $J$ on $\scrL[1]$ given by the following relation
	\begin{equation*}
	\{s\lambda_1,s\lambda_2\}=(-)^{|\lambda_1|}s(J(\lambda_1,\lambda_2)),
	\end{equation*}
	for all homogeneous $\lambda_1,\lambda_2\in\Gamma(\scrL[1])$,
	where $s:\Gamma(\scrL[1])\to\Gamma(\scrL)$ denotes the suspension map.
\end{proposition}

\begin{proof}
	It follows from the same argument used in the un-graded case, see~\cite[Lemma 2.11(2)]{LOTV} and~\cite[Prop.~2.7]{tortorella2017thesis}.
\end{proof}

\begin{remark}
	\label{rem:d_J}
	We will often denote by $\{-,-\}_J$ the (graded) Jacobi bracket corresponding to a (graded) Jacobi bi-derivation $J$.
	Sometimes we will simply identify $J$ and $\{-,-\}_J$ and write $J \equiv \{-,-\}_J$ (or $J \equiv \{-,-\}$).
	
	Let $(\scrM,\scrL,\J)$ be a graded Jacobi manifold.
	There is a differential graded Lie algebra attached to $\scrM$, namely $(\Diff^\star(\scrL[1]),d_\J,\ldsb-,-\rdsb)$, with $d_\J:=\ldsb\J,-\rdsb$.
	The cohomology of $(\Diff^\star(\scrL[1]),d_\J)$ is called the Chevalley--Eilenberg cohomology of $(\scrM,\scrL,\J)$, and it is denoted by $H_{CE}(\scrM,\scrL,\J)$.
\end{remark}

We now discuss automorphisms of a graded Jacobi manifold $(\scrM,\scrL,\J\equiv \{-,-\})$.
\begin{definition}
	\label{def:Jacobi_automorphism}
	A \emph{Jacobi automorphism} of $(\scrM,\scrL,\J)$ is a degree $0$ graded automorphism $\Phi$ of the graded line bundle $\scrL\to\scrM$ such that $\Phi^\ast\J=\J$, i.e.
	\begin{equation*}
	\Phi^\ast\{\lambda_1,\lambda_2\}_{\J}\equiv\{\Phi^\ast\lambda_1,\Phi^\ast\lambda_2\}_{\J},\ \textnormal{for all}\ \lambda_1,\lambda_2\in\Gamma(\scrL).
	\end{equation*}
\end{definition}
The group of Jacobi automorphisms of $(\scrM,\scrL,\J)$ is denoted by $\Aut(\scrM,\scrL,\J)$.
The Lie subalgebra $\mathfrak{aut}(\scrM,\scrL,\J)\subset\Diff(\scrL)^0$ of infinitesimal Jacobi automorphisms, or \emph{Jacobi derivations}, of $(\scrM,\scrL,\J)$ consists of those degree $0$ graded derivations $\square\in\Diff(\scrL)^0$ such that $\ldsb\J,\square\rdsb=0$, i.e.
\begin{equation*}
\square\{\lambda_1,\lambda_2\}_{\J}\equiv\{\square\lambda_1,\lambda_2\}_{\J}+\{\lambda_1,\square\lambda_2\}_{\J},\ \textnormal{for all}\ \lambda_1,\lambda_2\in\Gamma(\scrL).
\end{equation*}

\begin{definition}
	\label{def:Hamiltonian_family}
	For a smooth path $\{\lambda_t\}_{t\in I}\subset\Gamma(\scrL)^0$ and a smooth path $\{\Phi_t\}_{t\in I}$ of degree 0 automorphisms of $\scrL \to \scrM$, we say that $\{\lambda_t,-\}_{\J}$ \emph{integrates} to $\{\Phi_t\}_{t\in I}$ if
	\begin{equation*}
	\Phi_0=\id_{\scrL},\quad \frac{d}{dt}\Phi_t^\ast=\{\lambda_t,-\}_{\J}\circ\Phi_t^\ast.
	\end{equation*}
	In this case $\{\Phi_t\}_{t\in I}$ consists of Jacobi automorphisms and it is called the \emph{smooth path of Hamiltonian automorphisms} associated with the smooth path of \emph{Hamiltonian sections} $\{\lambda_t\}_{t\in I}$.
\end{definition}

\begin{definition}
	\label{def:Hamiltonian_single}
	A \emph{Hamiltonian automorphism} of $(\scrM,\scrL,\J)$ is a degree $0$ graded automorphism $\Phi\in\Aut(\scrM,\scrL,\J)$ such that $\Phi=\Phi_1$, for some smooth path of Hamiltonian automorphisms $\{\Phi_t\}_{t\in I}$.
\end{definition}
The group of Hamiltonian automorphisms of $(\scrM,\scrL,\J)$ is denoted by $\Ham(\scrM,\scrL,\J)$.
The Lie subalgebra $\mathfrak{ham}(\scrM,\scrL,\J)\subset\mathfrak{aut}(\scrM,\scrL,\J)$ of infinitesimal Hamiltonian automorphisms, or \emph{Hamiltonian derivations}, of $(\scrM,\scrL,\J)$ consists of those degree $0$ graded derivations $\square\in\Diff(\scrL)^0$ of the form $\ldsb\J,\lambda\rdsb=\{\lambda,-\}_{\J}$, for some $\lambda\in\Gamma(\scrL)^0$.

\subsection{Viewing \texorpdfstring{$\Diff^\star(\scrL)$}{D*(scrL)} as module of sections of a graded vector bundle over \texorpdfstring{$\scrM$}{scrM}}
\label{subsec:psi_nabla}

Let $\scrM$ be a graded manifold with support $M$, and let $\scrL\to\scrM$ be a graded line bundle.
From Batchelor-Gawedzki Theorem, there is a, non-necessarily canonical, graded module isomorphism 
\begin{equation}
\label{eq:Batchelor}
\Gamma(\scrL)\simeq S_{{\scriptscriptstyle C^\infty(M)}}\Gamma(F^\ast)\otimes_{{\scriptscriptstyle C^\infty(M)}}\Gamma(P), 
\end{equation}
covering a graded algebra isomorphism $C^\infty(\scrM)\simeq S_{{\scriptscriptstyle C^\infty(M)}}\Gamma(F^\ast)$, where $F\to M$ is a graded vector bundle, and $P\to M$ is a graded line bundle.
Additionally, a $\Diff(P)$-connection in $F\to M$ determines a graded module isomorphism
\begin{equation}\label{eq:iso_Batch}
\Diff^\star(\scrL)\simeq S_{{\scriptscriptstyle C^\infty(M)}}\Gamma(F^\ast\oplus F_P\oplus(J^1P)^\ast )\otimes_{{\scriptscriptstyle C^\infty(M)}}\Gamma(P),
\end{equation}
covering a graded algebra isomorphism $\Diff^\star(\scrL,\bbR_{\scrM})\simeq S_{{\scriptscriptstyle C^\infty(M)}}\Gamma(F^\ast\oplus F_P\oplus(J^1P)^\ast)$, where $F_P:=F\otimes P^\ast$.
The goal of this subsection is to describe isomorphism (\ref{eq:iso_Batch}) explicitly.

\begin{definition}
	\label{def:vertical_graded_derivation}
	A vector field $X\in\frakX(\scrM)$ is said to be \emph{vertical} if it is in the kernel of the canonical fibration $\scrM\to M$, i.e.~$X(f)=0$, for all $f\in C^\infty(M)\subset C^\infty(\scrM)$.
\end{definition}

Denote by $\vder(\scrM)$ the set of vertical vector fields on $\scrM$.
It is both a graded $C^\infty(\scrM)$-submodule and a graded Lie subalgebra of $\frakX(\scrM)$.

\begin{remark}
	\label{rem:A-derivations}	\
	
	1) By restricting vertical vector fields to $\Gamma(F^\ast)\subset C^\infty(\scrM)$, we get a degree $0$ graded $C^\infty(M)$-module isomorphism
	\begin{equation}
	\label{eq:A-derivations1}
	\vder(\scrM)\overset{\simeq}{\longrightarrow} C^\infty(\scrM)\underset{{\scriptscriptstyle C^\infty(M)}}{\otimes}\Gamma(F).
	\end{equation}
	
	2) There exists a short exact sequence of degree $0$ graded $C^\infty(\scrM)$-module morphisms
	\begin{equation}
	\label{eq:seq1}
	0 \longrightarrow\VDer(\scrM) \longrightarrow\Diff(\scrL)\longrightarrow C^\infty(\scrM)\underset{{\scriptscriptstyle C^\infty(M)}}{\otimes}\Diff(P) \longrightarrow 0.
	\end{equation}
	The arrow $\VDer(\scrM)\longrightarrow\Diff(\scrL)$, written $Z\longmapsto\bbD_Z$, is defined by setting $\bbD_Z(a\otimes\lambda)=Z(a)\otimes\lambda$, for all $Z\in\VDer(\scrM)$, $a\in C^\infty(\scrM)$, and $\lambda\in\Gamma(P)$.
	In its turn, the arrow $\Diff(\scrL)\longrightarrow C^\infty(\scrM)\otimes_{{\scriptscriptstyle C^\infty(M)}}\Diff(P)$ is obtained by restricting derivations to $\Gamma(P)$.
\end{remark}

Now, a $\Diff(P)$-connection $\nabla$ in $F\to M$ determines a $C^\infty(M)$-linear map $\bar{\nabla}:\Diff(P)\to\Diff(\scrL)$, $\square\mapsto\bar{\nabla}_{\square}$ via\begin{equation*}
\bar{\nabla}_\square(\alpha\otimes p)=(\nabla^\ast_\square\alpha)\otimes p+\alpha\otimes(\square p),
\end{equation*}
for all $p\in\Gamma(P)$, and $\alpha\in\Gamma(F^\ast)$, where $\nabla^\ast$ is the $D(P)$-connection in $F^\ast$ dual to $\nabla$.

\begin{proposition}
	\label{prop:psi_nabla}
	The $C^\infty(M)$-linear map $\bar{\nabla}:\Diff(P)\to\Diff(\scrL)$ extends, by $C^\infty(\scrM)$-linearity, to a splitting of the short exact sequence~\eqref{eq:seq1}, so it determines a degree 0 graded $C^\infty(\scrM)$-module isomorphism
	\begin{equation}
	\label{eq:prop:psi_nabla}
	\phi_\nabla : \Diff(\scrL)\overset{\simeq}{\longrightarrow}\left[C^\infty(\scrM)\underset{{\scriptscriptstyle C^\infty(M)}}{\otimes}\Gamma\left(F_P\oplus (J^1P)^\ast\right)\right]\underset{{\scriptscriptstyle C^\infty(\scrM)}}{\otimes}\Gamma(\scrL).
	\end{equation}
	Moreover, there is a unique degree $0$ graded $C^\infty(\scrM)$-module isomorphism 
	\begin{equation*}
	\psi_\nabla : \Diff^\star(\scrL)\overset{\simeq}{\longrightarrow} S_{{\scriptscriptstyle C^\infty(M)}}(\Gamma(F^\ast\oplus F_P))\otimes_{{\scriptscriptstyle C^\infty(M)}}\Diff^\star(P),
	\end{equation*}
	covering a graded $C^\infty(\scrM)$-algebra isomorphism $\underline{\smash{\psi_\nabla}}:\Diff^\star(\scrL,\bbR_\scrM)\to S_{{\scriptscriptstyle C^\infty(M)}}\Gamma(F^\ast\oplus F_P\oplus (J^1P)^\ast)$,
	such that: 1) $\psi_\nabla$ agrees with the identity map on $\Gamma(\scrL)$, and 2) $\psi_\nabla$ agrees with $\phi_\nabla$ on $\Diff(\scrL)$.
\end{proposition}

The preceding proposition, whose proof is straightforward, plays a key r\^ole in the construction of the set of contraction data~\eqref{eq:1contraction_data_1}, and in the proof of Proposition~\ref{prop:main_result} as well.

\section{A step-by-step obstruction method}
\label{app:SBSO}

Our construction of the BFV-complex of a coisotropic submanifold is entirely based on the step-by-step obstruction method of homological perturbation theory.
Indeed the central results in the BFV-construction, namely Theorems~\ref{theor:existence_BFV_brackets} and~\ref{theor:uniqueness_BFV_brackets} and Theorems~\ref{theor:existence_BRST-charge} and~\ref{theor:uniqueness_BRST-charge}, can be proved through a direct application of this method.
The aim of this Appendix is to provide a self-contained version of the step-by-step obstruction method, which is well-suited for the objectives of the present paper.
In doing this, we will adapt and integrate~\cite[Section 4.1]{Stasheff1997}.

The setting of the step-by-step obstruction method is the following:
a set of contraction data (cf.~Definition~\ref{def:contraction_data})
\begin{equation}
\label{eq:homotopy_equivalence}
\begin{tikzpicture}[>= stealth,baseline=(current bounding box.center)]
\node (u) at (0,0) {$(X,d)$};
\node (d) at (5,0) {$(Y,0)$};
\draw [transform canvas={yshift=-0.5ex},->] (d) to node [below] {\footnotesize $I$} (u);
\draw [transform canvas={yshift=+0.5ex},<-] (d) to node [above] {\footnotesize $P$} (u);
\draw [->] (u.south west) .. controls +(210:1) and +(150:1) .. node[left=2pt] {\footnotesize $H$} (u.north west);
\end{tikzpicture},
\end{equation}
and a decreasing filtration of $X$ by graded subspaces $\{\calF_n\}_{n\in\bbZ}$ such that
\begin{equation}
\label{eq:SBSO_1}
\begin{gathered}
d\calF_n\subset\calF_{n-1},\quad H\calF_n\subset\calF_{n+1},\ \textnormal{for all}\ n\in\bbZ,\\
P\calF_{N+1} = 0,\ \textnormal{for some}\ N\geq-1.
\end{gathered}
\end{equation}
Additionally, it is important to assume that decreasing filtration $\{\calF_n\}_{n\in\bbZ}$ eventually terminates, in each homogeneous component separately, i.e., for each $i\in\bbZ$, we have $X^i\cap\calF_n=0$ for all $n \gg 0$.
Finally suppose there is a (degree $0$) graded Lie bracket $[-,-]$ on $X$, and a degree $1$ element $\bar{Q}\in X^1\cap\calF_{-1}$, such that
\begin{equation}
\label{eq:SBSO_2}
\begin{gathered}{}
[\calF_m,\calF_n] \subset\calF_{m+n},\ \textnormal{for all}\ m,n\in\bbZ,\\
[\bar{Q},\Omega]\equiv d\Omega\Mod{\calF_n},\ \textnormal{for all}\ n\in\bbZ,\ \textnormal{and}\ \Omega\in\calF_n.
\end{gathered}
\end{equation}

The question is whether or not $\bar{Q}$ can be deformed to a Maurer-Cartan (MC) element $Q$ of $(X,[-,-])$ such that $Q\equiv\bar{Q}\Mod\calF_{N+1}$.
The step-by-step obstruction method provides an answer.
Firstly, Proposition~\ref{prop:SBSO_existence} points out a necessary and sufficient condition for the existence of such MC elements, and explicitly constructs one of them, in the affirmative case.
Secondly, Proposition~\ref{prop:SBSO_uniqueness} establishes the uniqueness, up to isomorphisms, of such MC elements.

\begin{proposition}[Existence]
	\label{prop:SBSO_existence}
	There exists a MC element $Q$ of $(X,[-,-])$, such that $Q\equiv\bar{Q}\Mod\calF_{N+1}$, iff the following condition holds:
	\begin{equation}
	\label{eq:prop:SBSO_existence}
	[\bar{Q},\bar{Q}]\in\calF_{N}\cap\ker P.
	\end{equation}
\end{proposition}

\begin{proof} \ 
	
	$(\Longrightarrow)$
	Let $Q$ be a MC element of $(X,[-,-])$ such that $Q\equiv\bar{Q}\Mod\calF_{N+1}$.
	Then we get immediately
	\begin{equation*}
	0=[Q,Q]=[\bar{Q},\bar{Q}]+2[\bar{Q},Q-\bar{Q}]+[Q-\bar{Q},Q-\bar{Q}]\equiv[\bar{Q},\bar{Q}]+2d(Q-\bar{Q})\Mod\calF_{N+1}.
	\end{equation*}
	Hence, from~\eqref{eq:SBSO_1} and $P\circ d=0$, it follows that $P([\bar{Q},\bar{Q}])=0$, and $[\bar{Q},\bar{Q}]\in\calF_N$.
	
	$(\Longleftarrow)$
	Assume that $[\bar{Q},\bar{Q}]\in\calF_N\cap\ker P$.
	The main idea of the proof is to construct the desired MC element $Q$ through a perturbative expansion
	\begin{equation}
	\label{eq:perturbative_expansion}
	Q=\bar{Q}+\sum_{n>N} Q_n,
	\end{equation}
	where $\{Q_n\}_{n>N}$ is a (necessarily finite) sequence in $X^1$, with $Q_n\in\calF_n$, such that, for every $k\geq N$,
	\begin{equation*}
	Q(k):=\bar{Q}+\sum_{N<h\leq k}Q_h\Longrightarrow [Q(k),Q(k)]\in\calF_k.
	\end{equation*}
	Clearly if such sequence $\{Q_n\}_{n>N}$ exists, then~\eqref{eq:perturbative_expansion} is actually a finite sum and  provides a MC element $Q$ of $(X,[-,-])$ such that $Q\equiv\bar{Q}\Mod\calF_{N+1}$.
	Now, we show how to set up a recursive procedure to construct the $Q_n$.
	
	{\sc Proof of the Main Idea.}
	Assume we constructed the sequence $\{Q_n\}_{n>N}$ up to the term of filtration degree $k$, for some $k\geq N$.
	Then the next term in the sequence can be obtained by setting $Q_{k+1}:=\frac 12H[Q(k),Q(k)]$.
	Indeed, from either the hypothesis on $\bar{Q}$ (if $k=N$), or the inductive hypothesis (if $k>N$), it follows, in any case, that $[Q(k),Q(k)]\in\ker P$.
	Moreover, from Jacobi identity, we get
	\begin{equation*}
	0=[Q(k),[Q(k),Q(k)]]\equiv[\bar{Q},[Q(k),Q(k)]]\Mod\calF_k\equiv d[Q(k),Q(k)]\Mod\calF_{k}.
	\end{equation*}
	Hence $[Q(k),Q(k)]$ is annihilated by $P$, and is $d$-closed up to terms of filtration degree $k$, so that
	\begin{align*}
	[Q(k)+Q_{k+1},Q(k)+Q_{k+1}]&=[Q(k),Q(k)]+2[Q(k),Q_{k+1}]+[Q_{k+1},Q_{k+1}]\\
	&\equiv(\id+d\circ H)[Q(k),Q(k)]\Mod\calF_{k+1}\\
	&\equiv(I\circ P-H\circ d)[Q(k),Q(k)]\Mod\calF_{k+1}\\
	&\equiv 0\Mod\calF_{k+1}.\qedhere
	\end{align*}
\end{proof}

\begin{proposition}[Uniqueness]
	\label{prop:SBSO_uniqueness}
	Let $Q^i$ be a MC element of $(X,[-,-])$, such that $Q^i\equiv\bar{Q}\Mod\calF_{N+1}$, for $i=0,1$.
	Then there exists an automorphism $\phi$ of $(X,[-,-])$ such that $\phi(Q^0)=Q^1$.
	Moreover such automorphism $\phi$ can be chosen so that $\phi(\Omega)\equiv\Omega\Mod\calF_{n+N+2}$, for all $n\in\bbZ$, and $\Omega\in\calF_n$.
\end{proposition}

\begin{proof}
	The main idea of the proof is to check that, for every $n\geq N+1$, if $Q^0$ and $Q^1$ coincide up to terms of filtration degree $n$, i.e.
	\begin{equation}
	\label{eq:proof:prop:SBSO_uniqueness1}
	Q^1\equiv Q^0\Mod\calF_n,
	\end{equation}
	then there is $R\in\calF_{n+1}\subset\calF_1$ such that $Q^1$ and $(\exp R)Q^0:=\sum_{k=0}^\infty\frac{1}{k!}\operatorname{ad}_{R}^k Q^0$, with $\operatorname{ad}_R:=[R,-]$, coincide up to terms of filtration degree $n+1$, i.e.
	\begin{equation}
	\label{eq:proof:prop:SBSO_uniqueness2}
	(\exp R)(Q^0)\equiv Q^1\Mod\calF_{n+1}.
	\end{equation}
	The statement of the Proposition will then follow because the decreasing filtration eventually terminates in each homogeneous component separately, and $Q^0\equiv Q^1\Mod\calF_{N+1}$, by hypothesis.
	
	{\sc Proof of the Main Idea.}
	If~\eqref{eq:proof:prop:SBSO_uniqueness1} holds, then~\eqref{eq:proof:prop:SBSO_uniqueness2} is satisfied by setting $R:=H(Q^1-Q^0)$.
	Indeed, from the Maurer-Cartan equation for $Q^0$ and $Q^1$, it follows that 
	\begin{equation*}
	0=[Q^1,Q^1]=2[Q^0,Q^1-Q^0]+[Q^1-Q^0,Q^1-Q^0]\equiv 2d(Q^1-Q^0)\Mod\calF_n,	\end{equation*}
	i.e.~$Q^1-Q^0$ is $d$-closed up to terms of filtration degree $n$.
	Hence we also get
	\begin{equation*}
	Q^1-Q^0=(I\circ P-d\circ H-H\circ d)(Q^1-Q^0)\equiv[R,Q^0]\Mod\calF_{n+1},
	\end{equation*}
	and so $(\exp R)Q^0\equiv Q^1\Mod\calF_{n+1}$, as needed.
\end{proof}

By the same argument used in the proof of Proposition~\ref{prop:SBSO_uniqueness}, we get the following.

\begin{corollary}
	\label{cor:SBSO_uniqueness}	
	Let $Q$ be a MC element of $(X,[-,-])$, with $Q\equiv\bar{Q}\Mod(\calF_{N}\cap\ker P)$.
	If $N\geq 0$ and~\eqref{eq:prop:SBSO_existence} holds, then there is an automorphism $\phi$ of $(X,[-,-])$ such that $\phi(Q)\equiv\bar{Q}\Mod\calF_{N+1}$.
	Additionally such automorphism $\phi$ can be chosen so that $\phi(\Omega)\equiv\Omega\Mod\calF_{n+N+1}$, for all $n\in\bbZ$, and $\Omega\in\calF_n$.
\end{corollary}

\begin{remark}
	Notice that the \emph{side conditions} satisfied by contraction data~\eqref{eq:homotopy_equivalence}, do not play any role in the proofs of Propositions~\ref{prop:SBSO_existence}, \ref{prop:SBSO_uniqueness} and Corollary~\ref{cor:SBSO_uniqueness}, and they could be actually relaxed.
	Actually, there is only one place in the paper where the side conditions are truly relevant. Namely, we need the homotopy equivalence~\eqref{eq:2contraction_data_1} to be a set of true contraction data when proving Theorem~\ref{theor:L_infty_qi}. Indeed, the side conditions are necessary to implement the \emph{homotopy transfer} that generates the higher brackets in the $L_\infty$-algebra from the BFV-complex.
\end{remark}

\section{Some auxiliary technical results}
\label{app:technical_tools}

The aim of this Appendix is to state and prove Propositions~\ref{prop:main_result} and~\ref{prop:main_result_bis}, which represent the technical tools at the basis of most of the main results in the paper, namely Theorems~\ref{theor:gauge_invariance_BFV_complex},~\ref{theor:Mod_Ham}, and~\ref{theor:Mod_Jac}.
Actually those theorems are just straightforward applications of Propositions~\ref{prop:main_result} and~\ref{prop:main_result_bis}.

In this section we will work within the local model established in Section~\ref{sec:BRST-charges}.
Let us start with two preliminary lemmas.
\begin{lemma}
	\label{lem:technical_lemma}
	Let $s$ be a section of the vector bundle $\pi:NS\to S$, and let $\{e_t\}_{t\in I}$ be a smooth path of sections of the pull-back vector bundle $E:=\pi^\ast(NS)\to NS$.
	Suppose that
	\begin{enumerate}
		\item\label{enum:lem:technical_lemma1}
		$e_0=\Omega_E[s]$,
		\item\label{enum:lem:technical_lemma2}
		$\textnormal{``zero locus of $e_t$''}=\im(s)$,
		\item\label{enum:lem:technical_lemma3}
		$\left.e_t\right|_{E_x}$ intersects transversally the restriction to $E_x$ of the zero section of $E\to NS$, for all $x\in S$.
	\end{enumerate}
	Then there exists a smooth path $\{A_t\}_{t\in I}\subset\Gamma(\operatorname{GL}(E))$, with $A_0=\id_E$, such that
	\begin{equation}
	\label{eq:lem:technical_lemma}
	e_t=A_te_0,
	\end{equation}
	or equivalently there is a smooth path $\{a_t\}_{t\in I}\subset\End(E)$ such that
	\begin{equation}
	\label{eq:lem:technical_lemma_bis}
	\frac{d}{dt}e_t=a_te_t.
	\end{equation}
\end{lemma}

\begin{proof}
	The proposition has a local character, and it is enough to work in a neighborhood of an arbitrary point $x\in NS$.
	We distinguish two cases: $x\notin\im(s)$ and $x\in\im(s)$.
	
	\textsc{First Case.}
	Set $N':=NS\smallsetminus\im(s)$, and equip the vector bundle $E':=E|_{N'}\to N'$ with a Riemannian metric.
	Denote by $F_t\to N'$ the rank-$1$ vector subbundle of $E'\to N'$ generated by $e_t|_{N'}$.
	A smooth path $\{a_t\}_{t\in I}\subset\End(E)$ satisfying~\eqref{eq:lem:technical_lemma_bis} can be obtained by the following composition of vector bundle morphisms: $E'\overset{P_t}{\longrightarrow}F_t\overset{I_t}{\longrightarrow}E'$, where $P_t:E'\to F_t$ is the orthogonal projection, and $I_t:F_t\to E'$ is given by $I_t(e_t|_{N'})=\left(\frac{d}{dt}e_t\right)|_{N'}$.
	
	\textsc{Second Case.}
	Since $e_t\in\Gamma(\hat{L})^{(1,0)}$, and $d[s]\in\Diff(\hat{L})^{(0,1)}$, it follows, by bi-degree reasons, that $d[s](e_t)=0$.
	Even more, $e_t$ is actually a co-boundary of $(\Gamma(\hat{L}),d[s])$.
	Indeed hypothesis~\eqref{enum:lem:technical_lemma2} guarantees that $\wp[s]e_t=0$.
	Hence, in view of Proposition~\ref{prop:2contraction_data}, setting $A_t:=-h[s]e_t$, we get
	\begin{equation*}
	e_t=d[s]A_t=\{\Omega_E[s],A_t\}_G=A_t\Omega_E[s]=A_te_0,
	\end{equation*}
	where, in the last step, we used hypothesis~\eqref{enum:lem:technical_lemma1}.
	Finally, a simple computation in local coordinates shows that hypothesis~\eqref{enum:lem:technical_lemma3} is equivalent to the fiberwise invertibility of $A_t:=-h[s]e_t$ on $\im(s)$, and so also on some open neighborhood of $\im(s)$ in $NS$.
\end{proof}

\begin{lemma}
	\label{lem:integrating_graded_do}
	Fix a smooth path $\{\square_t\}_{t\in I}\subset\Diff(\hat{L})^0$, with $I:=[0,1]$.
	Denote by $\{\tilde\square_t\}_{t\in I}\subset\Diff(\hat{L})^{(0,0)}$, $\{\slashed\square_t\}_{t\in I}\subset\Diff L$ and $\{X_t\}_{t\in I}\subset\frakX(M)$ the smooth paths defined by setting $\tilde\square_t:=\operatorname{pr}^{(0,0)}\square_t$, $\slashed\square_t:=\left.\square_t\right|_{\Gamma(L)}$, and $X_t:=\sigma_{\slashed\square_t}$ respectively.
	The following conditions are equivalent:
	\begin{enumerate}
		\item\label{enum:integrating_graded_do1}
		$\{\square_t\}_{t\in I}$ integrates to a smooth path $\{\Phi_t\}_{t\in I}$ of automorphisms of $\hat{L} \to \hat{M}$;
		\item\label{enum:integrating_graded_do2}
		$\{\tilde\square_t\}_{t\in I}$ integrates to a smooth path $\{\tilde\Phi_t\}_{t\in I}$ of bi-degree $(0,0)$ automorphisms of $\hat{L}\to\hat{M}$;  
		\item\label{enum:integrating_graded_do3}
		$\{\slashed\square_t\}_{t\in I}$ integrates to a smooth path $\{(F_t,\underline{\smash{F_t}})\}_{t\in I}$ of automorphisms of $L \to M$;
		\item\label{enum:integrating_graded_do4}
		$\{X_t\}_{t\in I}$ integrates to a smooth path $\{\underline{\smash{F_t}}\}_{t\in I}$ of diffeomorphisms of $M$.
	\end{enumerate}
	Moreover, if the equivalent conditions (1)-(4) are satisfied, then the following relations hold:
	\begin{gather*}
	\Phi_t^\ast\lambda=\tilde\Phi_t^\ast\lambda\Mod\bigoplus_{k\geq 1}\Gamma(\hat{L})^{(p+k,q+k)},\ \textnormal{for all}\ (p,q)\in\bbN^2,\ \lambda\in\Gamma(\hat{L})^{(p,q)},\\
	\tilde\Phi^\ast\lambda=F_t^\ast\lambda,\ \textnormal{for all}\ \lambda\in\Gamma(L).	
	\end{gather*}
\end{lemma}

\begin{proof}
	It is straightforward.
\end{proof}

Let us fix the setting for Proposition~\ref{prop:main_result}.
Assume we have the following smooth paths
\begin{itemize}
	\item $\{\J_t\}_{t\in I}\subset\Diff^2(L[1])$ such that $\J_t$ is a Jacobi structure on $L\to M$,
	\item $\{s_t\}_{t\in I}\subset\Gamma(\pi)$, such that $\im(s_t)$ is a coisotropic submanifold of $(M,L,\J_t)$,
	\item A smooth path $\{(F_t,\underline F_t)\}_{t\in I}$ of automorphisms of $L \to M$, with $\underline{\smash{F}}_0=\id_M$ and $F_0=\id_L$, such that $\im s_t=\underline F_t(\im s_0)$ and $\J_t=F_t^\ast\J_0$.
\end{itemize}
Fix moreover the following objects:
\begin{itemize}
	\item $\hat{\J}_0$, a lifting of $\J_0$ to a Jacobi structure on $\hat{L}\to\hat{M}$,
	\item $\Omega_0$, an $s_0$-BRST charge wrt $\hat{\J}_0$.
\end{itemize}
Our aim is to find a lifting of $\{F_t\}_{t\in I}$ to a suitable smooth path $\{\calF_t\}_{t\in I}$ of bi-degree $(0,0)$ automorphisms of $\hat{L}\to\hat{M}$, with $\calF_0=\id_{\hat{L}}$.
This is accomplished by the following.

\begin{proposition}
	\label{prop:main_result}
	There exists a smooth path $\{\calF_t\}_{t\in I}$ of bi-degree $(0,0)$ automorphisms of $\hat{L}\to\hat{M}$, with $\calF_0=\id_{\hat{L}}$, such that
	\begin{enumerate}
		\item\label{enum:main_result1}
		$\{\calF_t\}_{t\in I}$ is a lifting of $\{F_t\}_{t\in I}$, i.e.~$\calF_t|_L=F_t$,
		\item\label{enum:main_result2}
		$\hat\J_t:=(\calF_t)_\ast\hat\J_0$ is a lifting of $\J_t$ to $\hat{L}\to\hat{M}$,
		\item\label{enum:main_result3}
		$\Omega_t:=(\calF_t)_\ast\Omega_0$ is an $s_t$-BRST charge wrt $\hat{\J}_t$.
	\end{enumerate}	
\end{proposition}

\begin{proof}
	We will show explicitly how to construct a smooth path $\{\square_t\}_{t\in I}\subset\Diff(\hat{L})^{(0,0)}$ integrating to a smooth path $\{\calF_t\}_{t\in\ I}$ of bi-degree $(0,0)$ automorphisms of $\hat L\to\hat M$, so that
	\begin{equation*}
	\calF_0=\id_{\hat L},\quad\frac{d}{dt}\calF_t^\ast=\square_t\circ\calF_t^\ast,
	\end{equation*}
	and moreover the latter satisfies conditions~\eqref{enum:main_result1}--\eqref{enum:main_result3} in the statement.
	
	Fix an arbitrary $\Diff\ell$-connection $\nabla$ in $\pi:NS\to S$.
	By pull-back along $\pi$, we also get a $\Diff L$-connection in $E\to NS$, again denoted by $\nabla$.
	Arguing as in Sections~\ref{subsec:first_relevant_contraction_data} and~\ref{subsec:psi_nabla}, the latter $\nabla$ determines a degree $0$ graded $C^\infty(\hat{M})$-module isomorphism, of bi-degree $(0,0)$,
	\begin{equation*}
	\Diff\hat{L}\simeq C^\infty(\hat{M})\underset{{\scriptscriptstyle C^\infty(M)}}{\otimes}\big(\underbrace{\Gamma((E_L)^\ast)}_{(-1,0)}[1]\oplus\underbrace{\Gamma(E)}_{(0,-1)}[-1]\oplus\underbrace{\Diff L}_{(0,0)}\big)
	\end{equation*}
	Focussing on the ghost/anti-ghost bi-degree $(0,0)$ component, we get, in particular, an isomorphism of $C^\infty(M)$-modules
	\begin{equation*}
	\Diff(\hat{L})^{(0,0)}\simeq\End(E)\oplus\End(E^\ast)\oplus\Diff L.
	\end{equation*}
	
	Consequently, for any path $\{\square_t\}_{t\in I}\subset\Diff(\hat{L})^{(0,0)}$ there exist $\{a_t\}_{t\in I}\subset\End(E)$, $\{b_t\}_{t\in I}\subset\End(E^\ast)$, and $\{\slashed\square_t\}_{t\in I}\subset\Diff L$ uniquely determined by
	\begin{equation*}
	\square_t=\bbD_{a_t}+\bbD_{b_t}+\nabla_{\slashed\square_t}.
	\end{equation*}
	Here we interpret the endomorphisms $a_t,b_t$ as vertical vector fields on $\hat{M}$ via (\ref{eq:A-derivations1}) (with $F = (E_L)^\ast [1] \oplus E[-1]$) and use the arrow $\bbD:\VDer(\hat{M})\to D(\hat{L})$ of Remark \ref{rem:A-derivations} (with $\scrM = \hat{M}$ and $\scrL = \hat{L}$).
	For every $Z\in\VDer(\hat{M})$, the symbol of $\bbD_Z$ is $Z$ and vanishes on $C^\infty(M)$.
	Hence Lemma~\ref{lem:integrating_graded_do} guarantees that the condition on $\{\square_t\}_{t\in I}\subset\Diff(\hat{L})^{(0,0)}$ to integrate to a smooth path $\{\calF_t\}_{t\in I}$ of bi-degree $(0,0)$ automorphisms of $\hat{L}\to\hat{M}$, is equivalent to both
	\begin{itemize}
		\item $\{\nabla_{\slashed\square_t}\}_{t\in I}\subset\Diff(\hat{L})^{(0,0)}$, and integrates to a smooth path $\{\Phi_t\}_{t\in I}$ of bi-degree $(0,0)$ automorphisms of $\hat{L}\to\hat{M}$,
		\item $\{\slashed\square_t\}_{t\in I}\subset\Diff L$ integrates to a smooth path $\{\phi_t\}_{t\in I}$ of automorphisms of $L \to M$.
	\end{itemize}
	Now, it follows that $\{\Phi_t^{-1}(\bbD_{a_t}+\bbD_{b_t})\}_{t\in I}=\{\bbD_{{\tilde a}{}_t}+\bbD_{\tilde b_t}\}_{t\in I}$, for some $\{\tilde a_t\}_{t\in I}\subset\End(E)$, and $\{\tilde b_t\}_{t\in I}\subset\End(E^\ast)$.
	Therefore Lemma~\ref{lem:integrating_graded_do} again implies that $\{\Phi_t^{-1}(\bbD_{a_t}+\bbD_{b_t})\}_{t\in I}$ integrates to $\Psi_t:= S_{{\scriptscriptstyle C^\infty(M)}}((A_t\otimes\id_{L^\ast})\oplus B_t)\otimes\id_L$, for some smooth paths $\{A_t\}_{t\in I}\subset\Gamma(\operatorname{GL}(E))$ and $\{B_t\}_{t\in I}\subset\Gamma(\operatorname{GL}(E^\ast))$, with $A_0=\id_E$ and $B_0=\id_{E^\ast}$.
	Hence we get the following decomposition
	\begin{equation*}
	\calF_t=\Phi_t\circ\Psi_t.
	\end{equation*}
	
	The construction of $\{\Phi_t\}_{t\in I},\{\Psi_t\}_{t\in I}$ guarantees that
	\begin{equation*}
	\Phi_t|_L=\phi_t,\qquad \Phi_t G=G,\qquad
	\Psi_t|_L=\id_L,\qquad \Psi_t|_{\hat{L}^{(1,0)}}=A_t,\qquad \Psi_t|_{\hat{L}^{(0,1)}}=B_t\otimes\id_L.
	\end{equation*}
	Now condition~\eqref{enum:main_result1} can be equivalently rewritten as
	\begin{equation}
	\label{eq:proof:main_resultI}
	\phi_t=F_t,
	\end{equation}
	which completely determines $\Phi_t$, and so also $\slashed\square_t$.
	Condition~\eqref{enum:main_result2} splits into two conditions: $\phi_t\J_0=\J_t\equiv F_t\J_0$, and $\Psi_tG=G$.
	Hence it reduces to $B_t^\ast A_t=\id_E$, for all $t\in I$, and can be equivalently rewritten as
	\begin{equation*}
	\tilde b_t^\ast+\tilde a_t=0.
	\end{equation*}
	Hence $\bbD_{\tilde a_t}+\bbD_{\tilde b_t}=\{\tilde a_t,-\}_G$, and so $\{\Psi_t\}_{t\in I}$ is obtained by integration of
	\begin{equation}
	\label{eq:proof:main_resultII}
	\Psi_0=\id_{\hat{L}},\quad \frac{d}{dt}\Psi_t=\{\tilde a_t,-\}_G\circ\Psi_t.
	\end{equation}
	The latter does not yet tell us anything about $\tilde a_t$.
	However, condition~\eqref{enum:main_result3} becomes $A_t\Omega_E[s_0]=\Phi_t^\ast(\Omega_E[s_t])$, or equivalently
	\begin{equation}
	\label{eq:proof:main_resultIII}
	\frac{d}{dt}(e_t)=\tilde a_t(e_t),
	\end{equation}
	where $e_t:=\Phi_t^\ast(\Omega_E[s_t])$, for all $t\in I$.
	From the choice of $\nabla$ and the construction of $\{\Phi_t\}_{t\in I}$, the smooth path $\{e_t\}_{t\in I}\subset\Gamma(\hat{L})^{(1,0)}=\Gamma(E)$ satisfies the hypothesis of Lemma~\ref{lem:technical_lemma}, and so equation~\eqref{eq:proof:main_resultIII} admits a (non-unique) solution $\{\tilde a_t\}_{t\in I}\subset\Gamma(\hat{L})^{(1,0)}$.
	This concludes the proof.
\end{proof}

Now, we fix the setting for Proposition~\ref{prop:main_result_bis}.
Fix a Jacobi structure $\J$ on $L\to M$.
Suppose we have the following:
\begin{itemize}
	\item a smooth path $\{s_t\}_{t\in I}\subset C(L,\J)$ of coisotropic sections of $(L,\J)$, and
	\item a smooth path $\{(F_t,\underline F_t)\}_{t\in I}$ of automorphisms of $(M,L,\J)$, with $F_0=\id_L$, integrating $\{\lambda_t,-\}_{\J}$, for some smooth path $\{\lambda_t\}_{t\in I}\subset\Gamma(L)$, and $\im s_t=\underline F_t(\im s_0)$.
\end{itemize}
Choose a $\Diff\ell$-connection $\nabla$ in $\pi:M\equiv NS\to S$, and pull it back along $\pi$ to get a $\Diff L$-connection in $E=\pi^\ast(NS)\to NS$ which will be again denoted by $\nabla$.
Fix moreover $\Omega_0$, an $s_0$-BRST charge wrt $\hat{\J}^\nabla$.

Our aim is to find a lifting of $\{F_t\}_{t\in I}\subset\Aut(M,L,\J)$ to a suitable smooth path $\{\calF_t\}_{t\in I}\subset\Aut(\hat{M},\hat{L},\hat{\J}^\nabla)$, with $\calF_0=\id_{\hat{L}}$.
This aim is accomplished by the following.

\begin{proposition}
	\label{prop:main_result_bis}
	There exist a smooth path $\{\hat{\lambda}_t\}_{t\in I}\subset\Gamma(\hat{L})^0$, and a  smooth path $\{\calF_t\}_{t\in I}\subset\Aut(\hat{M},\hat{L},\hat{\J}^\nabla)$, which integrates $\{\hat{\lambda}_t,-\}_{\hat{\J}^\nabla}$, such that
	\begin{enumerate}
		\item\label{enum:main_result_bis1}
		$\{\calF_t\}_{t\in I}\subset\Aut(\hat{M},\hat{L},\hat{\J}^\nabla)$ is a lifting of $\{F_t\}_{t\in I}\subset\Aut(M,L,\J)$, i.e.~$\operatorname{pr}^{(0,0)}\circ\calF_t|_L=F_t$,
		\item\label{enum:main_result_bis2}
		$\Omega_t:=(\calF_t)_\ast\Omega_0$ is an $s_t$-BRST charge wrt $\hat{\J}^\nabla$.
	\end{enumerate}	
\end{proposition}

\begin{proof}
	Setting $\slashed\square_t=\{\lambda_t,-\}_{\J}$, from the local coordinate expression of $i_\nabla$ in Remark~\ref{rem:local_expression_i_nabla}, it is straightforward to get
	\begin{equation}
	\label{eq:main_result_bis1}
	\nabla_{\slashed\square_t}=\{\lambda_t,-\}_{i_\nabla\J}.
	\end{equation}
	The scheme used in the proof of Proposition~\ref{prop:main_result} applies as well in the current special situation, where $\hat{\J}^0=\hat{\J}^\nabla$, and $\J_t=\J$, and guarantees the existence of a smooth path $\{\tilde a_t\}_{t\in I}\subset\Gamma(\hat{L})^{(1,1)}=\Gamma(\End(E))$ and a smooth path $\{\tilde\calF_t\}_{t\in I}$ of bi-degree $(0,0)$ automorphisms of $\hat{L} \to \hat{M}$, with
	\begin{equation}
	\label{eq:main_result_bis2}
	\tilde\calF_0=\id_{\hat{L}},\quad\frac{d}{dt}\tilde\calF_t^\ast=(\{\tilde a_t,-\}_G+\nabla_{\slashed\square_t})\circ\tilde\calF_t^\ast,
	\end{equation}
	such that, in particular, $\tilde\calF_t$ is a lifting of $F_t$, and $\operatorname{pr}^{(1,0)}(\tilde\calF_t)_\ast\Omega_0=\Omega_E[s_t]$.
	Define a smooth path $\{\hat{\lambda}_t\}_{t\in I}\subset\Gamma(\hat{L})^0$ by setting $\hat{\lambda}_t:=\lambda_t+\tilde a_t$.
	Because of Lemma~\ref{lem:integrating_graded_do}, from~\eqref{eq:main_result_bis1} and~\eqref{eq:main_result_bis2} it follows that $\{\hat{\lambda}_t,-\}_{\hat{\J}^\nabla}$ integrates to a smooth path $\{\calF_t\}_{t\in I}$, with $\calF_0=\id_{\hat{L}}$, satifying conditions~\eqref{enum:main_result_bis1} and ~\eqref{enum:main_result_bis2}.
\end{proof}

\section*{Acknowledgements}
HVL thanks LV for inviting her to Salerno in 2014 where a part of this project has been discussed.
HVL is partially supported by RVO:67985840.
AGT is grateful to Janusz~Grabowski for having mentored him during the WCMCS Ph.D.~internship at IMPAN, where a part of this work was done.
The stay of AGT at IMPAN has been partially supported by the Warsaw Center of Mathematics and Computer Science and by PRIN 2010-11 ``Variet\`a reali e complesse: geometria, topologia e analisi armonica'' of University of Florence.
AGT is partially supported by GNSAGA of INdAM, Italy. LV is member of the GNSAGA of INdAM, Italy

\end{document}